\documentclass[a4paper,11pt]{article}
\usepackage[sort,numbers]{natbib}
\usepackage[T1]{fontenc}
\usepackage[utf8]{inputenc}
\usepackage[english]{babel}
\usepackage{mathtools}
\mathtoolsset{showonlyrefs}
\usepackage{makeidx}
\makeindex
\usepackage{amsmath,amssymb,amsthm}
\usepackage{algorithm}
\usepackage{booktabs}
\usepackage[noend]{algpseudocode}
\usepackage{xargs,accents,listings}
\usepackage{enumitem,multirow}
\usepackage{url}
\usepackage{graphicx,subcaption}
\usepackage{pgfplots}
\usepackage{tikz,tikz-3dplot}
\usepackage{textcomp}
\usetikzlibrary{spy}
\usetikzlibrary{calc}
\usetikzlibrary{positioning}
\usetikzlibrary{arrows.meta}
\pgfplotsset{compat=1.11}
\pgfplotsset{
	colormap={parula}{
		rgb255=(53,42,135)
		rgb255=(15,92,221)
		rgb255=(18,125,216)
		rgb255=(7,156,207)
		rgb255=(21,177,180)
		rgb255=(89,189,140)
		rgb255=(165,190,107)
		rgb255=(225,185,82)
		rgb255=(252,206,46)
		rgb255=(249,251,14)}
}
\usepackage{hyperref}

\newcommand{\N}{\ensuremath{\mathbb{N}}}

\newcommand{\R}{\ensuremath{\mathbb{R}}}

\renewcommand{\atop}[2]{\genfrac{}{}{0pt}{}{#1}{#2}}
\newcommand{\zb}[1]{\ensuremath{\boldsymbol{#1}}}

\newcommand{\tT}{{\scriptscriptstyle \operatorname{T}}}
\DeclareMathOperator*{\argmin}{arg\,min}

\DeclareMathOperator*{\diag}{diag}

\DeclareMathOperator{\diver}{div}

\newcommand{\IC}{\operatorname{IC}}
\newcommand{\TGV}{\operatorname{TGV}}
\newcommand{\HH}{\operatorname{H}}
\newcommand{\TV}{\operatorname{TV}}
\newcommand{\G}{\mathcal{G}}

\renewcommand{\sl}{\mathfrak{sl}}

\newcommand{\jj}{\mathbf{j}}

\newcommandx{\ubar}[1]{\underaccent{\bar}{#1}}

\newtheorem{theorem}{Theorem}
\newtheorem{remark}[theorem]{Remark}
\newtheorem{corollary}[theorem]{Corollary}
\newtheorem{proposition}[theorem]{Proposition}
\newtheorem{lemma}[theorem]{Lemma}
\newtheorem{example}[theorem]{Example}

\begin{document}	
	
	\title{Strain Analysis by a	Total Generalized Variation Regularized Optical Flow Model}
	\date{\today}
	\author{ 
		Frank Balle\footnotemark[1], 
		Tilmann Beck\footnotemark[1], 
		Dietmar Eifler\footnotemark[1], 
		Jan Henrik Fitschen\footnotemark[2], \\ 
		Sebastian Schuff\footnote{Department of Mechanical and Process Engineering, University of Kaiserslautern, Germany,		\{balle,beck,eifler,schuff\}@mv.uni-kl.de.}, \ and 
		Gabriele Steidl\footnote{Department of Mathematics,
			University of Kaiserslautern, Germany,
			\{fitschen, steidl\}@mathematik.uni-kl.de.
		}
	}
	\maketitle
	
	\begin{abstract}
		In this paper we deal with the important problem
		of estimating the local strain tensor from a sequence of micro-structural images 
		realized during deformation tests of engineering materials.
		Since the strain tensor is defined via the Jacobian of the displacement field,
		we propose to compute the displacement field by a variational model 
		which takes care of properties of the Jacobian of the displacement field. 
		In particular we are interested in areas of high strain.		
		The data term of our variational model relies on the brightness invariance
		property of the image sequence.
		As prior we choose the second order total generalized variation of the displacement field.
		This prior splits the Jacobian of the displacement field 
		into a smooth and a non-smooth part. 
		The latter reflects the material cracks. 
		An additional constraint is incorporated to handle physical properties of the non-smooth part 
		for tensile tests.
		We prove that the resulting convex model has a 
		minimizer
		and show how a primal-dual method can be applied to find 
		a minimizer.
		The corresponding algorithm has the advantage that the strain tensor is directly computed 
		within the iteration process.
		Our algorithm is further equipped with a coarse-to-fine strategy to cope with larger displacements. 		
		Numerical examples with simulated and experimental data demonstrate the 
		very good performance of our algorithm.
		In comparison to state-of-the-art engineering software for strain analysis 
		our method can resolve local phenomena much better.
	\end{abstract}
	
	\section{Introduction} \label{sec:strain:intro}
	The (Cauchy) strain tensor plays a fundamental role in mechanical engineering for 
	deriving local mechanical properties of materials. 
	In this paper, we are interested in estimating the strain tensor 
	from a sequence of microstructural images of a certain material 
	acquired during in-situ deformation tests. An experimental setup of a tensile test is
	shown in Figure~\ref{fig:sem_image}.
	
	%
	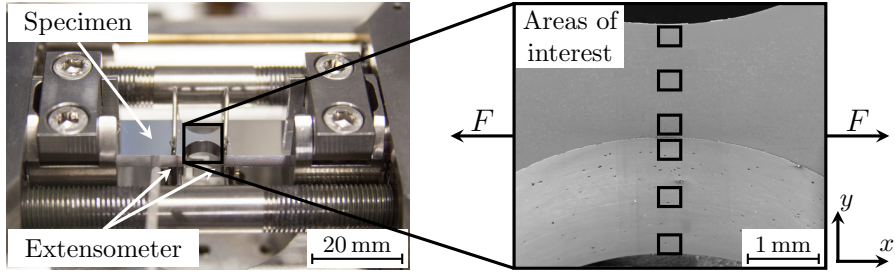
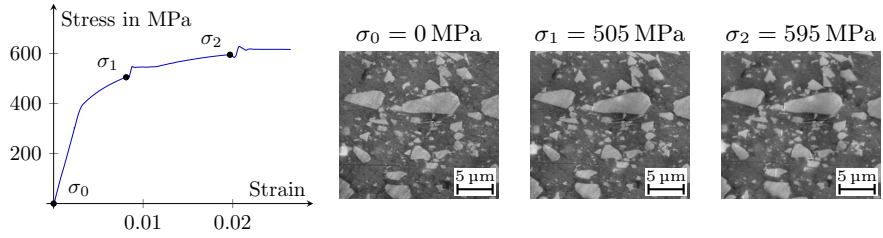
\begin{figure}
		\centering
		\begin{subfigure}[t]{0.98\textwidth}\centering
			\includegraphics[width=.95\textwidth]{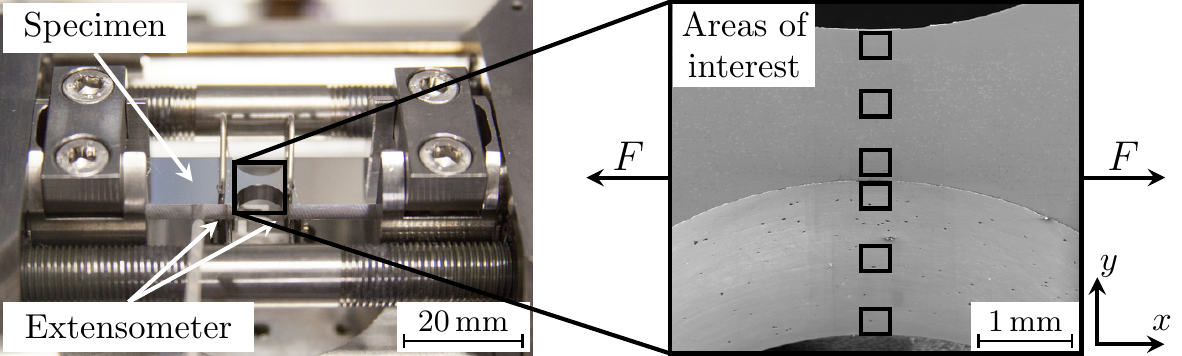}
			\caption{Experimental setup for the tensile test inside a scanning electron microscope.
			Left: Test material within the microscope. Right: Part of the material with the regions of interest,
			where a sequence of images (micrographs) is taken for increasing load.}
		\end{subfigure}\\[2ex]
		\begin{subfigure}[t]{0.98\textwidth}\centering	
			\raisebox{-.5\height}{
				\includegraphics[width=.33\textwidth]{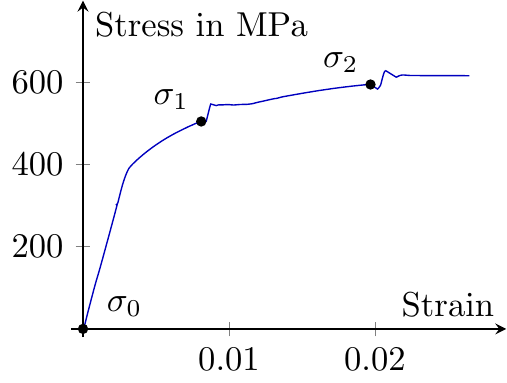}}
			\begin{tabular}{ccc}
				{\footnotesize $\sigma_0 = 0$\,MPa}
				&
				{\footnotesize $\sigma_1 = 505$\,MPa}
				&
				{\footnotesize $\sigma_2 = 595$\,MPa}			
				\\		
				\begin{tikzpicture}\tikzstyle{every node}=[font=\tiny] 
				\begin{axis}[ 
				width=.33\textwidth, 
				enlargelimits=false, 
				hide axis, 
				axis equal image, 
				axis on top, 
				] 
				\addplot
				graphics[xmin=0, xmax=683, ymin=0, ymax=636] 
				{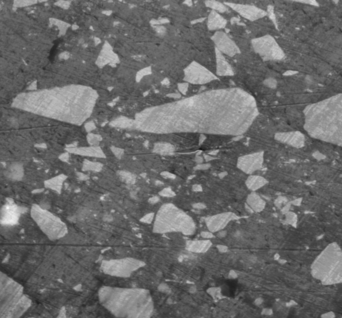};
				\draw[draw=none,fill=white] (493,10) rectangle (673,130);
				\draw[{|[width=4pt]}-{|[width=4pt]},thick] (503,40) -- (663,40) node[above=-2pt,pos=0.5] {5\,\textmu m}; 
				\end{axis} 
				\end{tikzpicture} 
				&
				\begin{tikzpicture}\tikzstyle{every node}=[font=\tiny] 
				\begin{axis}[ 
				width=.33\textwidth, 
				enlargelimits=false, 
				hide axis, 
				axis equal image, 
				axis on top, 
				] 
				\addplot
				graphics[xmin=0, xmax=683, ymin=0, ymax=636] 
				{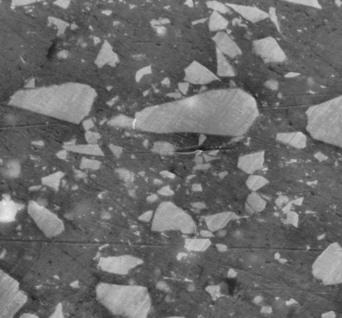};
				\draw[draw=none,fill=white] (493,10) rectangle (673,130);
				\draw[{|[width=4pt]}-{|[width=4pt]},thick] (503,40) -- (663,40) node[above=-2pt,pos=0.5] {5\,\textmu m}; 
				\end{axis} 
				\end{tikzpicture}
				&
				\begin{tikzpicture}\tikzstyle{every node}=[font=\tiny] 
				\begin{axis}[ 
				width=.33\textwidth, 
				enlargelimits=false, 
				hide axis, 
				axis equal image, 
				axis on top, 
				] 
				\addplot
				graphics[xmin=0, xmax=683, ymin=0, ymax=636] 
				{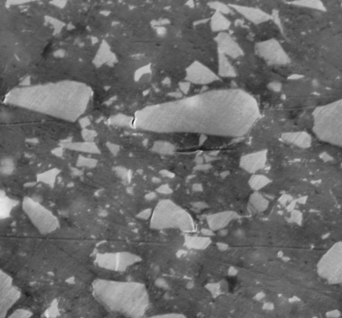};
				\draw[draw=none,fill=white] (493,10) rectangle (673,130);
				\draw[{|[width=4pt]}-{|[width=4pt]},thick] (503,40) -- (663,40) node[above=-2pt,pos=0.5] {5\,\textmu m}; 
				\end{axis} 
				\end{tikzpicture}
			\end{tabular}
			\caption{Stress-strain curve with three selected micrographs taken under increasing load. The stress $\sigma=\frac{F}{A_0}$ is defined as the applied force $F$ divided by the initial cross section of the specimen $A_0$. The (global) strain is defined as the relative elongation of the specimen measured by the extensometer.}\label{fig:sem_image:b}
		\end{subfigure}
		\caption{ \label{fig:sem_image}
			Illustration of an in-situ tensile test.
		} 
	\end{figure}
	%
	
	Let $u \coloneqq (u_1,u_2)^\tT$ be the displacement field that describes how each point moves from one image to another.       
	For deformations of a continuum body, the strain tensor $\varepsilon$ 
	is defined via the Jacobian of the displacement field $u$ by
	\begin{align} \label{eq:strain}
	\varepsilon &= 
	\begin{pmatrix}	\varepsilon_{11} & \varepsilon_{12} \\ 
	\varepsilon_{12} & \varepsilon_{22} 
	\end{pmatrix}
	\coloneqq 
	\frac12 \left( \nabla u + \nabla u^\tT \right)\\
	&=
	\begin{pmatrix}	\partial_x u_1 & \frac12 (\partial_y u_1 +  \partial_x u_2) \\ 
	\frac12 (\partial_y u_1 +  \partial_x u_2) & \partial_y u_2 
	\end{pmatrix}.
	\end{align}
	
	State-of-the-art software packages for strain analysis like 
	Veddac~\cite{Veddac}, VIC~\cite{VIC}, NCorr~\cite{BAA14}, and \cite{SWP99,TK03,WS00}
	apply correlation based methods to estimate the displacement field and use the result 
	to compute the strain tensor. Roughly speaking, correlation based methods compare certain windows 
	around each pixel in a predefined search window.
	Due to the extent of the window around the pixels and since the displacement is often only computed 
	on a coarser grid to reduce the computational effort, the local resolution of these methods is limited.
	This is especially a drawback when we are interested in the local strain behavior and the appearance of microstructural damage.
	
	In this paper, we propose to use a variational model to compute the displacement field $u$ from a sequence of images $f$. 
	Variational models are composed of a data term $E_{\text{Data}}(u;f)$, which incorporates the information
	given by the image sequence, and a prior $E_{\text{Reg}}(u)$, which has to be chosen in such a way that its minimizer reflects 
	known or desirable properties of the displacement field:
	\begin{align} \label{whole_energy}
	\argmin_u E(u), \quad E(u) \coloneqq E_{\text{Data}}(u;f) + \lambda E_{\text{Reg}}(u).
	\end{align}
	
	In image processing, the apparent displacement $u$ between image frames is known as optical flow.
	Variational methods for optical flow estimation go back to Horn and Schunck \cite{HS81}, and there
	is a vast number of refinements and extensions of their approach. We refer to \cite{BPS14} 
	for a comprehensive overview. 
	In particular, optical flow models with priors containing higher order derivatives of the flow
	were  successfully used, e.g.~in  \cite{ACGKMSS08,HK14,CMP02,TPCB08,YSM07,YSS07,RBP14,VRS13,BDB13}.	
	
	Although it seems natural to apply ideas from variational optical flow models also for
	strain analysis, such methods have rarely been addressed in the literature.
	The papers \cite{ACGKMSS08,CMP02} aim at
	computing derivatives simultaneously to the optical flow field 
	but are not related to engineering applications.
	The computation of the (Lagrangian) strain tensor by a variational method
	was addressed in \cite{HWSSD13}. 
	There, the authors proposed a {\it smooth} fourth order optical flow model 
	which directly computes the strain tensor from an image sequence obtained in 
	a biaxial tensile test with an elastomer. 
	In contrast to our work, they were interested 
	in the macro-scale behavior and compute the minimizer of their smooth energy function 
	by solving the corresponding Euler-Lagrange equations. 
	
	We propose a data term $E_{\text{Data}}(u;f)$ based on the brightness invariance
	assumption of the image sequence, which is usual in optical flow estimation.
	The more interesting part is the choice of the prior, where we have in mind that
	we are interested in local variations of the Jacobian of the displacement field, i.e.~of the strain. 
	In the conference paper \cite{own:strain}, we built up on the assumption that the displacement field $u$
	is additively composed of a non-smooth part $v$ and a smooth part $w$.
	Then, we penalized the Jacobian of $v$ and the second order derivative of $w$
	by an infimal convolution prior. The infimal convolution (IC) of first and second order derivatives was introduced 
	in imaging by Chambolle and Lions \cite{CL97}. For a recent generalization called ICTV, we refer to \cite{HK14}.
	In this paper, we will see that it is better for our task to split
	the Jacobian of $u$ instead of $u$ itself.
	To this end, we apply the second order total generalized variation (TGV) of the  displacement field as prior.
	For tensile tests, a physical prior on the non-smooth part of the strain is added.
	The TGV of functions was introduced by Bredies, Kunisch, and Pock~\cite{BKP10} for image restoration tasks
	and has meanwhile found many applications. 
	For a discrete variant, see also \cite{SS08,SST11}.
	In particular, TGV regularization was successfully used in connection with optical flow in \cite{RBP14,VRS13,BDB13}.
	We mention that there is a software package corresponding to \cite{VRS13} available since 2015.
	Our first code in \cite{own:strain} was written before any public software package with TGV or infimal convolution prior was available.
	Our applications show the potential of the method in an impressive way.
	
	The resulting variational model is convex and a minimizer can be computed with
	primal-dual methods, which are meanwhile standard in image processing.
	The corresponding algorithm computes the strain tensor directly 
	within the iteration process.
	This is an advantage compared to standard engineering software, which estimates the displacement field by
	correlation based methods first and then computes its Jacobian.
	Our algorithm is further equipped with a coarse-to-fine strategy to cope with larger displacements 
	as proposed, e.g., in \cite{Ana89,BBPW04,SRB10,SRB13}.
	
	We demonstrate by artificial and real-world examples that the
	proposed model leads to very good results, which cannot be obtain by state-of-the-art software packages.
	With our extensive numerical results, we aim at convincing engineers of the advantages of the proposed method.

	\section{Model} \label{sec:strain:model}
	We start with the continuous model and turn to the discrete model
	for digital images afterwards. 
	The reason is that the discrete notation appears more
	clumsy and it is easier to get the clue from the continuous notation.
	
	\paragraph{Continuous model.}
	Suppose that we are given gray-valued images 
	\begin{align}
	f_1, f_2 \colon \mathbb R^2 \supset \Omega\rightarrow \R.
	\end{align}
	We are interested in the flow field $u = (u_1,u_2) \colon \Omega \rightarrow \R^2$,
	which describes how $f_1$ transforms into $f_2$.
	Here, we focus on the brightness invariance assumption, which reads with $\zb x \coloneqq (x,y)$ as
	\begin{equation} \label{bas_flow}
	f_1(\zb x) - f_2\big(\zb x + u(\zb x) \big) = f_1(\zb x) - f_2\big(x + u_1(\zb x),y+u_2(\zb x) \big) \approx 0.
	\end{equation}
	A first order Taylor expansion around an initial optical flow field $\bar u = (\bar u_1,\bar u_2)$
	gives
	\begin{align} \label{taylor_bas_flow}
	f_2\big( \zb x + u(\zb x)) 
	&\approx 
	f_2 \big(\zb x + \bar u(\zb x)\big) 
	+ \left\langle
	\begin{pmatrix} \partial_x f_2 \\ \partial_y f_2 \end{pmatrix}
	\big( \zb x + \bar u (\zb x)\big), u(\zb x) - \bar u(\zb x) \right\rangle.
	\end{align}
	Plugging \eqref{taylor_bas_flow} into \eqref{bas_flow}, we obtain
	\begin{align}	
	0  &\approx f_1(\zb x) - f_2\big( \zb x + \bar u (\zb x)\big) - \left\langle
	\begin{pmatrix} \partial_x f_2 \\ \partial_y f_2 \end{pmatrix} \big( \zb x + \bar u (\zb x)\big), u(\zb x) - \bar u(\zb x) \right\rangle.
	\end{align}	
	Later, we apply a coarse-to-fine scheme \cite{Ana89,BBPW04} 
	and use the result from one scale as an initialization for the next one.
	Applying a non-negative increasing function
	$\varphi \colon \mathbb R \rightarrow \mathbb R_{\ge 0}$
	results in the  data term
	\begin{align}\label{continuous_data_term}
	\int_{\Omega} \varphi 
	\left(-
	\left\langle
	\begin{pmatrix} 
	\partial_x f_2 \big( \zb x + \bar u (\zb x) \big) \\ 
	\partial_y f_2 \big( \zb x + \bar u (\zb x) \big)  
	\end{pmatrix},
	u(\zb x) - \bar u(\zb x) \right\rangle 
	- f_2 \big( \zb x + \bar u (\zb x) \big) + f_1(\zb x) \right)
	\, d \zb x. \hspace{0.6cm}
	\end{align}	
	For the task at hand, we need a prior which recognizes local changes
	in the displacement field. 
	Therefore, it seems to be useful to consider the smooth global and non-smooth local behavior 
	of the displacement field. 
	For $z \in \mathbb R^{d_1,d_2}$ we denote by $|z|$ the square root of the sum of the squared components of $z$,
	i.e., its Frobenius norm.

	Let us first consider the following Banach spaces of scalar functions $u\colon \Omega \rightarrow \mathbb R$:
	\begin{itemize}
	\item[-] functions of bounded variation $BV(\Omega)$ with norm
$	\|u\|_{\mathrm{BV}} \coloneqq \|u\|_{L_1} + \mathrm{TV}(u)$,
where
\begin{align}
	\mathrm{TV}(u) \coloneqq \sup_{\varphi \in C_c^\infty(\Omega,\R^2), |\varphi|  \le 1}  \int_{\Omega} u \diver \varphi \ dx,
\end{align}
where $C_c^\infty(\Omega,\R^2)$ denotes the space of compactly supported, smooth functions on $\Omega$ mapping to $\R^2$.
The distributional first order derivative $Du$ is a vector-valued Radon measure with total variation $|Du|(\Omega) = \mathrm{TV}(u)$.
In particular for $u \in W^{1,1}(\Omega)$ the Sobolev space of functions 
with absolutely integrable weak first order partial derivatives, the regularizer becomes
$$	\mathrm{TV}(u) = \int_{\Omega} \lvert \nabla u \rvert \ dx,$$
\item[-]
functions of bounded Hessian $BH(\Omega)$ with norm
$
\|u\|_{\mathrm{BH}} \coloneqq  \|u\|_{\mathrm{BV}} + \mathrm{TV}_2(u)
$,
where
\begin{align}
\mathrm{TV}_2(u) \coloneqq \sup_{ \varphi \in C_c^2(\Omega,\R^{2,2}) , |\varphi| \le 1 } \int_{\Omega} u  \diver^2 \varphi \ dx
\end{align}
and
$
  \diver^2 \varphi
  \coloneqq
  \partial_{xx} \varphi_{11} + \partial_{xy} \varphi_{12} +
  \partial_{yx} \varphi_{21} + \partial_{yy} \varphi_{22}.
$
If $u \in W^{2,1}(\Omega)$ the term becomes
$$	
\mathrm{TV}_2(u) = \int_{\Omega} |\nabla^2 u | \ dx, \quad   \nabla^2 u = \begin{pmatrix} u_{xx}& u_{xy}\\ u_{yx} &u_{yy} \end{pmatrix},
$$
\item[-] of functions of total generalized variation of order two
$\mathrm{BGV}_\lambda^2(\Omega)$, $\lambda = (\lambda_1,\lambda_1) \in\mathbb R^2_{>0}$~\cite{BKP10}
with norm
$
			\|u\|_{\mathrm{BGV}_\lambda^2} \coloneqq \|u\|_{L_1} + \mathrm{TGV}_\lambda^2(u),
$
where
\begin{align}\label{intro:TGV}
			\TGV_\lambda^2(u)
			\coloneqq
			\sup_{\atop{
			\varphi \in C_c^\infty(\Omega,\mathrm{Sym}(\R^2)), |\varphi| \le \lambda_2,}
			{
			\left(
			(\diver \varphi_{1})^2 +  (\diver \varphi_{2})^2\right)^{1/2}  \le \lambda_1
			}}
			\int_{\Omega} u \diver^2\varphi \ d x .
\end{align}
Here $\mathrm{ Sym}(\R^2)$ denotes the space of symmetric $2\times 2$ matrices, i.e., we have $\varphi_{12} = \varphi_{21}$.
In \cite{BV11}, it was shown that $\mathrm{TGV}_\lambda$ can alternatively be characterized as
			\begin{align}\label{intro:TGV_meas}
				\mathrm{TGV}(u) = \inf_{a \in \mathrm{BD}(\Omega,\mathbb R^2)} &
				\left\{ \lambda_1 |D u - a|(\Omega)
				 + \lambda_2 |\mathcal{E} a|(\Omega) \right\},
			\end{align}
			where $\mathrm{BD}(\Omega,\mathbb R^2)$
			is the space of functions with bounded deformation,
			and $\mathcal{E} a$ is the weak symmetrized derivative of $a$.
	Under sufficient smoothness assumptions, the term reduces to
\begin{align}\label{intro:TGV_smooth}
	\mathrm{TGV}(u) = \inf_{a} \left\{ \lambda_1 \int_{\Omega} |\nabla u - a| \, dx  
	+  \lambda_2 \int_{\Omega} |\tilde \nabla a| \ d x \right\},
\end{align}
where $|\tilde \nabla a| \coloneqq \left( a_{1,x}^2 + \frac12 (a_{1,y} + a_{2,x})^2 + a_{2,y}^2 \right)^\frac12$.
\end{itemize}

In our task we deal with vector fields $u\colon \Omega \rightarrow \mathbb R^2$ and define $\TV(u)$, $\TV_2(u)$ and 
$\TGV(u)$ componentwise.

	In the previous paper \cite{own:strain}, we assumed that 
	$u= v+ w$ 
	can be decomposed additively into a (componentwise) non-smooth part
	$v \in \mathrm{ BV}(\Omega, \mathbb R^2)$ and a smooth part $w \in \mathrm{ BH}(\Omega,\mathbb R^2)$
	and suggested under the above smoothness assumptions the prior
	\begin{align} \label{IC_cont}
	\IC(u) \coloneqq& \inf_{v+w=u} \left\{ \int_{\Omega} \lambda_1 |\nabla v| + \lambda_2 |\nabla^2 w| \, d\zb x \right\}\\
	=& 
	\inf_{w} \left\{ \int_{\Omega} \lambda_1 |\nabla u - \nabla w| + \lambda_2 |\nabla^2 w| \, d\zb x \right\},
	\end{align}
	which is the infimal convolution of $|\nabla \cdot|$ and $|\nabla^2 \cdot |$.
	
	In this paper, we decompose $\nabla u = a + \tilde a$ instead of $u$
	and use the TGV of the displacement field $u$ as a prior,
	which reads under the corresponding smoothness assumptions as
	\begin{align} \label{TGV_cont}
	\TGV(u) \coloneqq& \inf_{a+\tilde a = \nabla u} \left\{ \int_{\Omega} \lambda_1 | \tilde a| + \lambda_2 |\tilde \nabla a| \, d\zb x \right\} \\
	=& \inf_{a} \left\{ \int_{\Omega} \lambda_1 |\nabla u - a| + \lambda_2 |\tilde \nabla a| \, d\zb x\right\}.
	\end{align}
	Of course, if $a$ is a conservative vector field, 
	i.e., the Jacobian of some function $w$, then $\IC$ and $\TGV$
	coincide. However, the latter is in general not the case. 
	In our applications, cracks are reflected by the non-smooth part $\tilde a$ of $\nabla u$.
	Here, it is important to notice the following difference:
	the IC prior \eqref{IC_cont} enforces sparsity of the Jacobian $\nabla v$ of the non-smooth part $v$ of $u$;
	while the TGV prior enforces sparsity of the non-smooth part $\tilde a$ of the Jacobian $\nabla u$ of $u$.
		
	\paragraph{Discrete model.}
	In practice, we are concerned with images
	$f_1, f_2 \colon \mathcal{G} \rightarrow \R$
	defined on a two-dimensional rectangular grid 
	$\mathcal{G} \coloneqq \{1,\ldots,N_1\} \times \{1,\ldots,N_2\}$.
	Besides this notion, we use the alternative representation of an image $f$ as a vector of length
	$N = |\mathcal{G}| = N_1 N_2$.
	In the task at hand, we are given micrographs $f_1$ and $f_2$, e.g.~from a tensile test, 
	corresponding to different loads as depicted in Figure~\ref{fig:sem_image}.
	Considering only grid points
	$\zb x = \jj \in \mathcal{G}$	
	and assuming for a moment that also the partial derivatives of $f_2$ are given,
	the data term in \eqref{continuous_data_term} becomes
	\begin{align}\label{partderiv_f2}
	\sum_{\jj \in \mathcal{G}} \varphi 
	\left(-
	\left\langle
	\begin{pmatrix} 
	\partial_x f_2 \big( \jj + \bar u (\jj) \big) \\ 
	\partial_y f_2 \big( \jj + \bar u (\jj) \big)  
	\end{pmatrix},
	u(\jj) - \bar u(\jj) \right\rangle 
	- f_2 \big( \jj + \bar u (\jj) \big) + f_1(\jj) \right).
	\end{align}	
	This term is only well defined if $\jj + \bar u(\jj) \in \mathcal{G}$.
	Here, we use bilinear interpolation to compute the required values of $f_2$ between grid points.
	If points lie outside the grid, we assume 
	mirror extension of the image at the boundary.
	By $\nabla_x f$ we denote
	the forward differences in $x$-direction with Neumann boundary conditions defined as
	\begin{align}
	(\nabla_x f)(\jj) \coloneqq \begin{cases}
	f(\jj+(1,0)) - f(\jj) & \mathrm{if} \  \jj+(1,0) \in \mathcal{G}, \\
	0 & \mathrm{otherwise},
	\end{cases}
	\qquad \jj \in \mathcal{G}.
	\end{align}
	Again, we use the notation $\nabla_x f$ both for the function and the vectorized version.
	Furthermore, we denote also the corresponding difference matrix by $\nabla_x$.
	The differences in $y$--direction can be defined analogously.
	We replace the partial derivatives of $f_2$ in \eqref{partderiv_f2} by the forward differences $\nabla_x f_2, \nabla_y f_2$ and apply 
	once more bilinear interpolation to compute values between grid points
	which results in the data term
	\begin{align}
	\sum_{\jj \in \mathcal{G}} \varphi 
	\left(-
	\left\langle
	\begin{pmatrix} 
	\nabla_x f_2(\jj+{\bar u(\jj)}) \\ 
	\nabla_y f_2(\jj+{\bar u(\jj)})  
	\end{pmatrix},
	u(\jj) - \bar u(\jj) \right\rangle 
	- f_2(\jj+{\bar u(\jj)}) + f_1(\jj) \right).
	\end{align}	
	For our task, we use the function
	\begin{align}
	\varphi(t) \coloneqq \lvert t \rvert
	\end{align} 
	because it is well-known that the $\ell_1$-norm reduces the influence of outliers.
	Then the data term can be rewritten in the convenient matrix-vector notation  
	\begin{align}\label{e_flow}
	E_{\text{Data}}(u;f_1,f_2) &\coloneqq  \| (A \ B) u + c\|_1 = \| A u_1 + B u_2 + c\|_1, 
	\end{align}
	where
	\begin{align} \label{a1}
	A \coloneqq &\diag \Big(\left(\nabla_x f_2(\jj+\bar u(\jj))\right)_{\jj \in \G} \Big), \quad
	B  \coloneqq \diag \Big(\left(\nabla_y f_2(\jj+\bar u(\jj))\right)_{\jj \in \G} \Big), \\
	c \coloneqq &- \diag \Big(\left(\nabla_x f_2(\jj+\bar u(\jj))\right)_{\jj \in \G} \Big) \bar u_1 \\
	&- \diag \Big(\left(\nabla_y f_2(\jj+\bar u(\jj))\right)_{\jj \in \G} \Big) \bar u_2\\
	&+ \left(f_2(\jj+\bar u(\jj))\right)_{\jj \in \G}  - f_1.
	\label{only_b}
	\end{align}
		
	For the discrete versions of the TGV prior, we also need
	the backward differences $\tilde \nabla_x u$ in $x$-direction, i.e., 
	\begin{align}
	(\tilde \nabla_x u)(\jj) = 
	\begin{cases}
	u(\jj) - u(\jj-(1,0)) & \mathrm{if} \  \jj \pm (1,0) \in \mathcal{G}, \\
	0 & \mathrm{otherwise},
	\end{cases}
	\qquad \jj \in \mathcal{G}
	\end{align}
	and analogously for the $y$-direction.
	We set
	\begin{equation}\label{not:1sttil}
	\nabla \coloneqq 
	\begin{pmatrix} 
	\nabla_x \\ \nabla_y 
	\end{pmatrix},
	\quad
	\tilde \nabla \coloneqq 
	\begin{pmatrix}
	\tilde \nabla_x & \\
	\frac12 \tilde \nabla_y & \frac12 \tilde \nabla_x \\
	& \tilde \nabla_y
	\end{pmatrix},
	\end{equation}
	and use the tensor products 
	\begin{align}
	\boldsymbol{\nabla} u \coloneqq (I_2 \otimes \nabla) u, \quad
	\boldsymbol{\tilde \nabla} u \coloneqq (I_2 \otimes \tilde \nabla) u.
	\end{align}
	Denoting by $\| \cdot\|_{2,1}$ the
	mixed norm defined for $x \in \R^{dN}$ as
	\begin{align}
	\| x \|_{2,1} \coloneqq \sum_{i=1}^N \|(x_{i+jN})_{j=0}^{d-1}\|_2, 
	\end{align} 
	the discrete counterpart of the TGV prior in~\eqref{TGV_cont} reads
	\begin{align}
	\TGV(u)  \coloneqq \min_{a} \left\{ \lambda_1( \| \boldsymbol{\nabla} u - a \|_{2,1} + \lambda_2 \| \boldsymbol{\tilde \nabla} a \|_{2,1} )\right\}.
	\tag{$\text{TGV}$} 
	\label{tgv}
	\end{align}
		
	In summary, our model is given by
	\begin{align} 	\label{TGV-model}	   
	E_{\TGV}(u) = \|(A \ B) u + c \|_1
	+ \min_{a} \left\{ \lambda_1 \| \boldsymbol{\nabla} u - a \|_{2,1} + \lambda_2 \| \boldsymbol{\tilde \nabla} a \|_{2,1} \right\}.
	\end{align}
		From a physical point of view, it makes sense to incorporate additional constraints on the displacement field or the strain.	
		In the TGV model, we split the strain into parts representing the global and local features.
		With this interpretation, it is a reasonable assumption that for tensile tests the strain corresponding 
		to local phenomena is positive in the direction of the applied force.
		This is motivated by the fact that cracks can only open or widen during a tensile test.
		Mathematically, if a force is applied in $x$-direction, we make the restriction 
		$
		\boldsymbol{\nabla}_x u_1 - a_1 \ge 0,
		$
		so that the model becomes
	\begin{align} 	\label{flow:constr:tgv}	   
	\tilde E_{\TGV}(u) = &\|(A \ B) u + c \|_1
	+ \min_{a} \left\{ \lambda_1 \| \boldsymbol{\nabla} u - a \|_{2,1} + \lambda_2 \| \boldsymbol{\tilde \nabla} a \|_{2,1} \right\}\\
	&\mbox{subject to} \quad \nabla_x u_1 - a_1 \ge 0.
	\end{align}	
	
	\begin{remark}
		Let us emphasize that strain in materials science is only defined for continuous deformation.
		In this sense, it is not defined for cracks.
		In accordance with the definition \eqref{eq:strain} in the continuous setting, we refer to its discrete counterpart $\nabla u$ as the strain.
		Nevertheless, we will see later that strain in the classical sense is sometimes better resembled by the smooth part~$a$. 
	\end{remark}
	
	By the next proposition, the functions~\eqref{TGV-model}, resp. \eqref{flow:constr:tgv}  possess minimizers under a mild assumption which was in particular fulfilled for all our practical examples.
	%
	\begin{proposition}	\label{prop:tgv}
		 Let $\ker ((A \ B)) \cap \ker (\boldsymbol{\tilde \nabla \nabla}) = \{ \mathbf{0} \}$.
		Then there exist minimizers of \eqref{TGV-model} and \eqref{flow:constr:tgv}. 		
	\end{proposition}
	%
	%
	\begin{proof} 
		It is sufficient to show that the continuous, convex function $E$ in \eqref{TGV-model} is coercive, 
		which implies the assertion for \eqref{TGV-model}.
		Regarding \eqref{flow:constr:tgv}, the assertion follows immediately since the function remains coercive after adding the constraint.
		
		Assume in contrast that $E$ is not coercive.
		Then, for some $C \in \R$, there exists a sequence $(u^{(r)})_{r \in \N}$ 
		with $\| u^{(r)} \|_2  \rightarrow \infty$ and $E(u^{(r)}) \le C$.
		We split 
		$u^{(r)} = u^{(r)}_a + u^{(r)}_b$ 
		with 
		$u^{(r)}_a \in \ker(\boldsymbol{\tilde \nabla \nabla})$ and 
		$u^{(r)}_b \in \ker(\boldsymbol{\tilde \nabla \nabla})^\perp$, 
		where $\ker(\boldsymbol{\tilde \nabla \nabla})^\perp$ denotes the orthogonal complement of $\ker(\boldsymbol{\tilde \nabla \nabla})$.
				
		Suppose that $\|u^{(r)}_b\|_2 \rightarrow \infty$ as $r \rightarrow  \infty$. 
		Then we have
		\begin{align}
			\TGV(u^{(r)})  
			&= \min_{a} \left\{ \lambda_1( \| \boldsymbol{\nabla} u^{(r)} - a \|_{2,1} + \lambda_2 \| \boldsymbol{\tilde \nabla} a \|_{2,1} )\right\} \\
			&= \min_{z} \left\{ \lambda_1( \| \boldsymbol{\nabla} u^{(r)}_b - z \|_{2,1} 
			+ \lambda_2 \| \boldsymbol{\tilde \nabla} z + \boldsymbol{\tilde \nabla} \boldsymbol{\nabla} u^{(r)}_a \|_{2,1} )\right\} \\
			&= \min_{z} \left\{ \lambda_1( \| \boldsymbol{\nabla} u^{(r)}_b - z \|_{2,1} + \lambda_2 \| \boldsymbol{\tilde \nabla} z \|_{2,1} )\right\},
		\end{align}
		where $z \coloneqq a - \boldsymbol{\nabla} u^{(r)}_a$.
		Let $\boldsymbol{\nabla} u^{(r)}_b = w^{(r)} + \tilde w^{(r)}$,
		where 
		$w^{(r)} \in \ker(\boldsymbol{\tilde \nabla})^\perp$ 
		and 
		$\tilde w^{(r)} \in \ker(\boldsymbol{\tilde \nabla})$.
		Then
		$$
		\boldsymbol{\tilde \nabla} \boldsymbol{\nabla} u^{(r)}_b = \boldsymbol{\tilde \nabla} w^{(r)}
		$$
		and since 
		$\boldsymbol{\tilde \nabla} \boldsymbol{\nabla}$ 
		is injective on 
		$\ker( \boldsymbol{\tilde \nabla} \boldsymbol{\nabla})^\perp$
		we have that 
		$\|\boldsymbol{\tilde \nabla} \boldsymbol{\nabla} u^{(r)}_b\|_2 =  \| \boldsymbol{\tilde \nabla} w^{(r)} \|_2 \rightarrow \infty$
		as $r \rightarrow \infty$.
		Now $\boldsymbol{\tilde \nabla}$ is injective on $\ker(\boldsymbol{\tilde \nabla})^\perp$ such that
		$\|w^{(r)} \|_2 \rightarrow \infty$ as $r \rightarrow \infty$.
		Let 
		$z^{(r)} = z^{(r)}_1 + z^{(r)}_2$ be a minimizer of $\TGV(u^{(r)})$,
		where
		$z_1^{(r)} \in \ker(\boldsymbol{\tilde \nabla})^\perp$ and $z_2^{(r)} \in \ker(\boldsymbol{\tilde \nabla})$.
		Then, by the equivalence of norms,
		$$
		\|w^{(r)} - z_1 + \tilde w^{(r)} - z_2\|_{2,1} 
		$$
		can only be bounded as $r \rightarrow \infty$ if 
		$$
		\|w^{(r)} - z_1^{(r)} + \tilde w^{(r)} - z_2^{(r)}\|_{2}^2 =  \|w^{(r)} - z_1^{(r)}\|_2^2 + \| \tilde w^{(r)} - z_2^{(r)}\|_2^2
		$$
		is bounded, which is only possible if $\|z_1^{(r)} \|_2 \rightarrow \infty$.
		But then we get
		$
		\|\boldsymbol{\tilde \nabla} z_1 ^{(r)} \|_{2,1} \rightarrow \infty	
		$
		as $r \rightarrow \infty$.
		Thus, we obtain the contradiction
		\begin{align}
		\TGV(u^{(r)})  \rightarrow \infty \qquad \mathrm{as} \ r \rightarrow \infty.
		\end{align}	
    			
		Assume that $\|u^{(r)}_b\|_2 \not \rightarrow \infty$ as $r \rightarrow  \infty$. 
		Then there exists a subsequence $(u^{(r_{j})}_b)_{j \in \N}$ with $\|u^{(r_{j})}_b\|_2 \le \tilde C$ for some $\tilde C \in \R$
		and it follows
		\begin{align}
		\| (A \ B) u^{(r_{j})} + c \|_1 &\ge \| (A \ B) u^{(r_{j})}_a + c \|_1 - \| (A \ B) u^{(r_{j})}_b  \|_1 \\ 
		&\ge \| (A \ B) u^{(r_{j})}_a + c \|_1 -  \tilde C. 
		\end{align}
		Since by assumption
		$(A \ B)$ is injective on $\ker(\boldsymbol{\tilde \nabla \nabla})$,
		the right-hand side goes to $+\infty$ as $j \rightarrow \infty$,
		which is a contradiction.		
	\end{proof}
	
	%
	\section{Algorithm} \label{sec:strain:alg}
	In this section, we use the primal-dual hybrid gradient method with modified dual variable (PDHGMp) \cite{CP11,PCCB09} 
	to compute a minimizer of \eqref{TGV-model}. The minimization of \eqref{flow:constr:tgv} follows similarly.
	
	We rewrite the minimization problem as
	\begin{align} \label{sec2}
	\min_{u,a,s,t}  & \left\{ \|(A \ B) u + c \|_1 + \lambda_1 \| s \|_{2,1} + \lambda_2 \| t \|_{2,1}  \right\}\\
	&\text{such that} \quad
	\left( \begin{array}{rr} \boldsymbol{\nabla} & -I_{4N} \\ 0 & \boldsymbol{\widetilde \nabla} \end{array} \right)
	\begin{pmatrix} u \\ a \end{pmatrix} 
	= \begin{pmatrix} s \\ t \end{pmatrix}.
	\end{align}	
	The basic PDHGMp algorithm for this problem is given in Algorithm \ref{strain:alg1tgv}.
	For the special form, we refer to~\cite[Alg.~8]{BSS2016}.
	
	{\small
		\begin{algorithm} 
			\begin{algorithmic}
				\State \textbf{Initialization:}
				{$u^{(0)}=u_0$, $a^{(0)}=a_0$, $b_1^{(0)}=0$, $b_2^{(0)}=\mathbf{0}$, $\bar b_1^{(0)}=\mathbf{0}$, $\bar b_2^{(0)}=\mathbf{0}$, $\theta =1$, $\tau_1 = \frac14$, $\tau_2=\frac14$.} \\
				\For{$r = 0,1,\ldots$}
				\State
				\begin{align}
				u^{(r+1)} = \ &\argmin_{u \in \R^{2N}} \Big\{\| (A \ B) u + c \|_1 + \frac{1}{2\tau_1} \|u-(u^{(r)}-
				\tau_1 \tau_2 \boldsymbol{\nabla}^\tT \bar b_1^{(r)})\|_2^2 \Big\} \\
				a^{(r+1)} = \ &a^{(r)} - 
				\tau_1 \tau_2 (\boldsymbol{\widetilde \nabla}^\tT  \bar b_2^{(r)} - \bar b_1^{(r)})
				\\
				s^{(r+1)} = \ &\argmin_{s \in \R^{4N}}\{\lambda_1 \|s\|_{2,1} + \frac{\tau_2}{2} \|s - (b_1^{(r)} + \boldsymbol{\nabla}   u^{(r+1)} - a^{(r+1)})\|_2^2\} \\
				t^{(r+1)} = \ &\argmin_{t \in \R^{6N}}\{\lambda_2 \|t\|_{2,1} + \frac{\tau_2}{2} \|t - (b_2^{(r)} + \boldsymbol{\widetilde \nabla} a^{(r+1)})\|_2^2\}
				\\
				b_1^{(r+1)} = \ & b_1^{(r)} + \boldsymbol{\nabla} u^{(r+1)} - a^{(r+1)} - s^{(r+1)} \\
				b_2^{(r+1)} = \ & b_2^{(r)} + \boldsymbol{\widetilde \nabla} a^{(r+1)}               - t^{(r+1)}			
				\\ 
				\bar b_1^{(r+1)} = \ & b_1^{(r+1)} + \theta (b_1^{(r+1)}-b_1^{(r)} ) \\
				\bar b_2^{(r+1)} = \ & b_2^{(r+1)} + \theta (b_2^{(r+1)}-b_2^{(r)} )
				\end{align}
				\EndFor
				\State \textbf{Output:} Strain components $s$ and $a$, optical flow $u$
			\end{algorithmic}
			\caption{PDHGMp for optical flow and strain computation by $E_{\TGV}$.}
			\label{strain:alg1tgv}
		\end{algorithm}
	}
		
	We have to comment on the proximal steps within the algorithm.
	The update for $a$ is straightforward.
	Regarding $s,t$,
	observe that the problems can be separated into $N$ subproblems, e.g.~for $s$ they have the form
	\begin{align}
	\hat s = \argmin_{s \in \R^4} \left\{ \| s \|_{2} + \frac{\tau_2}{2 \lambda_1} \|s-x\|_2^2\right\},
	\end{align}
	where $x = (b_1^{(r)} + \boldsymbol{\nabla}   u^{(r+1)} - a^{(r+1)})_{i+jN, j=0,\ldots,3} \in \R^4$, $i \in 1,\ldots,N$.
	The solution of these problems is given by a grouped or coupled shrinkage, see the appendix.	
	If we extend the model by the positivity constraint, i.e.~\eqref{flow:constr:tgv}, the update step for $s$ becomes a bit more involved since we need to compute $N$ problems of the form
	\begin{align}
	\hat s = \argmin_{s \in \R^4} \left\{ \| s \|_{2} + \iota_{\ge0}(s_1) + \frac{\tau_2}{2 \lambda_1} \|s-x\|_2^2\right\}
	\end{align}
	with $x$ as above.
	For $x_1 \ge 0$, the term $\iota_{\ge0}(s_1)$ can be neglected and we end up with the usual coupled shrinkage.
	For $x_1 < 0$, we have $\hat s_1 = 0$ and for the remaining three components the usual coupled shrinkage can be applied.	
	
	Due to the diagonal structure of $A$ and $B$, the proximal step to get $u$ can be separated into $N$ subproblems of the form
	\begin{align}\label{genshrink:eq}
	\hat y = \argmin_{y\in\R^2} \left\{|\alpha y_1 + \beta y_2 + \gamma| + \frac12 \|y-x\|_2^2\right\},
	\end{align}
	where $\alpha, \beta, \gamma \in \R$.
	The solution is a generalized soft shrinkage of $x$ explained in the appendix.
	
	\begin{remark}[Computation of the strain $\varepsilon$ within the primal-dual algorithm] \label{strain_comp}
		The primal-dual method uses the Lagrangian of \eqref{sec2}, which contains the summand
		$\langle (\boldsymbol{\nabla}u-a)-s,b\rangle$, with the dual variable $b$.
		Hence, the algorithm computes 
		the desired strain tensor
		directly within the iteration process
		and no subsequent computation of the derivative 
		of the optical flow is required. 		
	\end{remark}

	We updated the basic PDHGMp algorithm by a coarse-to-fine scheme as proposed, e.g., in \cite{Ana89,BBPW04,SRB10,SRB13}
	to cope with large displacements. Moreover, we applied the common trick median filtering on every scale to improve the results, see also \cite{SRB10,SRB13}.
	For more information and the pseudo code for the coarse-to-fine scheme we refer to \cite{Fitschen_2017}.
	To ensure convergence of the algorithm, we use $3000$ iterations of the PDHGMp algorithm in each level of the coarse-to-fine scheme.
	The high number of iterations leads to a computation time of about $20$ minutes which is not competitive to state-of-the-art optical flow methods.
	We want to emphasize that there is still some potential to optimize the computation time but for our applications this was not necessary since the computation time is negligible in comparison to the time demand of the experiments.
	

	\section{Numerical Examples} \label{sec:strain:num}

	In this section, we present and analyze numerical results of our approaches.	
	
	First, we consider artificial data to demonstrate the behavior of different optical flow models
	and to underline why the TGV regularized model appears to be appropriate for the described practical tasks.
	Then, we deal with displacements in aluminum matrix composites (AMC) during tensile tests, 
	where we focus on the detection of local damage and crack propagation.
	Finally, we demonstrate the flexibility of our method by showing results 
	for different experimental settings and materials.
	
	The algorithm is implemented in C++.
	Unless stated otherwise, we set 
	\begin{align}\label{eq:strain:param}
	\lambda_1 \coloneqq 0.2 \ \mathrm{and} \ \lambda_2 \coloneqq 10
	\end{align} 
	for all real-world examples.
	The positivity constraint \eqref{flow:constr:tgv} on the strain 
	is used for the real-world examples of tensile tests with AMCs.
	For the artificial examples 
	and the compression and fatigue tests,  
	we do not involve the constraints.	
	
	The data in Figure~\ref{fig:compress} was provided by K.~Lichtenberg from the ``Institute for Applied Materials (IAM)'' at the Karlsruhe Institute of Technology (KIT).
	
	\subsection{Artificial Examples} \label{sub:num_art} 
	
	To explain the differences  between various regularization terms,
	we use the following simple example:
	a segment of $100 \times 100$ pixels of one exemplary micrograph of an aluminum composite is taken 
	and the simulated displacement field in \textbf{Figure~\ref{fig:strain:toy}} consisting of the sum of a piecewise constant 
	and a linear part is applied to warp this image. 
	Then, the warped and the initial image are used as the input to reconstruct the simulated   
	displacement field and its strain tensor via variational models 
	with the same data term \eqref{e_flow} and 	the following priors:
	besides the TGV prior, 
	we used the
	\begin{itemize}
		\item[i)]  $\HH^1$ regularizer, e.g., used in the Horn-Schunck model \cite{HS81},	
		\begin{align}
		\lambda \HH^1(u) \coloneqq \lambda \| \boldsymbol{\nabla} u  \|_2^2, \tag{$\HH^1$} \label{H1}
		\end{align}
		\item[ii)]   $\mathrm{TV}$ regularizer proposed in \cite{ROF92} and used for optical flow, e.g.~in \cite{BBPW04},
		\begin{align}
		\lambda \TV(u) \coloneqq \lambda \| \boldsymbol{\nabla} u  \|_{2,1}, \tag{\text{TV}} \label{TV}
		\end{align}
		\item[iii)]   $\mathrm{TV}_2$ regularizer with second order differences
		\begin{align}
		\lambda \TV_2(u) \coloneqq \lambda \| \boldsymbol{\nabla}^2 u  \|_{2,1}, \tag{$\text{TV}_2$} \label{TV_2}
		\end{align}
		\item[iv)]    $\TV\operatorname{-}\TV_2$ regularizer, e.g.~in~\cite{PS13},
		\begin{align}
		\TV_{1\operatorname{-}2}(u) \coloneqq \lambda_1 \| \boldsymbol{\nabla} u  \|_{2,1} + \lambda_2 \|  \boldsymbol{\nabla}^2 u  \|_{2,1} ,\quad
		\tag{$\TV\operatorname{-}\TV_2$} \label{TV-TV2}
		\end{align}	
		\item[v)]    $\IC$ regularizer
		\begin{align}
		\IC(u) \coloneqq \min_{u=v+w} \lambda_1 \| \boldsymbol{\nabla} v  \|_{2,1} + \lambda_2 \|  \boldsymbol{\nabla}^2 w  \|_{2,1} + \lambda_3 \|w\|_2^2 ,\quad
		\tag{$\IC$} \label{IC}
		\end{align}			
		where we added a small quadratic term to 
		avoid the ambiguity introduced by the difference matrices.
		We set $\lambda_3 \coloneqq 5 \cdot 10^{-5}$.
		Note that such a term in relation with the difference operators was not necessary	for the $\TGV$ regularizer since we do not decompose $u$, but the strain.
	\end{itemize}
	
	\begin{figure}
		\centering
		\begin{tabular}{cccc}
			{\small original $u_1$\hspace{0.4cm}} &
			{\small constant parts} &
			{\small linear part}&
			{\small derivative $\nabla_x u_1$}
			\\ 			
			\begin{tikzpicture}\tikzstyle{every node}=[font=\tiny]
			\begin{axis}[width=.31\textwidth,
			enlargelimits=false, 
			hide axis,
			axis equal image,
			colorbar,
			colormap name=parula,
			colorbar style={overlay,width=0.15cm, yticklabel style={
					/pgf/number format/.cd,
					fixed,
					fixed zerofill,
					precision=1,
					/tikz/.cd}}
			]
			\addplot[point meta min=-2, point meta max=2] 
			graphics[xmin=0, xmax=1, ymin=0, ymax=1]{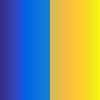};
			\end{axis}
			\end{tikzpicture} 	\hspace{0.2cm}
			&
			\hspace{0.2cm}
			\begin{tikzpicture}\tikzstyle{every node}=[font=\tiny]
			\begin{axis}[width=.31\textwidth,
			enlargelimits=false, 
			hide axis,
			axis equal image,
			colorbar,
			colormap name=parula,
			colorbar style={overlay,width=0.15cm, yticklabel style={
					/pgf/number format/.cd,
					fixed,
					fixed zerofill,
					precision=1,
					/tikz/.cd}}
			]
			\addplot[point meta min=-1, point meta max=1] 
			graphics[xmin=0, xmax=1, ymin=0, ymax=1]{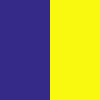};
			\end{axis}
			\end{tikzpicture} \hspace{0.2cm}
			&
			\hspace{0.2cm}
			\begin{tikzpicture}\tikzstyle{every node}=[font=\tiny]
			\begin{axis}[width=.31\textwidth,
			enlargelimits=false, 
			hide axis,
			axis equal image,
			colorbar,
			colormap name=parula,
			colorbar style={overlay,width=0.15cm, yticklabel style={
					/pgf/number format/.cd,
					fixed,
					fixed zerofill,
					precision=1,
					/tikz/.cd}}
			]
			\addplot[point meta min=-1, point meta max=1] 
			graphics[xmin=0, xmax=1, ymin=0, ymax=1]{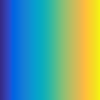};
			\end{axis}
			\end{tikzpicture} \hspace{0.2cm}
			&
			\hspace{0.2cm}
			\begin{tikzpicture}\tikzstyle{every node}=[font=\tiny]
			\begin{axis}[width=.31\textwidth,
			enlargelimits=false, 
			hide axis,
			axis equal image,
			colorbar,
			colormap name=parula,
			colorbar style={overlay,width=0.15cm, yticklabel style={
					/pgf/number format/.cd,
					fixed,
					fixed zerofill,
					precision=1,
					/tikz/.cd}}
			]
			\addplot[point meta min=0, point meta max=2] 
			graphics[xmin=0, xmax=1, ymin=0, ymax=1]{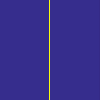};
			\end{axis}
			\end{tikzpicture} \hspace{0.4cm}
		\end{tabular}
		\caption{\label{fig:strain:toy}
			Original displacement $u_1$ in pixels, additively composed of a piecewise constant and a linear part and 
			$\nabla_x u_1$.} 
	\end{figure}
	
	The minimizers $u_1$ of the various models and the strain component $\varepsilon_{11} = \nabla_x u_1$  
	are shown in \textbf{Figures~\ref{fig:strain:toy_deriv:a}} and \textbf{\ref{fig:strain:toy_deriv:b}}, respectively.
	The parameters stated in the caption are optimized with respect to the best visual impression.
	The models with $\HH^1$ and $\TV_2$ priors cannot find a sharp displacement field boundary, 
	while the $\TV$ model introduces additional boundaries due to the so-called staircasing effect.
	The $\TV\operatorname{-}\TV_2$ regularized model is better, but it is clearly worse than the IC and TGV methods.
	Clearly, for this example both methods IC and TGV give good results since $u_1$ is indeed
	additively composed of a non-smooth and a smooth component. The next example does not have this property.
	
	%
	\begin{figure} \centering
		\begin{subfigure}[b]{0.99\textwidth}\centering
			\begin{tabular}{ccc}
				\eqref{H1}& 
				\eqref{TV}&
				\eqref{TV_2} 
				\\
				\begin{tikzpicture}\tikzstyle{every node}=[font=\tiny]
				\begin{axis}[width=.4\textwidth,
				enlargelimits=false, 
				hide axis,
				axis equal image
				]
				\addplot[point meta min=-2, point meta max=2] 
				graphics[xmin=0, xmax=1, ymin=0, ymax=1]{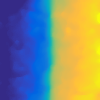};
				\end{axis}
				\end{tikzpicture} 
				&
				\begin{tikzpicture}\tikzstyle{every node}=[font=\tiny]
				\begin{axis}[width=.4\textwidth,
				enlargelimits=false, 
				hide axis,
				axis equal image
				]
				\addplot[point meta min=-2, point meta max=2] 
				graphics[xmin=0, xmax=1, ymin=0, ymax=1]{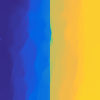};
				\end{axis}
				\end{tikzpicture} 	
				&		
				\begin{tikzpicture}\tikzstyle{every node}=[font=\tiny]
				\begin{axis}[width=.4\textwidth,
				enlargelimits=false, 
				hide axis,
				axis equal image,
				colorbar,
				colormap name=parula,
				colorbar  style={overlay,width=0.15cm, yticklabel style={
						/pgf/number format/.cd,
						fixed,
						fixed zerofill,
						precision=1,
						/tikz/.cd}}
				]
				\addplot[point meta min=-2, point meta max=2] 
				graphics[xmin=0, xmax=1, ymin=0, ymax=1]{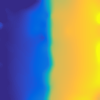};
				\end{axis}
				\end{tikzpicture} 	\hspace{0.3cm}
				\\[1ex]
				\eqref{TV-TV2} &
				\eqref{IC} 	   &
				\eqref{tgv} 	 
				\\	
				\begin{tikzpicture}\tikzstyle{every node}=[font=\tiny]
				\begin{axis}[width=.4\textwidth,
				enlargelimits=false, 
				hide axis,
				axis equal image
				]
				\addplot[point meta min=-2, point meta max=2] 
				graphics[xmin=0, xmax=1, ymin=0, ymax=1]{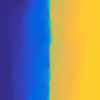};
				\end{axis}
				\end{tikzpicture} 
				&
				\begin{tikzpicture}\tikzstyle{every node}=[font=\tiny]
				\begin{axis}[width=.4\textwidth,
				enlargelimits=false, 
				hide axis,
				axis equal image
				]
				\addplot[point meta min=-2, point meta max=2] 
				graphics[xmin=0, xmax=1, ymin=0, ymax=1]{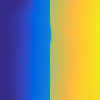};
				\end{axis}
				\end{tikzpicture} 	
				&
				\begin{tikzpicture}\tikzstyle{every node}=[font=\tiny]
				\begin{axis}[width=.4\textwidth,
				enlargelimits=false, 
				hide axis,
				axis equal image,
				colorbar,
				colormap name=parula,
				colorbar  style={overlay,width=0.15cm, yticklabel style={
						/pgf/number format/.cd,
						fixed,
						fixed zerofill,
						precision=1,
						/tikz/.cd}}
				]
				\addplot[point meta min=-2, point meta max=2] 
				graphics[xmin=0, xmax=1, ymin=0, ymax=1]{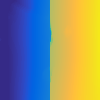};
				\end{axis}
				\end{tikzpicture} \hspace{0.3cm}
			\end{tabular}
			\caption{ \label{fig:strain:toy_deriv:a} Displacement $u_1$ in pixels using various regularization terms.}
		\end{subfigure}\\[3ex]
		\begin{subfigure}[b]{0.99\textwidth}\centering
			\begin{tabular}{ccc}
				\eqref{H1}	&
				\eqref{TV}	&
				\eqref{TV_2} 
				\\
				\begin{tikzpicture}\tikzstyle{every node}=[font=\tiny]
				\begin{axis}[width=.4\textwidth,
				enlargelimits=false, 
				hide axis,
				axis equal image
				]
				\addplot[point meta min=0, point meta max=2] 
				graphics[xmin=0, xmax=1, ymin=0, ymax=1]{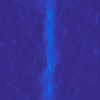};
				\end{axis}
				\end{tikzpicture} 
				&
				\begin{tikzpicture}\tikzstyle{every node}=[font=\tiny]
				\begin{axis}[width=.4\textwidth,
				enlargelimits=false, 
				hide axis,
				axis equal image
				]
				\addplot[point meta min=0, point meta max=2] 
				graphics[xmin=0, xmax=1, ymin=0, ymax=1]{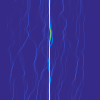};
				\end{axis}
				\end{tikzpicture}
				&	
				\begin{tikzpicture}\tikzstyle{every node}=[font=\tiny]
				\begin{axis}[width=.4\textwidth,
				enlargelimits=false, 
				hide axis,
				axis equal image,
				colorbar,
				colormap name=parula,
				colorbar  style={overlay,width=0.15cm, yticklabel style={
						/pgf/number format/.cd,
						fixed,
						fixed zerofill,
						precision=1,
						/tikz/.cd}}
				]
				\addplot[point meta min=0, point meta max=2] 
				graphics[xmin=0, xmax=1, ymin=0, ymax=1]{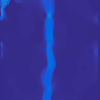};
				\end{axis}
				\end{tikzpicture} 		\hspace{0.3cm}	 
				\\[1ex]
				\eqref{TV-TV2} &
				\eqref{IC} 	  &
				\eqref{tgv} 	  
				\\
				\begin{tikzpicture}\tikzstyle{every node}=[font=\tiny]
				\begin{axis}[width=.4\textwidth,
				enlargelimits=false, 
				hide axis,
				axis equal image
				]
				\addplot[point meta min=0, point meta max=2] 
				graphics[xmin=0, xmax=1, ymin=0, ymax=1]{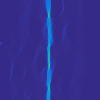};
				\end{axis}
				\end{tikzpicture} 
				&
				\begin{tikzpicture}\tikzstyle{every node}=[font=\tiny]
				\begin{axis}[width=.4\textwidth,
				enlargelimits=false, 
				hide axis,
				axis equal image
				]
				\addplot[point meta min=0, point meta max=2] 
				graphics[xmin=0, xmax=1, ymin=0, ymax=1]{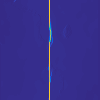};
				\end{axis}
				\end{tikzpicture}	
				&
				\begin{tikzpicture}\tikzstyle{every node}=[font=\tiny]
				\begin{axis}[width=.4\textwidth,
				enlargelimits=false, 
				hide axis,
				axis equal image,
				colorbar,
				colormap name=parula,
				colorbar  style={overlay,width=0.15cm, yticklabel style={
						/pgf/number format/.cd,
						fixed,
						fixed zerofill,
						precision=1,
						/tikz/.cd}}
				]
				\addplot[point meta min=0, point meta max=2] 
				graphics[xmin=0, xmax=1, ymin=0, ymax=1]{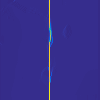};
				\end{axis}
				\end{tikzpicture}	\hspace{0.3cm}	
			\end{tabular}
			\caption{ \label{fig:strain:toy_deriv:b} Derivative $\nabla_x u_1$ using various regularization terms.}
		\end{subfigure}
		\caption{ \label{fig:strain:toy_deriv}
			Results for the simulated example from \textbf{Figure~\ref{fig:strain:toy}} for various regularization terms.
			Parameters: $\lambda = 50$ for \eqref{H1}, $\lambda = 0.1$ for \eqref{TV}, $\lambda = 0.1$ for \eqref{TV_2}, $\lambda_1 = 0.1$, $\lambda_2=0.02$ for \eqref{TV-TV2}, $\lambda_1 = 0.1$, $\lambda_2=1$, $\lambda_3=0.5\cdot 10^{-5}$ for  \eqref{IC} and $\lambda_1 = 0.1$, $\lambda_2=2$ for \eqref{tgv}.} 
	\end{figure}	
	
	Our next artificial example in
	\textbf{Figure~\ref{fig:art_tgv_ic}} shows a displacement field consisting of two parts. 
	The lower part contains a jump whereas we have a purely linear transition in the upper part.
	In terms of a tensile test, this might be seen as a crack in the lower part 
	and a purely elastic deformation in the upper part of the image.
	As the displacement field is a transition between the non-smooth upper part and the smooth lower one,
	it can not be split additively in an appropriate way, whereas it is actually possible to split the strain.
	Therefore, the TGV model clearly outperforms the IC model.
	The parameters for the TGV model are those in \eqref{eq:strain:param}.
	Compared to the TGV model, we chose smaller parameters for the IC model, namely $\lambda_1 = 0.1$, $\lambda_2 = 1$, $\lambda_3=0.5\cdot 10^{-5}$,
	since otherwise the resulting displacement was too smooth.
	We also added the result of the large displacement optical flow (LDOF) model by Brox and Malik~\cite{BM11} 
	using their implementation with standard parameters.
	However, it is clearly visible that the result contains too much structures in the smooth upper part and at the same time is not able to reconstruct the jump in the lower part.
	
	\begin{figure} \centering
		\begin{subfigure}[b]{0.99\textwidth}\centering
			\begin{tabular}{cccc}
				true &
				\eqref{IC} &
				\eqref{tgv}  &
				LDOF \cite{BM11}
				\\	
				\begin{tikzpicture}\tikzstyle{every node}=[font=\tiny]
				\begin{axis}[width=.35\textwidth,
				enlargelimits=false, 
				hide axis,
				axis equal image
				]
				\addplot[point meta min=-2, point meta max=2] 
				graphics[xmin=0, xmax=1, ymin=0, ymax=1]{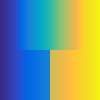};
				\end{axis}
				\end{tikzpicture} 
				&
				\begin{tikzpicture}\tikzstyle{every node}=[font=\tiny]
				\begin{axis}[width=.35\textwidth,
				enlargelimits=false, 
				hide axis,
				axis equal image
				]
				\addplot[point meta min=-2, point meta max=2] 
				graphics[xmin=0, xmax=1, ymin=0, ymax=1]{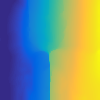};
				\end{axis}
				\end{tikzpicture} 
				&
				\begin{tikzpicture}\tikzstyle{every node}=[font=\tiny]
				\begin{axis}[width=.35\textwidth,
				enlargelimits=false, 
				hide axis,
				axis equal image
				]
				\addplot[point meta min=-2, point meta max=2] 
				graphics[xmin=0, xmax=1, ymin=0, ymax=1]{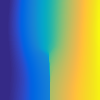};
				\end{axis}
				\end{tikzpicture} 
				&
				\begin{tikzpicture}\tikzstyle{every node}=[font=\tiny]
				\begin{axis}[width=.35\textwidth,		
				enlargelimits=false, 
				hide axis, 
				axis equal image, 
				axis on top, 
				colorbar, 
				scaled ticks=false, 
				colormap name=parula, 
				colorbar right, 
				colorbar  style={width=0.15cm, yticklabel style={overlay,
						/pgf/number format/.cd,
						fixed,
						fixed zerofill,
						precision=1,
						/tikz/.cd},
					overlay}
				]
				\addplot[point meta min=-2, point meta max=2] 
				graphics[xmin=0, xmax=1, ymin=0, ymax=1]{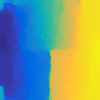};
				\end{axis}
				\end{tikzpicture} \hspace{0.3cm}	
			\end{tabular}
			\caption{Displacement $u_1$ in pixels using IC, TGV regularization, and the method in~\cite{BM11}.}
		\end{subfigure}\\[2ex]
		\begin{subfigure}[b]{0.99\textwidth}\centering
			\begin{tabular}{cccc}
				true&
				\eqref{IC} &
				\eqref{tgv} &
				LDOF \cite{BM11}
				\\	
				\begin{tikzpicture}\tikzstyle{every node}=[font=\tiny]
				\begin{axis}[width=.35\textwidth,
				enlargelimits=false, 
				hide axis,
				axis equal image
				]
				\addplot[point meta min=0, point meta max=2] 
				graphics[xmin=0, xmax=1, ymin=0, ymax=1]{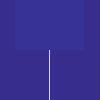};
				\end{axis}
				\end{tikzpicture} 
				&
				\begin{tikzpicture}\tikzstyle{every node}=[font=\tiny]
				\begin{axis}[width=.35\textwidth,
				enlargelimits=false, 
				hide axis,
				axis equal image
				]
				\addplot[point meta min=0, point meta max=2] 
				graphics[xmin=0, xmax=1, ymin=0, ymax=1]{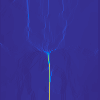};
				\end{axis}
				\end{tikzpicture} 
				&
				\begin{tikzpicture}\tikzstyle{every node}=[font=\tiny]
				\begin{axis}[width=.35\textwidth,
				enlargelimits=false, 
				hide axis,
				axis equal image
				]
				\addplot[point meta min=0, point meta max=2] 
				graphics[xmin=0, xmax=1, ymin=0, ymax=1]{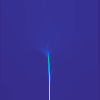};
				\end{axis}
				\end{tikzpicture} 
				&
				\begin{tikzpicture}\tikzstyle{every node}=[font=\tiny]
				\begin{axis}[width=.35\textwidth,
				enlargelimits=false, 
				hide axis,
				axis equal image,
				colorbar,
				colormap name=parula,
				colorbar  style={overlay,width=0.15cm, yticklabel style={
						/pgf/number format/.cd,
						fixed,
						fixed zerofill,
						precision=1,
						/tikz/.cd}}
				]
				\addplot[point meta min=0, point meta max=2] 
				graphics[xmin=0, xmax=1, ymin=0, ymax=1]{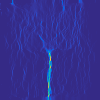};
				\end{axis}
				\end{tikzpicture}	 \hspace{0.3cm}	
			\end{tabular}
			\caption{Derivative $\nabla_x u_1$ using IC, TGV regularization, and the method in~\cite{BM11}.}
		\end{subfigure}
		\caption{ \label{fig:art_tgv_ic}
			Results for a simulated example 
			with a linear transition of the displacement field in the upper half and a jump in the lower half of the image.} 
	\end{figure}

	\subsection{Tensile Tests with AMC} \label{sub:num_tensile}

	Next, we deal with tensile tests of aluminum matrix composites (AMC).
	An AMC material is highly suitable for lightweight applications in aerospace, 
	defense and automotive industry.
	Compared to monolithic materials, composite materials have advantageous properties such as
	higher ultimate tensile strength or a higher stiffness to density ratio. Here, we mainly focus on calculating local strains 
	of silicon carbide particle reinforced AMCs 
	from a sequence of scanning electron microscope (SEM) images acquired during tensile tests, 
	where the specimen is pulled in $x$-direction and elongates with increasing force.
	Figure~\ref{fig:sem_image} illustrates the experimental setup and the resulting image sequence schematically. 
	We are interested in the local deformation behavior of the composite material on a microscale.
	Therefore, it is necessary to perform SEM monitored tensile tests to study
	deformation and crack initiation due to the inhomogeneous microstructure.
	We will especially focus on $\varepsilon_{11}$, which describes the change in displacement $u_1$ for the $x$-direction.
	A positive value indicates tension and a negative one compression.
	
	Unless stated otherwise, we use the image under no or low load as the reference image $f_1$ 
	and compute the displacement between the reference image and images under higher load.
	Hence, the displacement and strain are given in the coordinates of the reference image and we overlay the results with the reference image.
	
	The real-world example in \textbf{Figure~\ref{fig:real_tgv_ic}} illustrates again the difference between the TGV and IC regularized models. 
	In this example, the material has a few cracks in its initial state and they widen up, but the surrounding aluminum matrix deforms smoothly.
	Whereas the computed displacement $u_1$ looks roughly the same for both methods, the differences are visible in its derivative $\nabla_x u_1$.
	The TGV model only shows some structure where the cracks actually are, but the IC model leads to structures 
	that pass through the whole image since it is, in analogy to the previous artificial example, 
	impossible to split the displacement into two parts for this example.
	The TGV model outperforms the IC model in this setting.
	Nevertheless, for the cracks themselves, both methods lead to equally good results 
	as the lower half of the enlarged region in Figure~\ref{fig:real_tgv_ic} shows.
	In summary, both methods show the main structures, but the IC model introduces wrong structures.
	Therefore, we focus on the TGV model in the rest of our examples.
	
	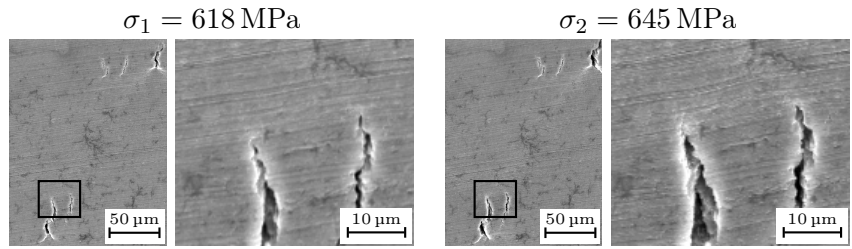
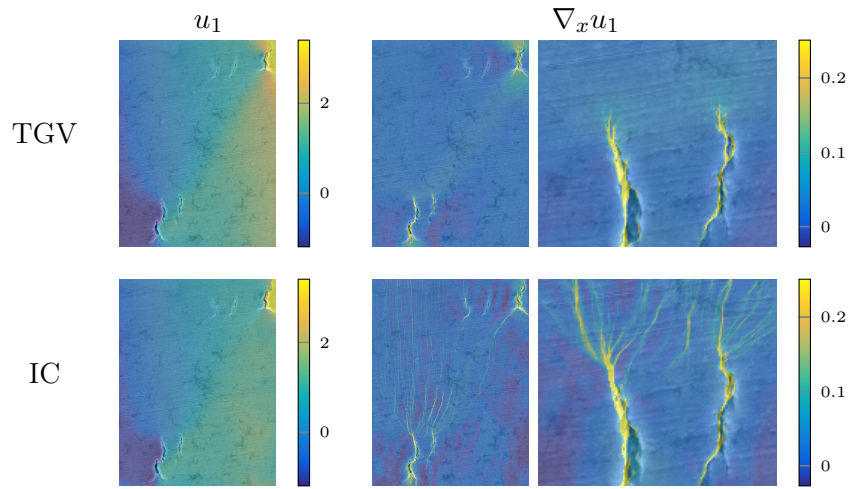
\begin{figure} 
		\centering	
		\begin{subfigure}[b]{0.99\textwidth}\centering
			\begin{tabular}{cc}
				$\sigma_1 = 618$\,MPa & $\sigma_2 = 645$\,MPa\\
				\begin{tikzpicture}\tikzstyle{every node}=[font=\tiny] 
				\begin{axis}[ 
				width=.40\textwidth, 
				enlargelimits=false, 
				hide axis, 
				axis equal image, 
				axis on top, 
				] 
				\addplot
				graphics[xmin=0, xmax=600, ymin=0, ymax=800] 
				{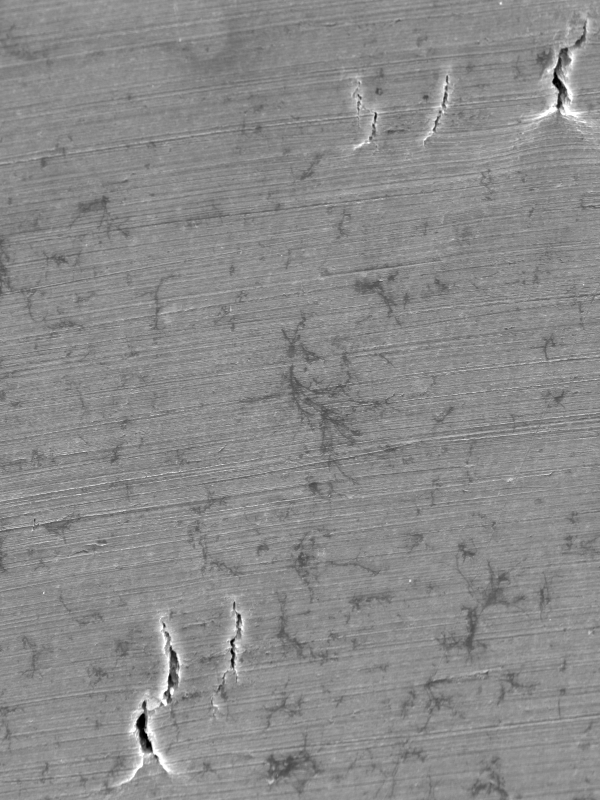};  
				\draw[color=black,thick] (111,111) rectangle (270,250);
				\draw[draw=none,fill=white] (364,5) rectangle (594,140);
				\draw[{|[width=4pt]}-{|[width=4pt]},thick] (379,50) -- (579,50) node[above=-2pt,pos=0.5] {50\,\textmu m};
				\end{axis} 
				\end{tikzpicture}
				\begin{tikzpicture}\tikzstyle{every node}=[font=\tiny] 
				\begin{axis}[ 
				width=.40\textwidth, 
				enlargelimits=false, 
				hide axis, 
				axis equal image, 
				axis on top, 
				] 
				\addplot
				graphics[xmin=0, xmax=160, ymin=0, ymax=140] 
				{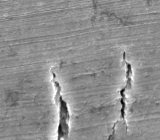};
				\draw[draw=none,fill=white] (111,1) rectangle (159,25);
				\draw[{|[width=4pt]}-{|[width=4pt]},thick] (115,9) -- (155,9) node[above=-2pt,pos=0.5] {10\,\textmu m}; 
				\end{axis} 
				\end{tikzpicture} 
				&
				\begin{tikzpicture}\tikzstyle{every node}=[font=\tiny] 
				\begin{axis}[ 
				width=.40\textwidth, 
				enlargelimits=false, 
				hide axis, 
				axis equal image, 
				axis on top, 
				] 
				\addplot
				graphics[xmin=0, xmax=600, ymin=0, ymax=800] 
				{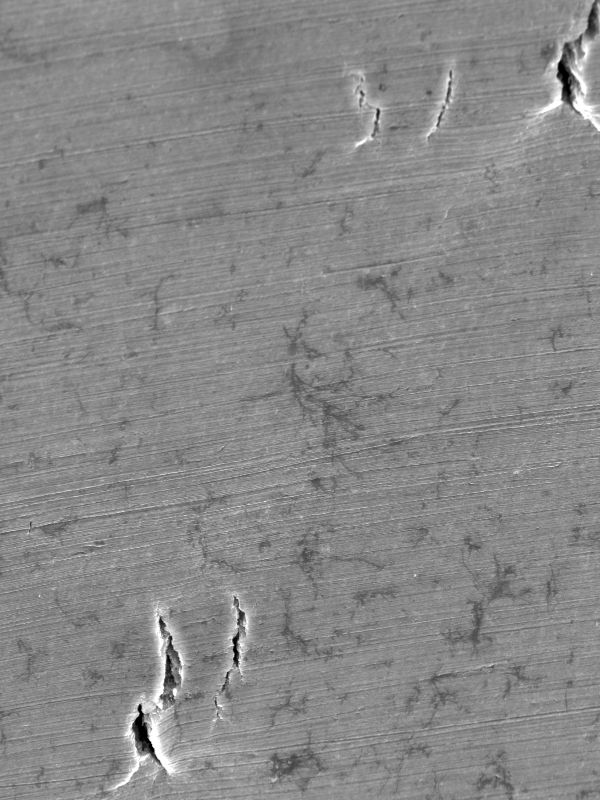};  
				\draw[color=black,thick] (111,111) rectangle (270,250);
				\draw[draw=none,fill=white] (364,5) rectangle (594,140);
				\draw[{|[width=4pt]}-{|[width=4pt]},thick] (379,50) -- (579,50) node[above=-2pt,pos=0.5] {50\,\textmu m};
				\end{axis} 
				\end{tikzpicture}
				\begin{tikzpicture}\tikzstyle{every node}=[font=\tiny] 
				\begin{axis}[ 
				width=.40\textwidth, 
				enlargelimits=false, 
				hide axis, 
				axis equal image, 
				axis on top, 
				] 
				\addplot
				graphics[xmin=0, xmax=160, ymin=0, ymax=140] 
				{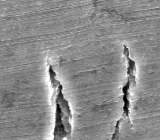}; 
				\draw[draw=none,fill=white] (111,1) rectangle (159,25);
				\draw[{|[width=4pt]}-{|[width=4pt]},thick] (115,9) -- (155,9) node[above=-2pt,pos=0.5] {10\,\textmu m}; 
				\end{axis} 
				\end{tikzpicture}
			\end{tabular}
			\caption{Microstructure images of the specimen under low and high load.}
		\end{subfigure}
		\\[2ex]
		\begin{subfigure}[b]{0.99\textwidth}\centering
			\begin{tabular}{ccc}
				& $u_1$ & $\nabla_x u_1$\\[.1ex]
				TGV
				&
				\raisebox{-.5\height}{
					\begin{tikzpicture}\tikzstyle{every node}=[font=\tiny] 
					\begin{axis}[ 
					width=.40\textwidth, 
					enlargelimits=false, 
					hide axis, 
					axis equal image, 
					axis on top, 
					colorbar, 
					scaled ticks=false, 
					colormap name=parula, 
					colorbar right, 
					colorbar style={ 
						width=0.15cm, 
						overlay
					}
					] 
					\addplot[surf,point meta min=-1.195, point meta max=3.4075] 
					graphics[xmin=0, xmax=600, ymin=0, ymax=800] 
					{./images/strain/results_Serie58/image_1.png}; 
					\addplot[opacity=0.5]  
					graphics[xmin=0, xmax=600, ymin=0, ymax=800] 
					{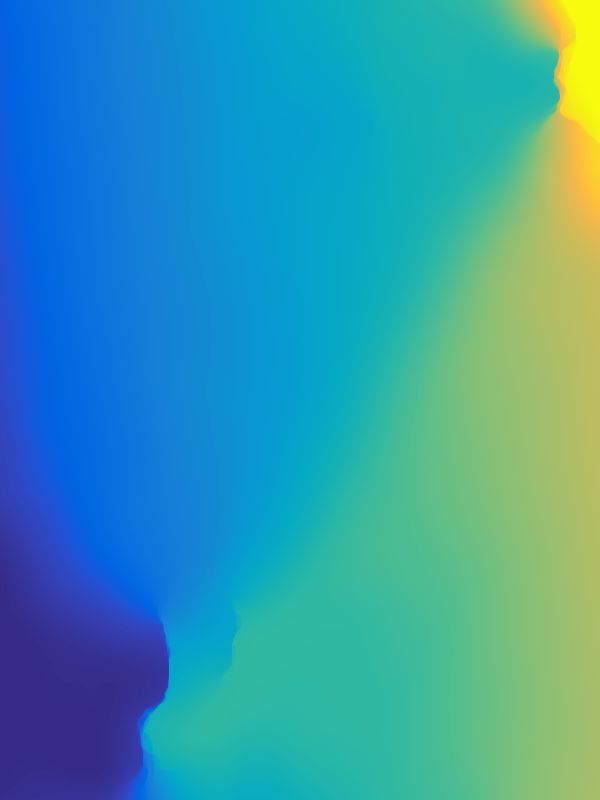}; 
					\end{axis} 
					\end{tikzpicture} 
				}
				&	
				\raisebox{-.5\height}{\hspace{0.3cm}
					\begin{tikzpicture}\tikzstyle{every node}=[font=\tiny] 
					\begin{axis}[ 
					width=.40\textwidth, 
					enlargelimits=false, 
					hide axis, 
					axis equal image, 
					axis on top
					] 
					\addplot[surf,point meta min=-0.027, point meta max=0.251]  
					graphics[xmin=0, xmax=600, ymin=0, ymax=800] 
					{./images/strain/results_Serie58/image_1.png}; 
					\addplot[opacity=0.5]  
					graphics[xmin=0, xmax=600, ymin=0, ymax=800] 
					{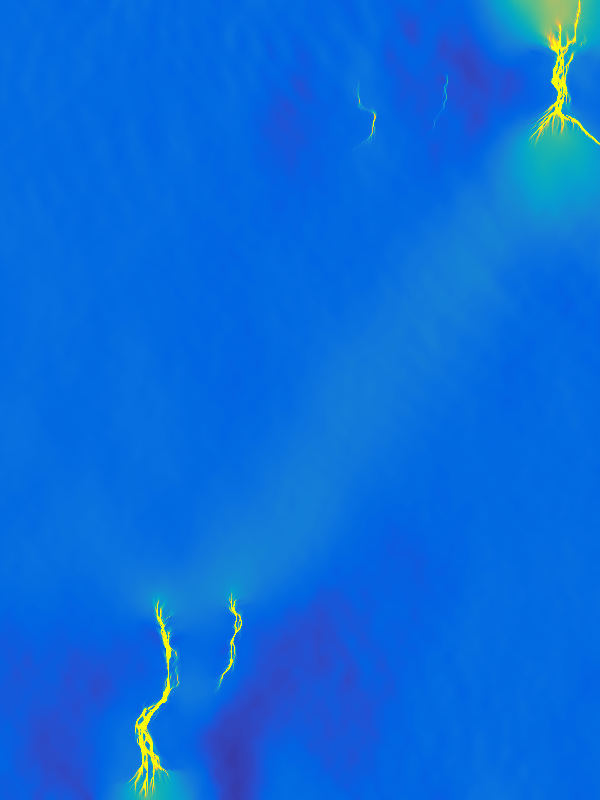}; 
					\end{axis} 
					\end{tikzpicture} 
					\begin{tikzpicture}\tikzstyle{every node}=[font=\tiny] 
					\begin{axis}[ 
					width=.40\textwidth, 
					enlargelimits=false, 
					hide axis, 
					axis equal image, 
					axis on top, 
					colorbar, 
					scaled ticks=false, 
					colormap name=parula, 
					colorbar right, 
					colorbar style={ 
						width=0.15cm,
						overlay 
					} 
					] 
					\addplot[surf,point meta min=-0.027, point meta max=0.251] 
					graphics[xmin=0, xmax=160, ymin=0, ymax=140] 
					{./images/strain/results_Serie58/regionTGV_im0.png}; 
					\addplot[opacity=0.5]  
					graphics[xmin=0, xmax=160, ymin=0, ymax=140] 
					{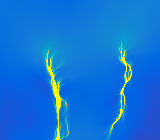}; 
					\end{axis} 
					\end{tikzpicture}\hspace{0.3cm}
				} 
				\\[10ex]
				IC
				&
				\raisebox{-.5\height}{
					\begin{tikzpicture}\tikzstyle{every node}=[font=\tiny] 
					\begin{axis}[ 
					width=.40\textwidth, 
					enlargelimits=false, 
					hide axis, 
					axis equal image, 
					axis on top, 
					colorbar, 
					scaled ticks=false, 
					colormap name=parula, 
					colorbar right, 
					colorbar style={ 
						width=0.15cm, 
						overlay
					}
					] 
					\addplot[surf,point meta min=-1.195, point meta max=3.4075] 
					graphics[xmin=0, xmax=600, ymin=0, ymax=800] 
					{./images/strain/results_Serie58/image_1.png}; 
					\addplot[opacity=0.5]  
					graphics[xmin=0, xmax=600, ymin=0, ymax=800] 
					{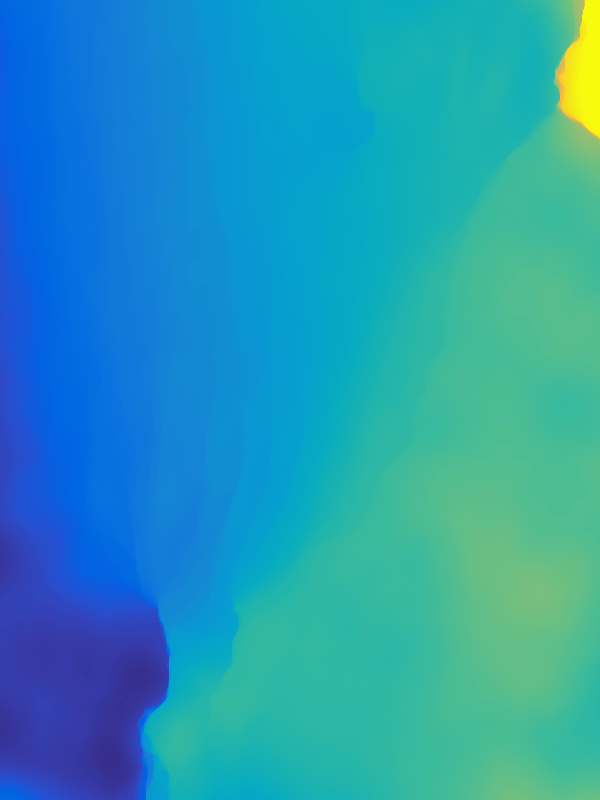}; 
					\end{axis} 
					\end{tikzpicture} 
				}
				&	
				\raisebox{-.5\height}{\hspace{0.3cm}
					\begin{tikzpicture}\tikzstyle{every node}=[font=\tiny] 
					\begin{axis}[ 
					width=.40\textwidth, 
					enlargelimits=false, 
					hide axis, 
					axis equal image, 
					axis on top
					] 
					\addplot[surf,point meta min=-0.027, point meta max=0.251]  
					graphics[xmin=0, xmax=600, ymin=0, ymax=800] 
					{./images/strain/results_Serie58/image_1.png}; 
					\addplot[opacity=0.5]  
					graphics[xmin=0, xmax=600, ymin=0, ymax=800] 
					{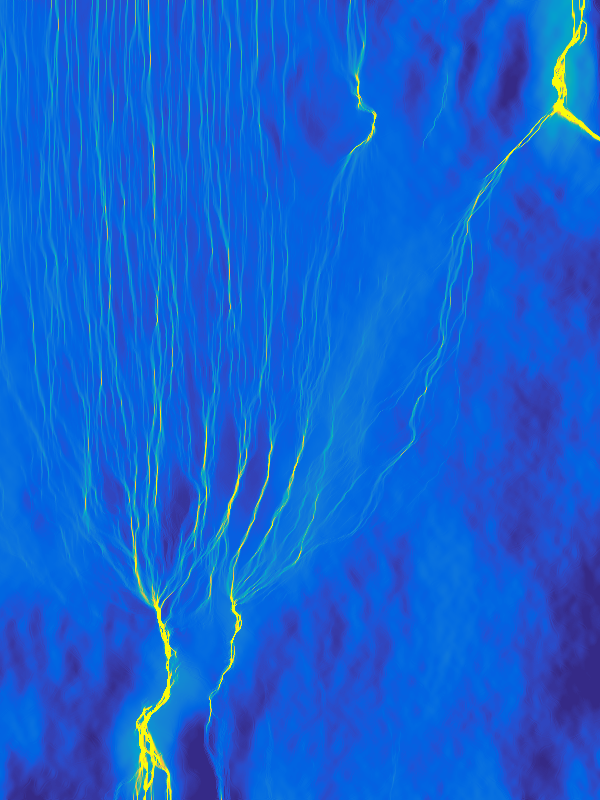}; 
					\end{axis} 
					\end{tikzpicture} 
					\begin{tikzpicture}\tikzstyle{every node}=[font=\tiny] 
					\begin{axis}[ 
					width=.40\textwidth, 
					enlargelimits=false, 
					hide axis, 
					axis equal image, 
					axis on top, 
					colorbar, 
					scaled ticks=false, 
					colormap name=parula, 
					colorbar right, 
					colorbar style={ 
						width=0.15cm, 
						overlay
					} 
					] 
					\addplot[surf,point meta min=-0.027, point meta max=0.251] 
					graphics[xmin=0, xmax=160, ymin=0, ymax=140] 
					{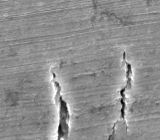}; 
					\addplot[opacity=0.5]  
					graphics[xmin=0, xmax=160, ymin=0, ymax=140] 
					{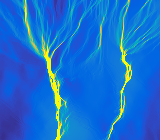}; 
					\end{axis} 
					\end{tikzpicture}\hspace{0.3cm}
				} 
			\end{tabular}
			\caption{Results of the TGV and IC models showing the displacement $u_1$ in \textmu m ($1$\,\textmu m $= 4$\,px) and the strain $\nabla_x u_1$.}
		\end{subfigure}
		\caption{ \label{fig:real_tgv_ic}
			Comparison of IC and TGV models for a real-world example 
			where the cracks widen up but the surrounding aluminum matrix deforms smoothly.} 
	\end{figure}	
	
	In our next example, we are interested in the detection of different kind of cracks.
	\textbf{Figure \ref{fig:real_tgv}} shows
	other real-world data arising from a tensile test.
	Cracks correspond to peaks in $\nabla_x u_1$.
	In the enlarged regions, three different types of cracks are shown, 
	namely 
	\begin{enumerate}
		\item
		a crack due to decohesion of a particle from the surrounding matrix, 
		\item
		a crack in the aluminum matrix, and 
		\item
		a crack inside a particle.
	\end{enumerate}
	
	\begin{figure}
		\centering
		\begin{subfigure}[b]{0.99\textwidth}\centering
			\begin{tabular}{cc}
				$\sigma_0 = 0$\,MPa & $\sigma_1 = 700$\,MPa \hspace{0.07cm} \\
				\begin{tikzpicture}\tikzstyle{every node}=[font=\tiny] 
				\begin{axis}[ 
				width=.52\textwidth, 
				enlargelimits=false, 
				hide axis, 
				axis equal image, 
				axis on top
				] 
				\addplot 
				graphics[xmin=0, xmax=750, ymin=0, ymax=600] 
				{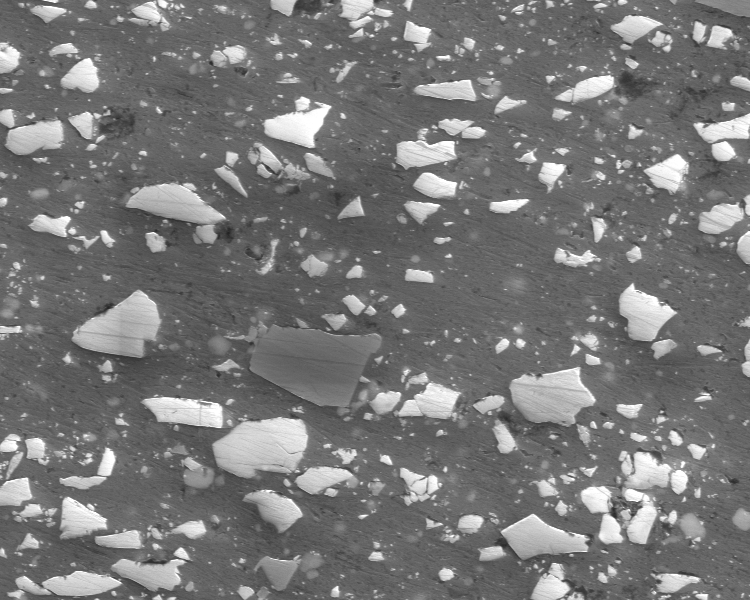};  
				\draw[color=black,thick] (216,212) rectangle (295,291);
				\draw[color=black,thick] (226,305) rectangle (305,384);
				\draw[color=black,thick] (159,82) rectangle (238,161);
				\draw[draw=none,fill=white] (554,5) rectangle (744,80);
				\draw[{|[width=4pt]}-{|[width=4pt]},thick] (569,30) -- (729,30) node[above=-2pt,pos=0.5] {10\,\textmu m};
				\end{axis} 
				\end{tikzpicture}
				& 
				\hspace{0.58cm}
				\begin{tikzpicture}\tikzstyle{every node}=[font=\tiny] 
				\begin{axis}[ 
				width=.52\textwidth, 
				enlargelimits=false, 
				hide axis, 
				axis equal image, 
				axis on top
				] 
				\addplot 
				graphics[xmin=0, xmax=750, ymin=0, ymax=600] 
				{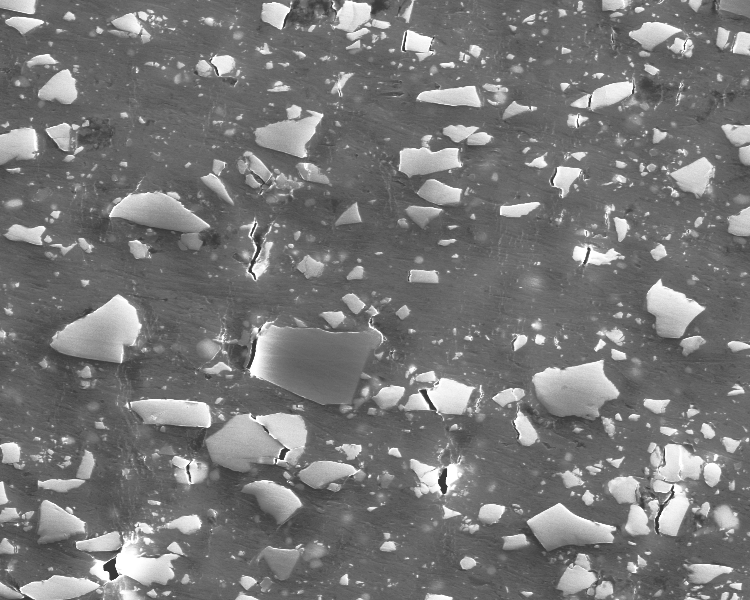};  
				\draw[color=black,thick] (216,212) rectangle (295,291);
				\draw[color=black,thick] (226,305) rectangle (305,384);
				\draw[color=black,thick] (159,82) rectangle (238,161);
				\draw[draw=none,fill=white] (554,5) rectangle (744,80);
				\draw[{|[width=4pt]}-{|[width=4pt]},thick] (569,30) -- (729,30) node[above=-2pt,pos=0.5] {10\,\textmu m};
				\end{axis} 
				\end{tikzpicture}\hspace{0.72cm}
				\\
				\begin{tikzpicture}\tikzstyle{every node}=[font=\tiny] 
				\begin{axis}[ 
				width=.145\textwidth, 
				enlargelimits=false, 
				xlabel={x}, 
				ylabel={y}, 
				scale only axis, 
				hide axis, 
				axis equal image, 
				axis on top, 
				scaled ticks=false, 
				] 
				\addplot[surf,point meta min=-5.1702, point meta max=10.5975]  
				graphics[xmin=0, xmax=80, ymin=0, ymax=80] 
				{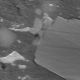}; 
				\draw[draw=none,fill=white] (44.5,1.5) rectangle (78.5,24);
				\draw[{|[width=4pt]}-{|[width=4pt]},thick] (45.5,9) -- (77.5,9) node[above=-2pt,pos=0.5] {2\,\textmu m};
				\end{axis} 
				\end{tikzpicture}
				\begin{tikzpicture}\tikzstyle{every node}=[font=\tiny] 
				\begin{axis}[ 
				width=.145\textwidth, 
				enlargelimits=false, 
				xlabel={x}, 
				ylabel={y}, 
				scale only axis, 
				hide axis, 
				axis equal image, 
				axis on top,  
				scaled ticks=false, 
				] 
				\addplot[surf,point meta min=-5.1702, point meta max=10.5975]  
				graphics[xmin=0, xmax=80, ymin=0, ymax=80] 
				{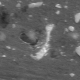}; 
				\draw[draw=none,fill=white] (44.5,1.5) rectangle (78.5,24);
				\draw[{|[width=4pt]}-{|[width=4pt]},thick] (45.5,9) -- (77.5,9) node[above=-2pt,pos=0.5] {2\,\textmu m};
				\end{axis} 
				\end{tikzpicture}
				\begin{tikzpicture}\tikzstyle{every node}=[font=\tiny] 
				\begin{axis}[ 
				width=.145\textwidth, 
				enlargelimits=false, 
				xlabel={x}, 
				ylabel={y}, 
				scale only axis, 
				hide axis, 
				axis equal image, 
				axis on top
				] 
				\addplot[point meta min=-0.01,
				point meta max=0.344] 
				graphics[xmin=0, xmax=80, ymin=0, ymax=80] 
				{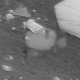};
				\draw[draw=none,fill=white] (44.5,1.5) rectangle (78.5,24);
				\draw[{|[width=4pt]}-{|[width=4pt]},thick] (45.5,9) -- (77.5,9) node[above=-2pt,pos=0.5] {2\,\textmu m};
				\end{axis} 
				\end{tikzpicture}
				&
				\begin{tikzpicture}\tikzstyle{every node}=[font=\tiny] 
				\begin{axis}[ 
				width=.145\textwidth, 
				enlargelimits=false, 
				xlabel={x}, 
				ylabel={y}, 
				scale only axis, 
				hide axis, 
				axis equal image, 
				axis on top, 
				scaled ticks=false, 
				]  
				\addplot[surf,point meta min=-5.1702, point meta max=10.5975]  
				graphics[xmin=0, xmax=80, ymin=0, ymax=80] 
				{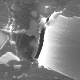}; 
				\draw[draw=none,fill=white] (44.5,1.5) rectangle (78.5,24);
				\draw[{|[width=4pt]}-{|[width=4pt]},thick] (45.5,9) -- (77.5,9) node[above=-2pt,pos=0.5] {2\,\textmu m};
				\end{axis} 
				\end{tikzpicture}
				\begin{tikzpicture}\tikzstyle{every node}=[font=\tiny] 
				\begin{axis}[ 
				width=.145\textwidth, 
				enlargelimits=false, 
				xlabel={x}, 
				ylabel={y}, 
				scale only axis, 
				hide axis, 
				axis equal image, 
				axis on top,  
				scaled ticks=false, 
				]  
				\addplot[surf,point meta min=-5.1702, point meta max=10.5975]  
				graphics[xmin=0, xmax=80, ymin=0, ymax=80] 
				{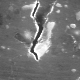}; 
				\draw[draw=none,fill=white] (44.5,1.5) rectangle (78.5,24);
				\draw[{|[width=4pt]}-{|[width=4pt]},thick] (45.5,9) -- (77.5,9) node[above=-2pt,pos=0.5] {2\,\textmu m};
				\end{axis} 
				\end{tikzpicture}
				\begin{tikzpicture}\tikzstyle{every node}=[font=\tiny] 
				\begin{axis}[ 
				width=.145\textwidth, 
				enlargelimits=false, 
				xlabel={x}, 
				ylabel={y}, 
				scale only axis, 
				hide axis, 
				axis equal image, 
				axis on top
				] 
				\addplot[point meta min=-0.01,
				point meta max=0.344] 
				graphics[xmin=0, xmax=80, ymin=0, ymax=80] 
				{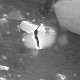}; 
				\draw[draw=none,fill=white] (44.5,1.5) rectangle (78.5,24);
				\draw[{|[width=4pt]}-{|[width=4pt]},thick] (45.5,9) -- (77.5,9) node[above=-2pt,pos=0.5] {2\,\textmu m};
				\end{axis} 
				\end{tikzpicture}
			\end{tabular} 
			\caption{Microstructure images of the specimen under low and high load.}
		\end{subfigure}
		\\[2ex]
		\begin{subfigure}[b]{0.99\textwidth}\centering
			\begin{tabular}{cc}
				$u_1$ & $\nabla_x u_1$	\\
				\begin{tikzpicture}\tikzstyle{every node}=[font=\tiny] 
				\begin{axis}[ 
				width=.52\textwidth, 
				enlargelimits=false, 
				hide axis, 
				axis equal image, 
				axis on top, 
				colorbar, 
				scaled ticks=false, 
				colormap name=parula, 
				colorbar right, 
				colorbar style={ 
					width=0.15cm,
					overlay
				}
				] 
				\addplot[point meta min=-1.71,
				point meta max=2.2]
				graphics[xmin=0, xmax=750, ymin=0, ymax=600] 
				{./images/strain/results_Serie5/image_1.png}; 
				\addplot[opacity=0.5]  
				graphics[xmin=0, xmax=750, ymin=0, ymax=600] 
				{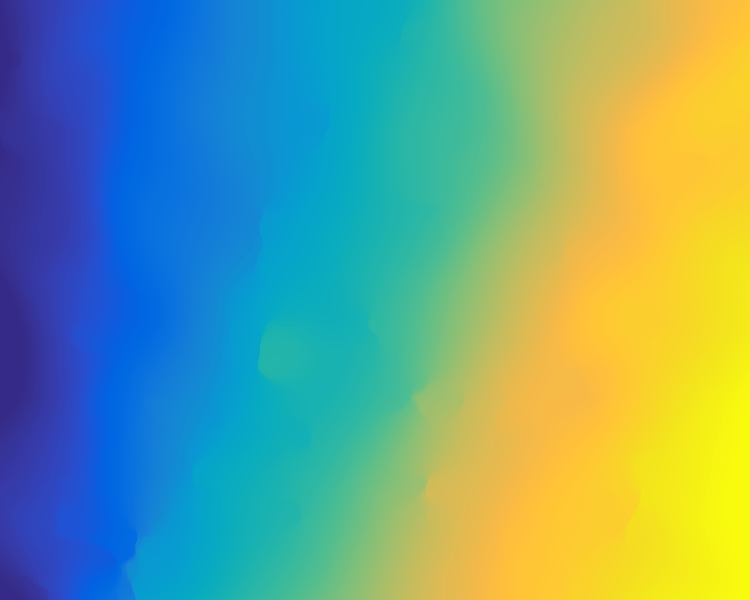}; 
				\end{axis} 
				\end{tikzpicture}
				&
				\hspace{.3cm}
				\begin{tikzpicture}\tikzstyle{every node}=[font=\tiny] 
				\begin{axis}[ 
				width=.52\textwidth, 
				enlargelimits=false, 
				hide axis, 
				axis equal image, 
				axis on top, 
				colorbar, 
				scaled ticks=false, 
				colormap name=parula, 
				colorbar right, 
				colorbar style={ 
					width=0.15cm,
					overlay
				}
				] 
				\addplot[point meta min=-0.01,
				point meta max=0.344]
				graphics[xmin=0, xmax=750, ymin=0, ymax=600] 
				{./images/strain/results_Serie5/image_1.png}; 
				\addplot[opacity=0.5]  
				graphics[xmin=0, xmax=750, ymin=0, ymax=600] 
				{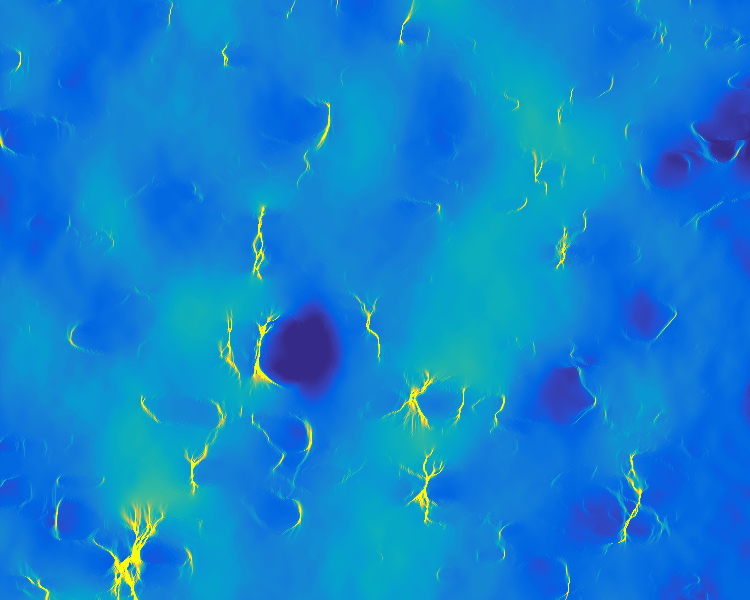}; 
				\end{axis} 
				\end{tikzpicture} 
				\hspace{.3cm}
				\\ 
				&
				\begin{tikzpicture}\tikzstyle{every node}=[font=\tiny] 
				\begin{axis}[ 
				width=.145\textwidth, 
				enlargelimits=false, 
				xlabel={x}, 
				ylabel={y}, 
				scale only axis, 
				hide axis, 
				axis equal image, 
				axis on top, 
				scaled ticks=false, 
				] 
				\addplot[surf,point meta min=-5.1702, point meta max=10.5975]  
				graphics[xmin=0, xmax=80, ymin=0, ymax=80] 
				{./images/strain/results_Serie5/regiona_im0.png}; 
				\addplot[surf,point meta min=-5.1702, point meta max=10.5975,opacity=0.5]  
				graphics[xmin=0, xmax=80, ymin=0, ymax=80] 
				{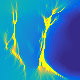}; 
				\end{axis} 
				\end{tikzpicture}
				\begin{tikzpicture}\tikzstyle{every node}=[font=\tiny] 
				\begin{axis}[ 
				width=.145\textwidth, 
				enlargelimits=false, 
				xlabel={x}, 
				ylabel={y}, 
				scale only axis, 
				hide axis, 
				axis equal image, 
				axis on top,  
				scaled ticks=false, 
				] 
				\addplot[surf,point meta min=-5.1702, point meta max=10.5975]  
				graphics[xmin=0, xmax=80, ymin=0, ymax=80] 
				{./images/strain/results_Serie5/regionb_im0.png}; 
				\addplot[surf,point meta min=-5.1702, point meta max=10.5975,opacity=0.5]  
				graphics[xmin=0, xmax=80, ymin=0, ymax=80] 
				{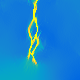}; 
				\end{axis} 
				\end{tikzpicture}
				\begin{tikzpicture}\tikzstyle{every node}=[font=\tiny] 
				\begin{axis}[ 
				width=.145\textwidth, 
				enlargelimits=false, 
				xlabel={x}, 
				ylabel={y}, 
				scale only axis, 
				hide axis, 
				axis equal image, 
				axis on top
				] 
				\addplot[point meta min=-0.01,
				point meta max=0.344] 
				graphics[xmin=0, xmax=80, ymin=0, ymax=80] 
				{./images/strain/results_Serie5/regionc_im0.png}; 
				\addplot[opacity=0.5]  
				graphics[xmin=0, xmax=80, ymin=0, ymax=80] 
				{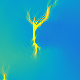}; 
				\end{axis} 
				\end{tikzpicture}\hspace{0.3cm}
			\end{tabular}
			\caption{Results of the TGV model showing the displacement $u_1$ in \textmu m ($1$\,\textmu m $= 16$\,px) and the strain $\nabla_x u_1$.}
		\end{subfigure}
		\caption{Detection of different kind of cracks by the TGV model.
		} \label{fig:real_tgv}
	\end{figure}
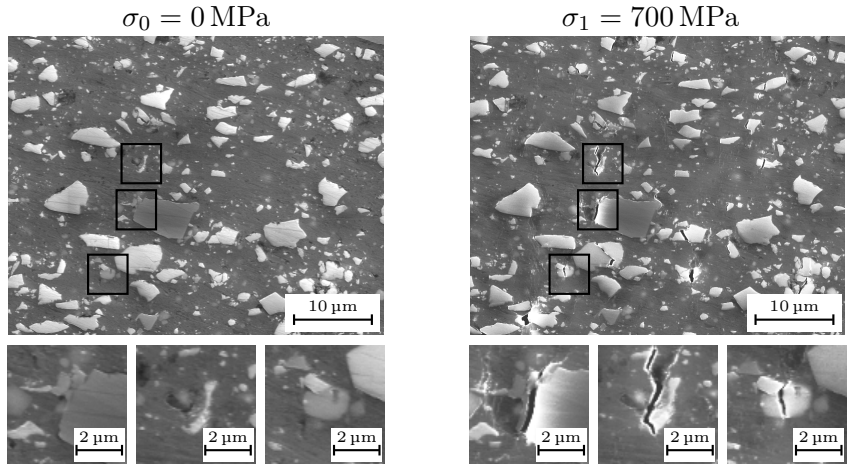
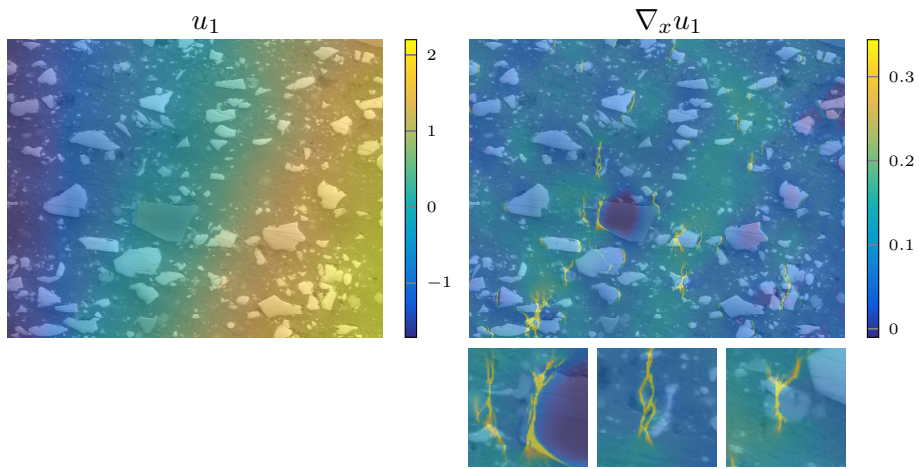
	
	Next, we are interested in crack propagation.	
	\textbf{Figure~\ref{fig:enlarge_tgv}} shows the displacement $u_1$ and the strain $\varepsilon_{11} = \nabla_x u_1$ for certain regions under different load.
	It is remarkable that even under low load, when the cracks are not or hardly visible in the images, 
	the strain in the corresponding regions is large. 
	Thus, it is a sensitive and useful tool to study crack initiation mechanisms.
	\begin{figure}
		\centering
		\begin{subfigure}[b]{0.99\textwidth}\centering
			\begin{tikzpicture}\tikzstyle{every node}=[font=\tiny] 
			\begin{axis}[ 
			width=.367\textwidth, 
			enlargelimits=false, 
			xlabel={x}, 
			ylabel={y}, 
			scale only axis, 
			hide axis, 
			axis equal image, 
			axis on top, 
			scaled ticks=false, 
			] 
			\addplot[point meta min=0, point meta max=0.1]
			graphics[xmin=0, xmax=800, ymin=0, ymax=600] 
			{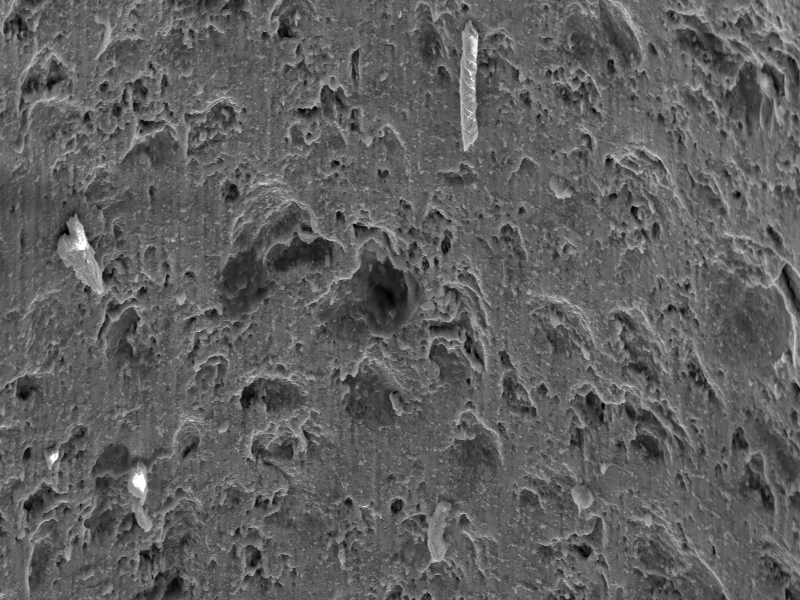};
			\draw[color=black,thick] (403,312) rectangle (482,391);
			\draw[color=black,thick] (430,221) rectangle (509,300);
			\draw[draw=none,fill=white] (564,5) rectangle (794,80);
			\draw[{|[width=4pt]}-{|[width=4pt]},thick] (579,30) -- (779,30) node[above=-2pt,pos=0.5] {100\,\textmu m};
			\end{axis} 
			\end{tikzpicture}
			\begin{tikzpicture}\tikzstyle{every node}=[font=\tiny] 
			\begin{axis}[ 
			width=.32\textwidth, 
			enlargelimits=false, 
			xlabel={x}, 
			ylabel={y}, 
			scale only axis, 
			hide axis, 
			axis equal image, 
			axis on top, 
			scaled ticks=false, 
			] 
			\addplot[surf,point meta min=-5.1702, point meta max=10.5975]  
			graphics[xmin=0, xmax=80, ymin=0, ymax=80] 
			{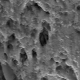}; 
			\addplot[surf,point meta min=-5.1702, point meta max=10.5975]  
			graphics[xmin=0, xmax=80, ymin=-90, ymax=-10] 
			{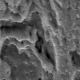};   
			\draw[draw=none,fill=white] (36.5,1.5) rectangle (78.5,24);
			\draw[{|[width=4pt]}-{|[width=4pt]},thick] (37.5,9) -- (77.5,9) node[above=-2pt,pos=0.5] {20\,\textmu m};  
			\draw[draw=none,fill=white] (36.5,-88.5) rectangle (78.5,-66);
			\draw[{|[width=4pt]}-{|[width=4pt]},thick] (37.5,-81) -- (77.5,-81) node[above=-2pt,pos=0.5] {20\,\textmu m};
			\end{axis} 
			\end{tikzpicture}
			\includegraphics[width=0.45\textwidth]{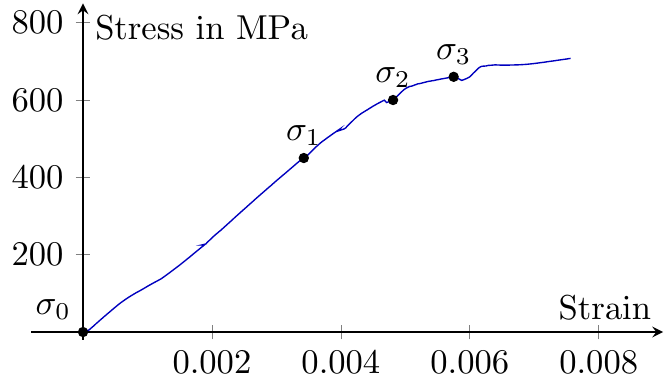}
			\caption{Microstructure images without load ($\sigma_0 = 0$\,MPa) and stress-strain curve.}
		\end{subfigure}
		\\[2ex]
		\begin{subfigure}[b]{0.99\textwidth}\centering
			\begin{tabular}{cccc}
				&
				$\sigma_1 = 450$\,MPa &
				$\sigma_2 = 600$\,MPa &
				$\sigma_3 = 660$\,MPa \qquad \   \\
				&
				\hspace{0.05cm}
				\begin{tikzpicture}\tikzstyle{every node}=[font=\tiny] 
				\begin{axis}[ 
				width=.22\textwidth, 
				enlargelimits=false, 
				xlabel={x}, 
				ylabel={y}, 
				scale only axis, 
				hide axis, 
				axis equal image, 
				axis on top, 
				scaled ticks=false, 
				] 
				\addplot[point meta min=0, point meta max=0.1]
				graphics[xmin=0, xmax=800, ymin=0, ymax=600] 
				{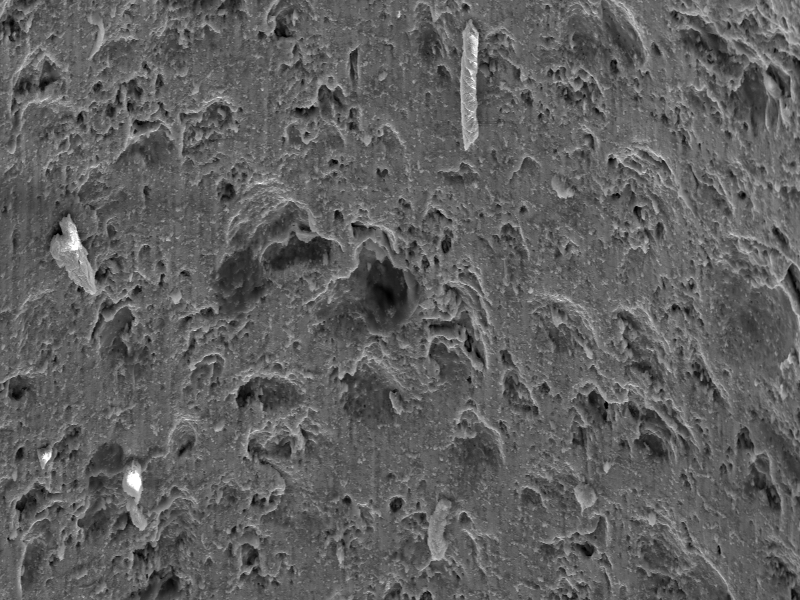};
				\end{axis} 
				\end{tikzpicture}
				&
				\hspace{0.05cm}
				\begin{tikzpicture}\tikzstyle{every node}=[font=\tiny] 
				\begin{axis}[ 
				width=.22\textwidth, 
				enlargelimits=false, 
				xlabel={x}, 
				ylabel={y}, 
				scale only axis, 
				hide axis, 
				axis equal image, 
				axis on top, 
				scaled ticks=false, 
				] 
				\addplot[point meta min=0, point meta max=0.1]
				graphics[xmin=0, xmax=800, ymin=0, ymax=600] 
				{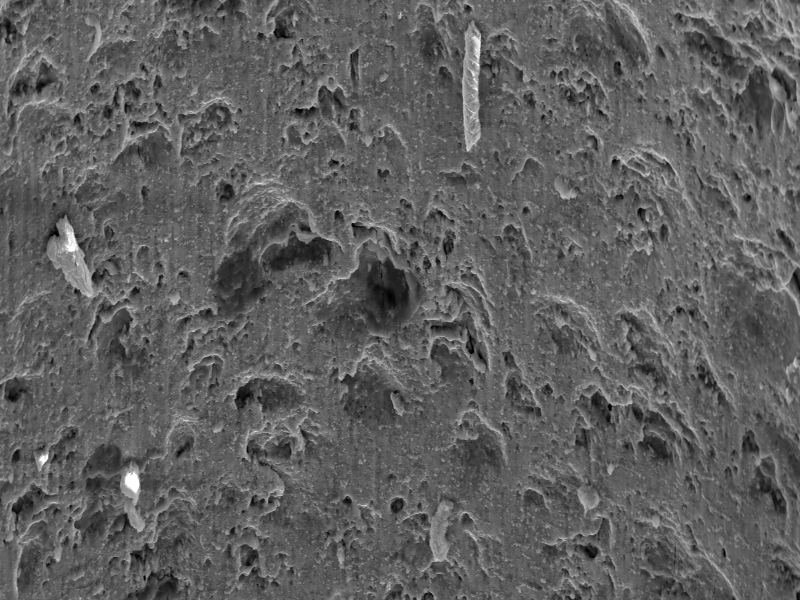};
				\end{axis} 
				\end{tikzpicture}
				&
				\begin{tikzpicture}\tikzstyle{every node}=[font=\tiny] 
				\begin{axis}[ 
				width=.22\textwidth, 
				enlargelimits=false, 
				xlabel={x}, 
				ylabel={y}, 
				scale only axis, 
				hide axis, 
				axis equal image, 
				axis on top, 
				scaled ticks=false, 
				] 
				\addplot[point meta min=0, point meta max=0.1]
				graphics[xmin=0, xmax=800, ymin=0, ymax=600] 
				{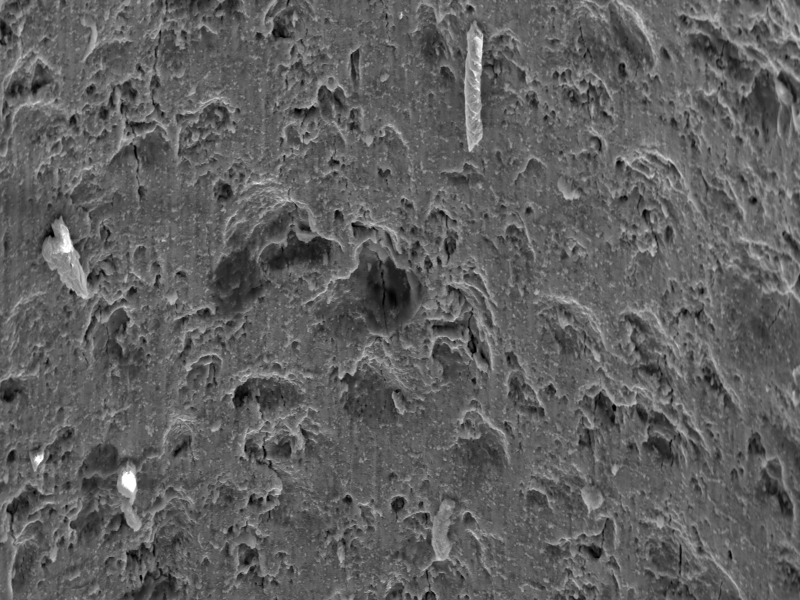};
				\end{axis} 
				\end{tikzpicture}\hspace{0.72cm}
				\\[1ex]
				$u_1$ &
				\raisebox{-.5\height}{
					\begin{tikzpicture}\tikzstyle{every node}=[font=\tiny] 
					\begin{axis}[ 
					width=.22\textwidth, 
					enlargelimits=false, 
					xlabel={x}, 
					ylabel={y}, 
					scale only axis, 
					hide axis, 
					axis equal image, 
					axis on top, 
					scaled ticks=false, 
					] 
					\addplot
					graphics[xmin=0, xmax=800, ymin=0, ymax=600] 
					{./images/strain/results_Serie59/image_1.png}; 
					\addplot[opacity=0.5]  
					graphics[xmin=0, xmax=800, ymin=0, ymax=600] 
					{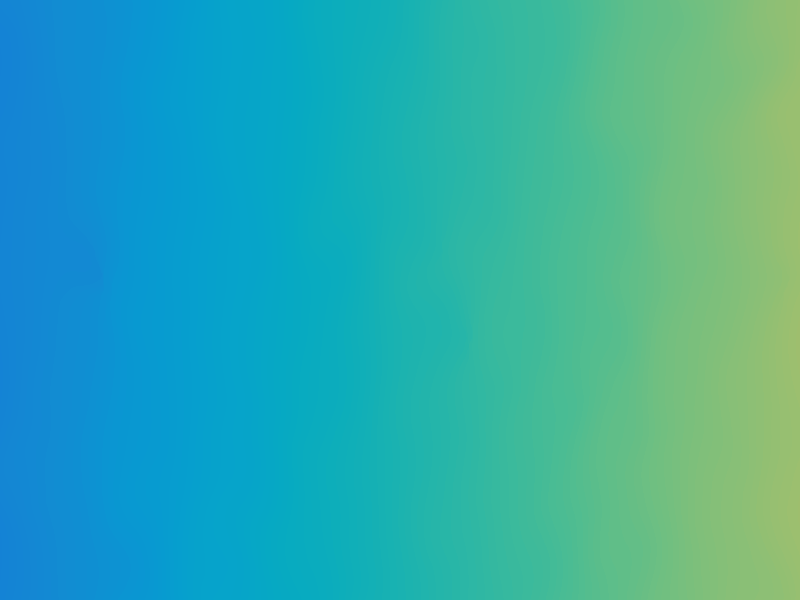};  
					\end{axis} 
					\end{tikzpicture}}
				&
				\raisebox{-.5\height}{
					\begin{tikzpicture}\tikzstyle{every node}=[font=\tiny] 
					\begin{axis}[ 
					width=.22\textwidth, 
					enlargelimits=false, 
					xlabel={x}, 
					ylabel={y}, 
					scale only axis, 
					hide axis, 
					axis equal image, 
					axis on top, 
					scaled ticks=false, 
					] 
					\addplot
					graphics[xmin=0, xmax=800, ymin=0, ymax=600] 
					{./images/strain/results_Serie59/image_1.png}; 
					\addplot[opacity=0.5]  
					graphics[xmin=0, xmax=800, ymin=0, ymax=600] 
					{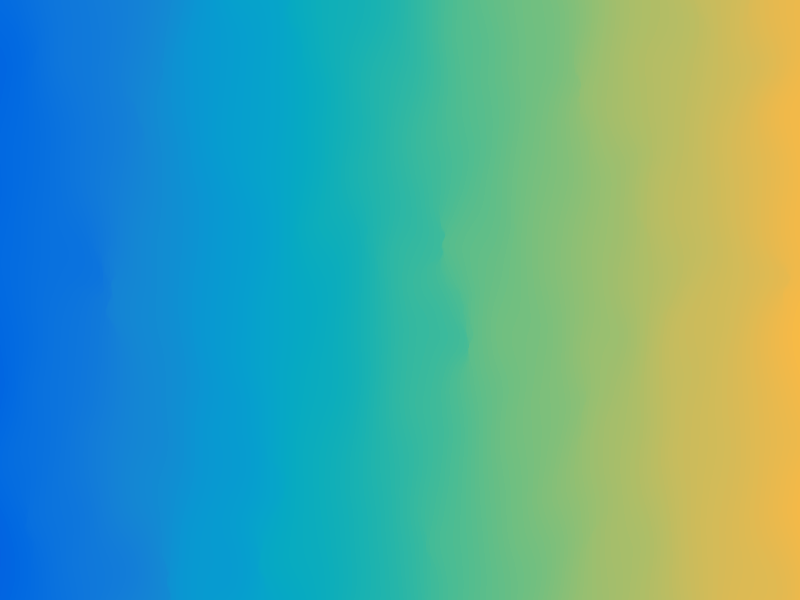};  
					\end{axis} 
					\end{tikzpicture}}
				&
				\raisebox{-.5\height}{
					\begin{tikzpicture}\tikzstyle{every node}=[font=\tiny] 
					\begin{axis}[ 
					width=.22\textwidth, 
					enlargelimits=false, 
					xlabel={x}, 
					ylabel={y}, 
					scale only axis, 
					hide axis, 
					axis equal image, 
					axis on top, 
					colorbar, 
					scaled ticks=false, 
					colormap name=parula, 
					colorbar right, 
					colorbar style={ 
						width=0.15cm,
						overlay
					},
					] 
					\addplot[point meta min=-8.57,point meta max=9.46]
					graphics[xmin=0, xmax=800, ymin=0, ymax=600] 
					{./images/strain/results_Serie59/image_1.png}; 
					\addplot[opacity=0.5]  
					graphics[xmin=0, xmax=800, ymin=0, ymax=600] 
					{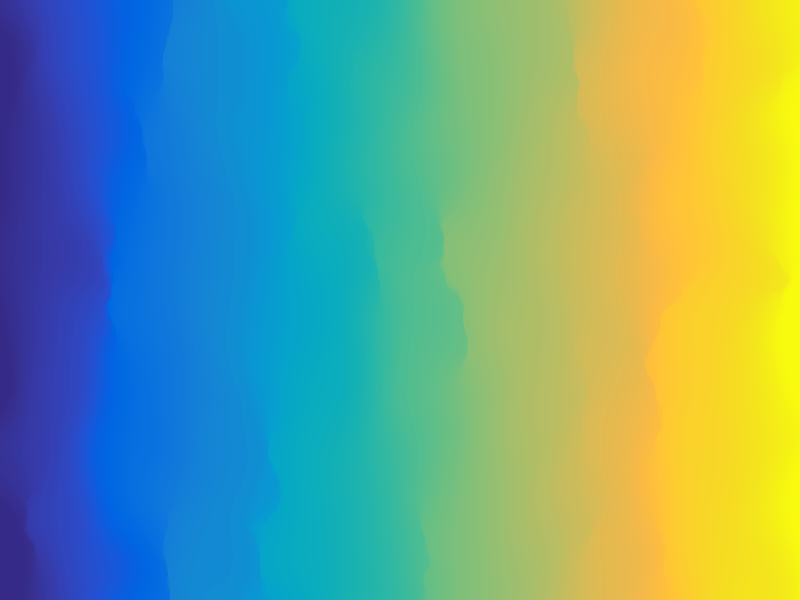};   
					\end{axis} 
					\end{tikzpicture}\hspace{0.6cm}}
				\\[7ex]
				$\nabla_x u_1$ &
				\raisebox{-.5\height}{
					\begin{tikzpicture}\tikzstyle{every node}=[font=\tiny] 
					\begin{axis}[ 
					width=.22\textwidth, 
					enlargelimits=false, 
					xlabel={x}, 
					ylabel={y}, 
					scale only axis, 
					hide axis, 
					axis equal image, 
					axis on top, 
					scaled ticks=false, 
					] 
					\addplot
					graphics[xmin=0, xmax=800, ymin=0, ymax=600] 
					{./images/strain/results_Serie59/image_1.png}; 
					\addplot[opacity=0.5]  
					graphics[xmin=0, xmax=800, ymin=0, ymax=600] 
					{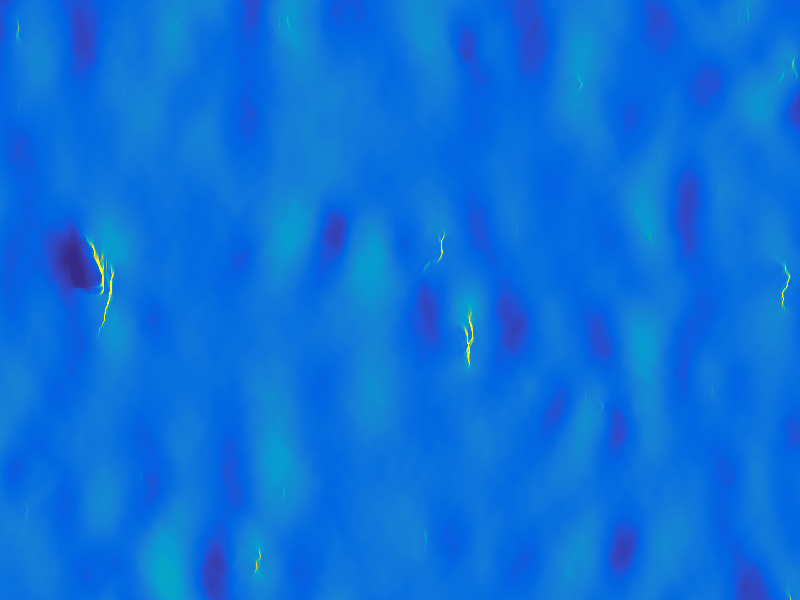};  
					\end{axis} 
					\end{tikzpicture}}
				&
				\raisebox{-.5\height}{
					\begin{tikzpicture}\tikzstyle{every node}=[font=\tiny] 
					\begin{axis}[ 
					width=.22\textwidth, 
					enlargelimits=false, 
					xlabel={x}, 
					ylabel={y}, 
					scale only axis, 
					hide axis, 
					axis equal image, 
					axis on top, 
					scaled ticks=false, 
					] 
					\addplot
					graphics[xmin=0, xmax=800, ymin=0, ymax=600] 
					{./images/strain/results_Serie59/image_1.png}; 
					\addplot[opacity=0.5]  
					graphics[xmin=0, xmax=800, ymin=0, ymax=600] 
					{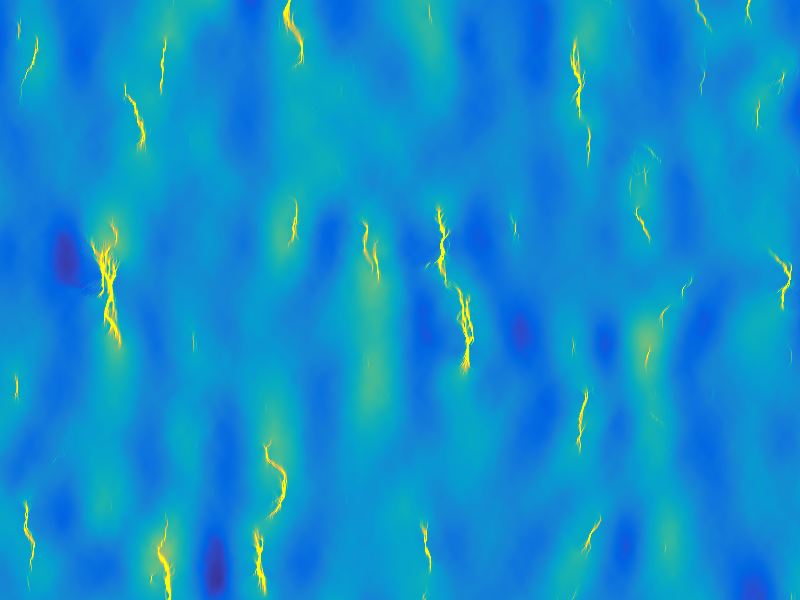};  
					\end{axis} 
					\end{tikzpicture}}
				&			
				\raisebox{-.5\height}{
					\begin{tikzpicture}\tikzstyle{every node}=[font=\tiny] 
					\begin{axis}[ 
					width=.22\textwidth, 
					enlargelimits=false, 
					xlabel={x}, 
					ylabel={y}, 
					scale only axis, 
					hide axis, 
					axis equal image, 
					axis on top, 
					colorbar, 
					scaled ticks=false, 
					colormap name=parula, 
					colorbar right, 
					colorbar style={ 
						width=0.15cm,
						overlay,
						yticklabel style = {
							/pgf/number format/.cd,
							fixed,
							fixed zerofill
							/tikz/.cd
						}
					},
					] 
					\addplot[point meta min=0,point meta max=0.1]
					graphics[xmin=0, xmax=800, ymin=0, ymax=600] 
					{./images/strain/results_Serie59/image_1.png}; 
					\addplot[opacity=0.5]  
					graphics[xmin=0, xmax=800, ymin=0, ymax=600] 
					{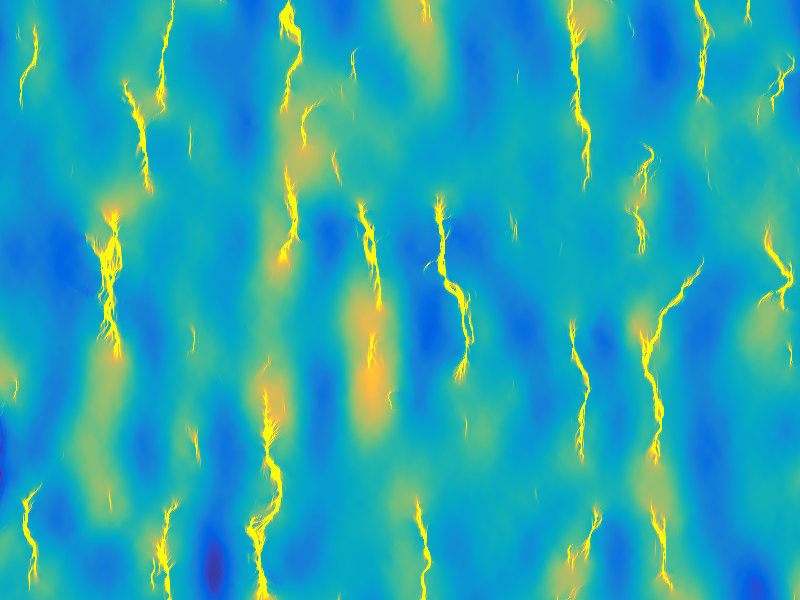};   
					\end{axis} 
					\end{tikzpicture}\hspace{0.6cm}}
			\end{tabular}
			\caption{Displacement $u_1$ in \textmu m ($1$\,\textmu m $= 2$\,px) and the strain $\nabla_x u_1$. The colorbar for the strain is cut off at $0.1$ to make smaller values under low load visible.}
		\end{subfigure}
		\\[2ex]
		\begin{subfigure}[b]{0.99\textwidth}\centering
			\begin{tabular}{ccc}
				$\sigma_1 = 450$\,MPa &
				$\sigma_2 = 600$\,MPa &
				$\sigma_3 = 660$\,MPa \\
				\begin{tikzpicture}\tikzstyle{every node}=[font=\tiny] 
				\begin{axis}[ 
				width=.150\textwidth, 
				enlargelimits=false, 
				xlabel={x}, 
				ylabel={y}, 
				scale only axis, 
				hide axis, 
				axis equal image, 
				axis on top, 
				scaled ticks=false, 
				] 
				\addplot[surf,point meta min=-5.1702, point meta max=10.5975]  
				graphics[xmin=-90, xmax=-10, ymin=0, ymax=80] 
				{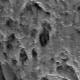};  
				\end{axis} 
				\end{tikzpicture}
				\begin{tikzpicture}\tikzstyle{every node}=[font=\tiny] 
				\begin{axis}[ 
				width=.150\textwidth, 
				enlargelimits=false, 
				xlabel={x}, 
				ylabel={y}, 
				scale only axis, 
				hide axis, 
				axis equal image, 
				axis on top, 
				scaled ticks=false, 
				] 
				\addplot[surf,point meta min=-5.1702, point meta max=10.5975]  
				graphics[xmin=0, xmax=80, ymin=0, ymax=80] 
				{./images/strain/results_Serie59/regiona_2_im0.png}; 
				\addplot[surf,point meta min=-5.1702, point meta max=10.5975,opacity=0.5]  
				graphics[xmin=0, xmax=80, ymin=0, ymax=80] 
				{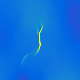}; 
				\end{axis} 
				\end{tikzpicture}
				&
				\begin{tikzpicture}\tikzstyle{every node}=[font=\tiny] 
				\begin{axis}[ 
				width=.150\textwidth, 
				enlargelimits=false, 
				xlabel={x}, 
				ylabel={y}, 
				scale only axis, 
				hide axis, 
				axis equal image, 
				axis on top, 
				scaled ticks=false, 
				] 
				\addplot[surf,point meta min=-5.1702, point meta max=10.5975]  
				graphics[xmin=-90, xmax=-10, ymin=0, ymax=80] 
				{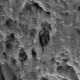}; 
				\end{axis} 
				\end{tikzpicture}
				\begin{tikzpicture}\tikzstyle{every node}=[font=\tiny] 
				\begin{axis}[ 
				width=.150\textwidth, 
				enlargelimits=false, 
				xlabel={x}, 
				ylabel={y}, 
				scale only axis, 
				hide axis, 
				axis equal image, 
				axis on top, 
				scaled ticks=false, 
				] 
				\addplot[surf,point meta min=-5.1702, point meta max=10.5975]  
				graphics[xmin=0, xmax=80, ymin=0, ymax=80] 
				{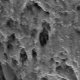}; 
				\addplot[surf,point meta min=-5.1702, point meta max=10.5975,opacity=0.5]  
				graphics[xmin=0, xmax=80, ymin=0, ymax=80] 
				{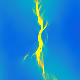}; 
				\end{axis} 
				\end{tikzpicture}
				&
				\begin{tikzpicture}\tikzstyle{every node}=[font=\tiny] 
				\begin{axis}[ 
				width=.150\textwidth, 
				enlargelimits=false, 
				xlabel={x}, 
				ylabel={y}, 
				scale only axis, 
				hide axis, 
				axis equal image, 
				axis on top, 
				scaled ticks=false, 
				] 
				\addplot[surf,point meta min=-5.1702, point meta max=10.5975]  
				graphics[xmin=-90, xmax=-10, ymin=0, ymax=80] 
				{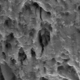}; 
				\end{axis} 
				\end{tikzpicture}
				\begin{tikzpicture}\tikzstyle{every node}=[font=\tiny] 
				\begin{axis}[ 
				width=.150\textwidth, 
				enlargelimits=false, 
				xlabel={x}, 
				ylabel={y}, 
				scale only axis, 
				hide axis, 
				axis equal image, 
				axis on top, 
				scaled ticks=false, 
				] 
				\addplot[surf,point meta min=-5.1702, point meta max=10.5975]  
				graphics[xmin=0, xmax=80, ymin=0, ymax=80] 
				{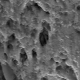}; 
				\addplot[surf,point meta min=-5.1702, point meta max=10.5975,opacity=0.5]  
				graphics[xmin=0, xmax=80, ymin=0, ymax=80] 
				{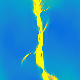}; 
				\end{axis} 
				\end{tikzpicture}
				\\
				\begin{tikzpicture}\tikzstyle{every node}=[font=\tiny] 
				\begin{axis}[ 
				width=.150\textwidth, 
				enlargelimits=false, 
				xlabel={x}, 
				ylabel={y}, 
				scale only axis, 
				hide axis, 
				axis equal image, 
				axis on top, 
				scaled ticks=false, 
				] 
				\addplot[surf,point meta min=-5.1702, point meta max=10.5975]  
				graphics[xmin=-90, xmax=-10, ymin=0, ymax=80] 
				{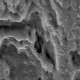}; 
				\end{axis} 
				\end{tikzpicture}
				\begin{tikzpicture}\tikzstyle{every node}=[font=\tiny] 
				\begin{axis}[ 
				width=.150\textwidth, 
				enlargelimits=false, 
				xlabel={x}, 
				ylabel={y}, 
				scale only axis, 
				hide axis, 
				axis equal image, 
				axis on top, 
				scaled ticks=false, 
				] 
				\addplot[surf,point meta min=-5.1702, point meta max=10.5975]  
				graphics[xmin=0, xmax=80, ymin=0, ymax=80] 
				{./images/strain/results_Serie59/regionb_2_im0.png}; 
				\addplot[surf,point meta min=-5.1702, point meta max=10.5975,opacity=0.5]  
				graphics[xmin=0, xmax=80, ymin=0, ymax=80] 
				{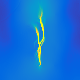}; 
				\end{axis} 
				\end{tikzpicture}
				&
				\begin{tikzpicture}\tikzstyle{every node}=[font=\tiny] 
				\begin{axis}[ 
				width=.150\textwidth, 
				enlargelimits=false, 
				xlabel={x}, 
				ylabel={y}, 
				scale only axis, 
				hide axis, 
				axis equal image, 
				axis on top, 
				scaled ticks=false, 
				] 
				\addplot[surf,point meta min=-5.1702, point meta max=10.5975]  
				graphics[xmin=-90, xmax=-10, ymin=0, ymax=80] 
				{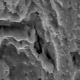}; 
				\end{axis} 
				\end{tikzpicture}
				\begin{tikzpicture}\tikzstyle{every node}=[font=\tiny] 
				\begin{axis}[ 
				width=.150\textwidth, 
				enlargelimits=false, 
				xlabel={x}, 
				ylabel={y}, 
				scale only axis, 
				hide axis, 
				axis equal image, 
				axis on top, 
				scaled ticks=false, 
				] 
				\addplot[surf,point meta min=-5.1702, point meta max=10.5975]  
				graphics[xmin=0, xmax=80, ymin=0, ymax=80] 
				{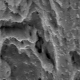}; 
				\addplot[surf,point meta min=-5.1702, point meta max=10.5975,opacity=0.5]  
				graphics[xmin=0, xmax=80, ymin=0, ymax=80] 
				{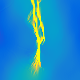}; 
				\end{axis} 
				\end{tikzpicture}
				&
				\begin{tikzpicture}\tikzstyle{every node}=[font=\tiny] 
				\begin{axis}[ 
				width=.150\textwidth, 
				enlargelimits=false, 
				xlabel={x}, 
				ylabel={y}, 
				scale only axis, 
				hide axis, 
				axis equal image, 
				axis on top, 
				scaled ticks=false, 
				] 
				\addplot[surf,point meta min=-5.1702, point meta max=10.5975]  
				graphics[xmin=-90, xmax=-10, ymin=0, ymax=80] 
				{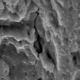}; 
				\end{axis} 
				\end{tikzpicture}
				\begin{tikzpicture}\tikzstyle{every node}=[font=\tiny] 
				\begin{axis}[ 
				width=.150\textwidth, 
				enlargelimits=false, 
				xlabel={x}, 
				ylabel={y}, 
				scale only axis, 
				hide axis, 
				axis equal image, 
				axis on top, 
				scaled ticks=false, 
				] 
				\addplot[surf,point meta min=-5.1702, point meta max=10.5975]  
				graphics[xmin=0, xmax=80, ymin=0, ymax=80] 
				{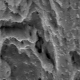}; 
				\addplot[surf,point meta min=-5.1702, point meta max=10.5975,opacity=0.5]  
				graphics[xmin=0, xmax=80, ymin=0, ymax=80] 
				{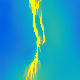}; 
				\end{axis} 
				\end{tikzpicture}
			\end{tabular}
			\caption{Two enlarged regions and the strain $\nabla_x u_1$ for increasing loads.
				The color map for the strain is cut off at $0.1$ to make smaller values  under low load visible.}
		\end{subfigure}
		\caption{Results on crack propagation by the TGV model.} \label{fig:enlarge_tgv}
	\end{figure}
	
	\paragraph{Comparison to Correlation Based Methods.}
	Now, we draw our attention to a comparison of the proposed method to correlation based methods.
	In the following, we use NCorr~\cite{BAA14} for a comparison 
	since it is a freely available software package.
	Note that other state-of-the-art software packages for strain analysis such as Veddac~\cite{Veddac} or VIC~\cite{VIC} are based on similar methods.
	In particular in \cite{HR14}, it is shown that NCorr produces equally good results as VIC.
	The underlying idea of correlation based methods is the comparison of windows around each pixel in certain search windows.
	The window sizes are parameters which must be appropriately chosen. 
	The size of the search window directly influences the computation time and needs to be chosen according to the maximal displacement.
	The smoothness of the results can be steered by the size of the window around the pixels, similar to the regularization parameters in our model.
	Due to the extent of this windows, the local resolution of correlation based methods is limited.
	\textbf{Figure~\ref{fig:tgv_ncorr}} shows three examples where the proposed TGV regularized method is compared to NCorr.
	The first example in Figure~\ref{fig:tgv_ncorr} shows the same specimen as Figure~\ref{fig:real_tgv_ic}.
	Although the cracks are relatively large, it is not possible to resolve them with NCorr and there occur some artifacts around them. 
	Thus, the low resolution is especially a drawback in our applications since we are interested in the local behavior and in particular cracks.
	Besides the cracks, the smooth parts look very similar to the result using the proposed method.
	For the cracks, $a_1$ is a good result.
	A peak is visible at the crack tip, which is the position where the crack is expected to propagate.
	In contrast, the result of NCorr shows a peak everywhere around the crack.
	
	\begin{figure} \centering		
		\begin{tabular}{clll}
			{\small image}\\ {\small under load} 
			&
			{\small \quad \ $\nabla_x u_1$}&
			{\small \qquad $a_{1}$}&
			{\small $\nabla_x u_1$ by NCorr}
			\\ 			
			\begin{tikzpicture}\tikzstyle{every node}=[font=\tiny]
			\begin{axis}[width=.36\textwidth,
			enlargelimits=false, 
			hide axis,
			axis equal image
			]
			\addplot
			graphics[xmin=0, xmax=600, ymin=0, ymax=800]{./images/strain/results_Serie58/image_2.png};
			\draw[color=black,thick] (111,111) rectangle (270,250);
			\draw[draw=none,fill=white] (364,5) rectangle (594,160);
			\draw[{|[width=4pt]}-{|[width=4pt]},thick] (379,50) -- (579,50) node[above=-2pt,pos=0.5] {50\,\textmu m};
			\end{axis}
			\end{tikzpicture} 			
			&
			\begin{tikzpicture}\tikzstyle{every node}=[font=\tiny]
			\begin{axis}[width=.36\textwidth,
			enlargelimits=false, 
			hide axis,
			axis equal image,
			colorbar,
			colormap name=parula,
			colorbar style={overlay, width=0.15cm, yticklabel style={
					/pgf/number format/.cd,
					fixed,
					fixed zerofill,
					/tikz/.cd}}
			]
			\addplot 
			graphics[xmin=0, xmax=6, ymin=0, ymax=8]{./images/strain/results_Serie58/image_1.png};
			\addplot[point meta min=-0.027, point meta max=0.251, opacity=0.5] 
			graphics[xmin=0, xmax=6, ymin=0, ymax=8]{./images/strain/results_Serie58/u1x_TGV.png};
			\end{axis}
			\end{tikzpicture} \hspace{0.5cm}
			&
			\begin{tikzpicture}\tikzstyle{every node}=[font=\tiny]
			\begin{axis}[width=.36\textwidth,
			enlargelimits=false, 
			hide axis,
			axis equal image,
			colorbar,
			colormap name=parula,
			colorbar style={overlay, width=0.15cm, yticklabel style={
					scaled ticks=false,
					/pgf/number format/.cd,
					fixed,
					fixed zerofill,
					/tikz/.cd}}
			]
			\addplot 
			graphics[xmin=0, xmax=6, ymin=0, ymax=8]{./images/strain/results_Serie58/image_1.png};
			\addplot[point meta min=-0.0054, point meta max=0.0847,opacity=0.5] 
			graphics[xmin=0, xmax=6, ymin=0, ymax=8]{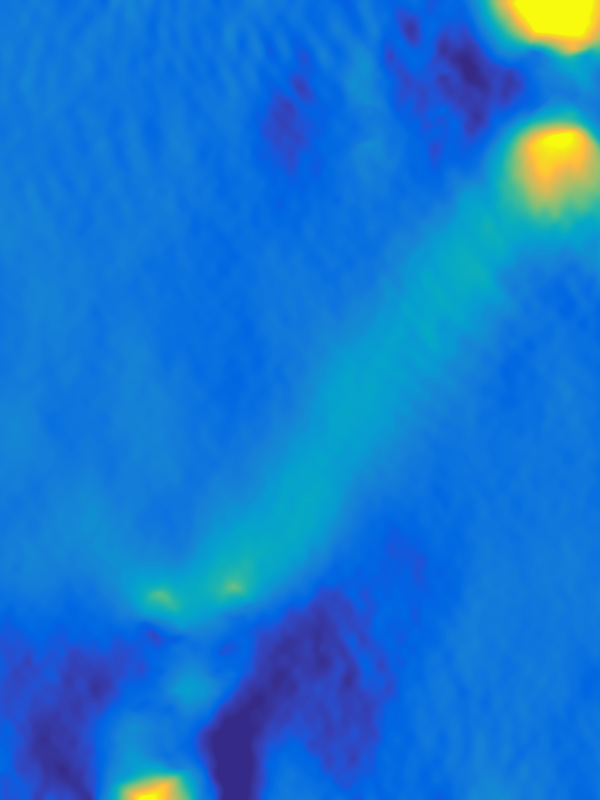};
			\end{axis}
			\end{tikzpicture} \hspace{0.5cm}
			&
			\begin{tikzpicture}\tikzstyle{every node}=[font=\tiny]
			\begin{axis}[width=.36\textwidth,
			enlargelimits=false, 
			hide axis,
			axis equal image,
			colorbar,
			scaled ticks=false,
			colormap name=parula,
			colorbar style={overlay, width=0.15cm, yticklabel style={
					scaled ticks=false,
					/pgf/number format/.cd,
					fixed,
					fixed zerofill,
					/tikz/.cd}}
			]
			\addplot 
			graphics[xmin=0, xmax=6, ymin=0, ymax=8]{./images/strain/results_Serie58/image_1.png};
			\addplot[point meta min=-0.0054, point meta max=0.0847,opacity=0.5] 
			graphics[xmin=0, xmax=6, ymin=0, ymax=8]{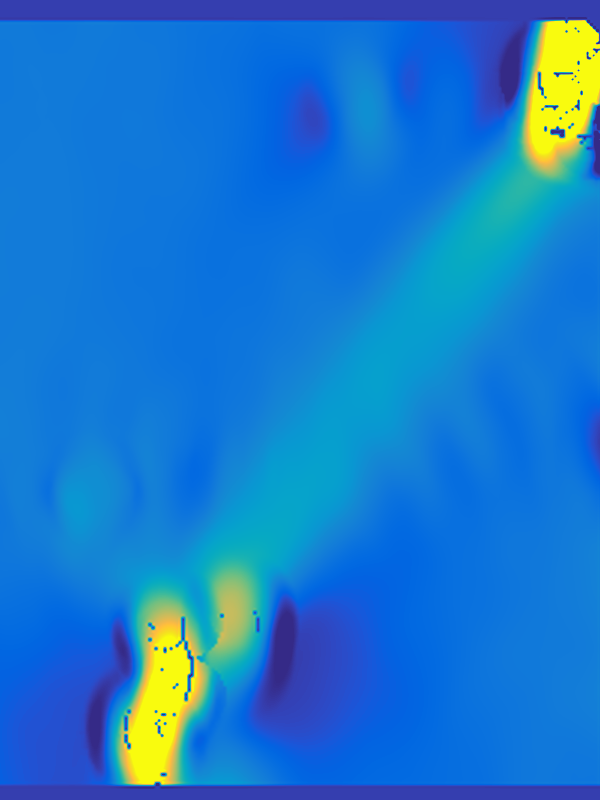};
			\end{axis}
			\end{tikzpicture} \hspace{0.5cm}
			\\ 			
			\begin{tikzpicture}\tikzstyle{every node}=[font=\tiny]
			\begin{axis}[width=.285\textwidth,
			enlargelimits=false, 
			hide axis,
			axis equal image
			]
			\addplot
			graphics[xmin=0, xmax=160, ymin=0, ymax=140]{./images/strain/results_Serie58/regionTGV_im1.png};
			\draw[draw=none,fill=white] (95,1) rectangle (159,40);
			\draw[{|[width=4pt]}-{|[width=4pt]},thick] (107,9) -- (147,9) node[above=-2pt,pos=0.5] {10\,\textmu m}; 
			\end{axis}
			\end{tikzpicture} 			
			&
			\begin{tikzpicture}\tikzstyle{every node}=[font=\tiny]
			\begin{axis}[width=.285\textwidth,
			enlargelimits=false, 
			hide axis,
			axis equal image
			]
			\addplot
			graphics[xmin=0, xmax=16, ymin=0, ymax=14]{./images/strain/results_Serie58/regionTGV_im1.png};
			\addplot[opacity=0.5]
			graphics[xmin=0, xmax=16, ymin=0, ymax=14]{./images/strain/results_Serie58/regionTGV_u.png};
			\end{axis}
			\end{tikzpicture}
			&
			\begin{tikzpicture}\tikzstyle{every node}=[font=\tiny]
			\begin{axis}[width=.285\textwidth,
			enlargelimits=false, 
			hide axis,
			axis equal image
			]
			\addplot
			graphics[xmin=0, xmax=16, ymin=0, ymax=14]{./images/strain/results_Serie58/regionTGV_im1.png};
			\addplot[opacity=0.5]
			graphics[xmin=0, xmax=16, ymin=0, ymax=14]{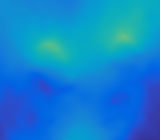};
			\end{axis}
			\end{tikzpicture} 
			&
			\begin{tikzpicture}\tikzstyle{every node}=[font=\tiny]
			\begin{axis}[width=.285\textwidth,
			enlargelimits=false, 
			hide axis,
			axis equal image
			]
			\addplot
			graphics[xmin=0, xmax=16, ymin=0, ymax=14]{./images/strain/results_Serie58/regionTGV_im1.png};
			\addplot[opacity=0.5]
			graphics[xmin=0, xmax=16, ymin=0, ymax=14]{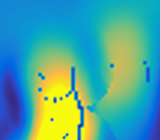};
			\end{axis}
			\end{tikzpicture} 
			\\[3ex] 			
			\begin{tikzpicture}\tikzstyle{every node}=[font=\tiny]
			\begin{axis}[width=.31\textwidth,
			enlargelimits=false, 
			hide axis,
			axis equal image
			]
			\addplot
			graphics[xmin=0, xmax=700, ymin=0, ymax=700]{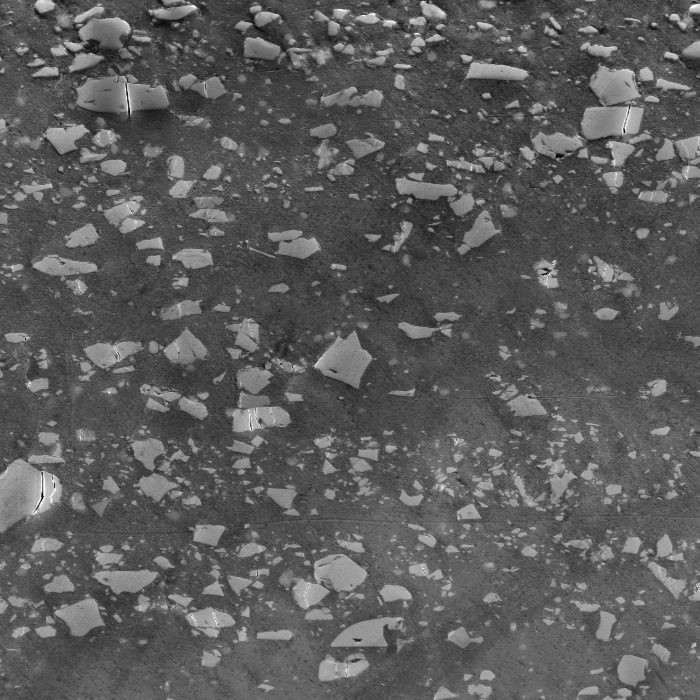};
			\draw[color=black,thick] (75,551) rectangle (174,650);
			\draw[draw=none,fill=white] (424,5) rectangle (694,160);
			\draw[{|[width=4pt]}-{|[width=4pt]},thick] (479,50) -- (639,50) node[above=-2pt,pos=0.5] {20\,\textmu m};
			\end{axis}
			\end{tikzpicture} 			
			&
			\begin{tikzpicture}\tikzstyle{every node}=[font=\tiny]
			\begin{axis}[width=.31\textwidth,
			enlargelimits=false, 
			hide axis,
			axis equal image,
			colorbar,
			colormap name=parula,
			colorbar style={overlay, width=0.15cm, yticklabel style={
					/pgf/number format/.cd,
					fixed,
					fixed zerofill,
					/tikz/.cd}}
			]
			\addplot 
			graphics[xmin=0, xmax=6, ymin=0, ymax=6]{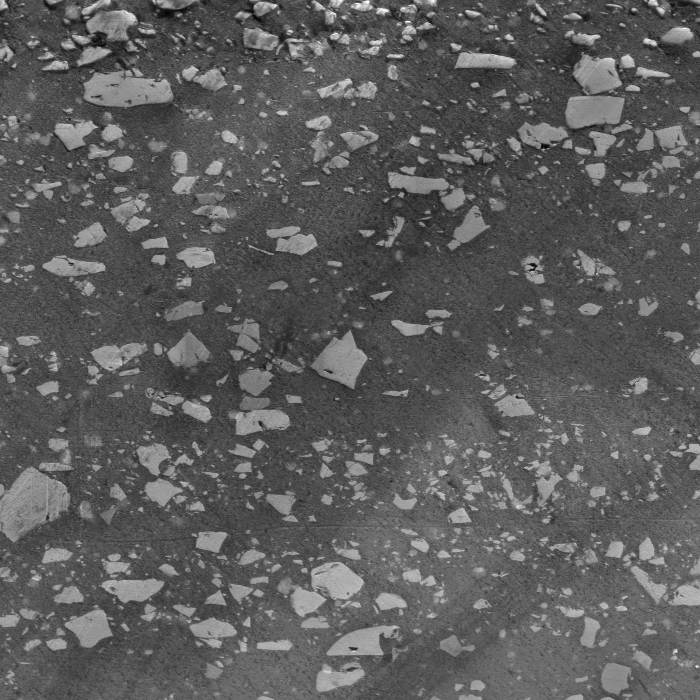};
			\addplot[point meta min=-0.015, point meta max=0.1, opacity=0.5] 
			graphics[xmin=0, xmax=6, ymin=0, ymax=6]{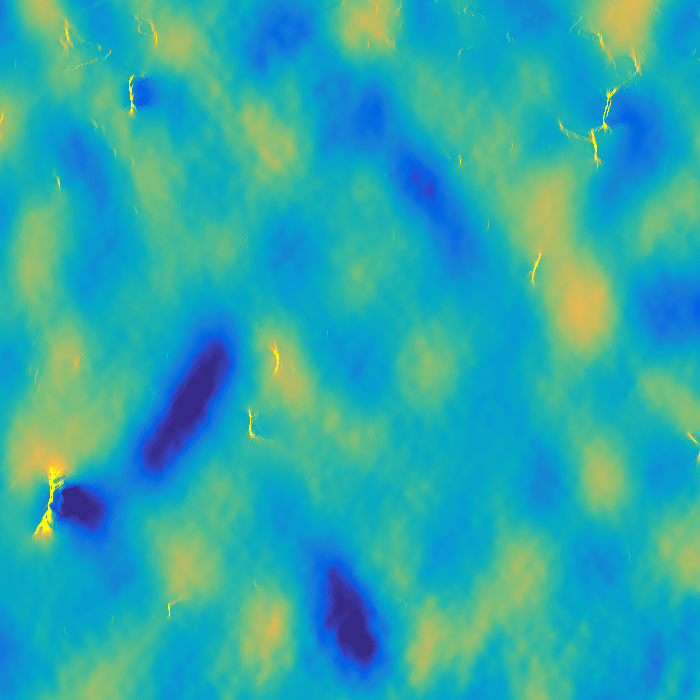};
			\end{axis}
			\end{tikzpicture} \hspace{0.5cm} 
			&
			\begin{tikzpicture}\tikzstyle{every node}=[font=\tiny]
			\begin{axis}[width=.31\textwidth,
			enlargelimits=false, 
			hide axis,
			axis equal image,
			colorbar,
			colormap name=parula,
			colorbar style={overlay, width=0.15cm, yticklabel style={
					scaled ticks=false,
					/pgf/number format/.cd,
					fixed,
					fixed zerofill,
					/tikz/.cd}}
			]
			\addplot 
			graphics[xmin=0, xmax=6, ymin=0, ymax=6]{./images/strain/results_Serie71/image_1.png};
			\addplot[point meta min=-0.01, point meta max=0.0756,opacity=0.5] 
			graphics[xmin=0, xmax=6, ymin=0, ymax=6]{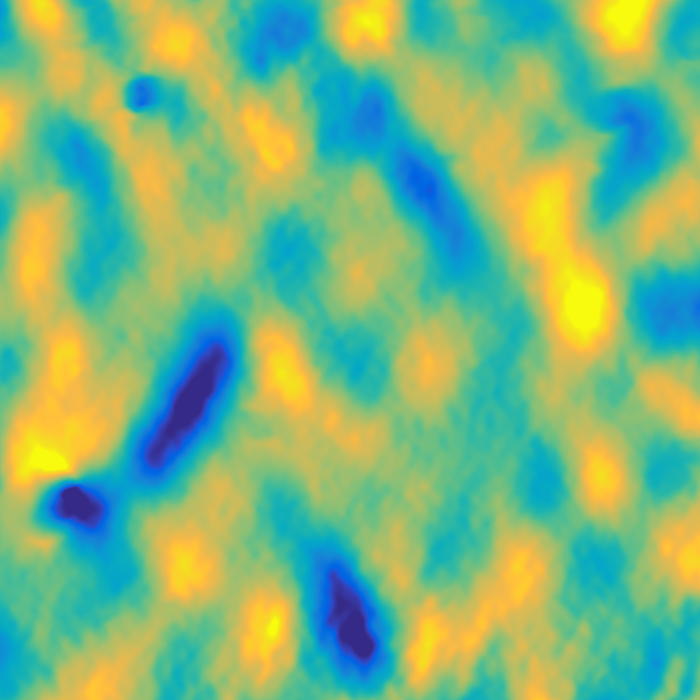};
			\end{axis}
			\end{tikzpicture} \hspace{0.5cm} 
			&
			\begin{tikzpicture}\tikzstyle{every node}=[font=\tiny]
			\begin{axis}[width=.31\textwidth,
			enlargelimits=false, 
			hide axis,
			axis equal image,
			colorbar,
			scaled ticks=false,
			colormap name=parula,
			colorbar style={overlay, width=0.15cm, yticklabel style={
					scaled ticks=false,
					/pgf/number format/.cd,
					fixed,
					fixed zerofill,
					/tikz/.cd}}
			]
			\addplot 
			graphics[xmin=0, xmax=6, ymin=0, ymax=6]{./images/strain/results_Serie71/image_1.png};
			\addplot[point meta min=-0.01, point meta max=0.0756,opacity=0.5] 
			graphics[xmin=0, xmax=6, ymin=0, ymax=6]{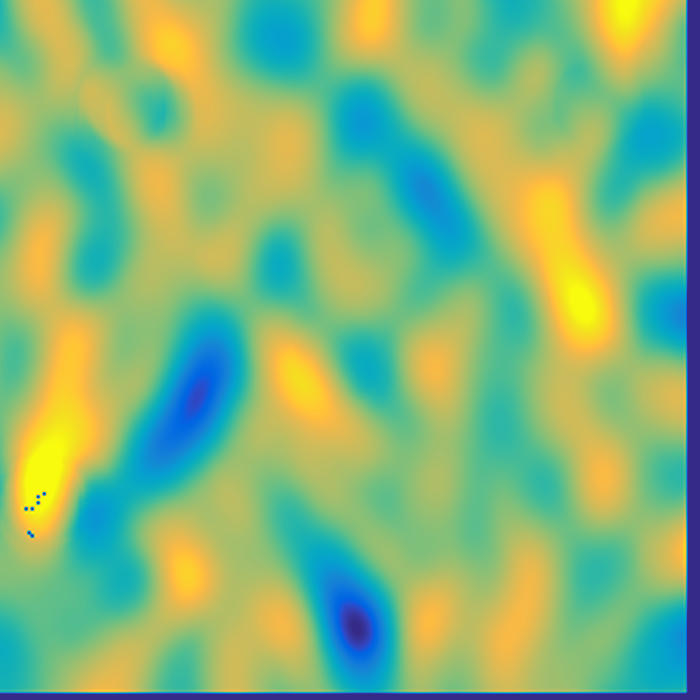};
			\end{axis}
			\end{tikzpicture} \hspace{0.5cm} 
			\\		
			\begin{tikzpicture}\tikzstyle{every node}=[font=\tiny]
			\begin{axis}[width=.31\textwidth,
			enlargelimits=false, 
			hide axis,
			axis equal image
			]
			\addplot
			graphics[xmin=0, xmax=100, ymin=0, ymax=100]{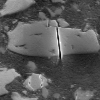};
			\draw[draw=none,fill=white] (51,1) rectangle (99,25);
			\draw[{|[width=4pt]}-{|[width=4pt]},thick] (55,9) -- (95,9) node[above=-2pt,pos=0.5] {5\,\textmu m}; 
			\end{axis}
			\end{tikzpicture} 			
			&
			\begin{tikzpicture}\tikzstyle{every node}=[font=\tiny]
			\begin{axis}[width=.31\textwidth,
			enlargelimits=false, 
			hide axis,
			axis equal image
			]
			\addplot
			graphics[xmin=0, xmax=16, ymin=0, ymax=16]{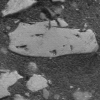};
			\addplot[opacity=0.5]
			graphics[xmin=0, xmax=16, ymin=0, ymax=16]{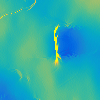};
			\end{axis}
			\end{tikzpicture} 
			&
			\begin{tikzpicture}\tikzstyle{every node}=[font=\tiny]
			\begin{axis}[width=.31\textwidth,
			enlargelimits=false, 
			hide axis,
			axis equal image
			]
			\addplot
			graphics[xmin=0, xmax=16, ymin=0, ymax=16]{./images/strain/results_Serie71/regionTGV_im0.png};
			\addplot[opacity=0.5]
			graphics[xmin=0, xmax=16, ymin=0, ymax=16]{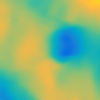};
			\end{axis}
			\end{tikzpicture} 
			&
			\begin{tikzpicture}\tikzstyle{every node}=[font=\tiny]
			\begin{axis}[width=.31\textwidth,
			enlargelimits=false, 
			hide axis,
			axis equal image
			]
			\addplot
			graphics[xmin=0, xmax=16, ymin=0, ymax=16]{./images/strain/results_Serie71/regionTGV_im0.png};
			\addplot[opacity=0.5]
			graphics[xmin=0, xmax=16, ymin=0, ymax=16]{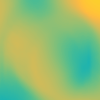};
			\end{axis}
			\end{tikzpicture} 	 
			\\[3ex] 			
			\begin{tikzpicture}\tikzstyle{every node}=[font=\tiny]
			\begin{axis}[width=.284\textwidth,
			enlargelimits=false, 
			hide axis,
			axis equal image
			]
			\addplot
			graphics[xmin=0, xmax=650, ymin=0, ymax=550]{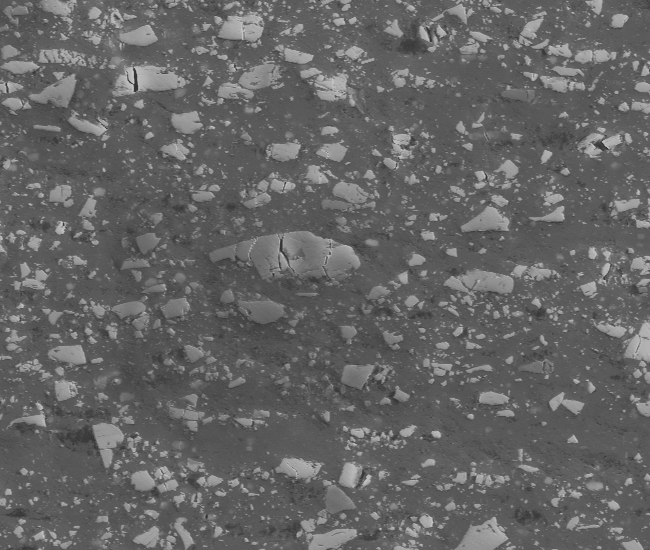};
			\draw[color=black,thick] (451,321) rectangle (530,400);
			\draw[draw=none,fill=white] (374,5) rectangle (644,160);
			\draw[{|[width=4pt]}-{|[width=4pt]},thick] (429,50) -- (589,50) node[above=-2pt,pos=0.5] {20\,\textmu m};
			\end{axis}
			\end{tikzpicture}			
			&
			\begin{tikzpicture}\tikzstyle{every node}=[font=\tiny]
			\begin{axis}[width=.284\textwidth,
			enlargelimits=false, 
			hide axis,
			axis equal image,
			colorbar,
			colormap name=parula,
			colorbar style={overlay, width=0.15cm, yticklabel style={
					/pgf/number format/.cd,
					fixed,
					fixed zerofill,
					/tikz/.cd}}
			]
			\addplot 
			graphics[xmin=0, xmax=650, ymin=0, ymax=550]{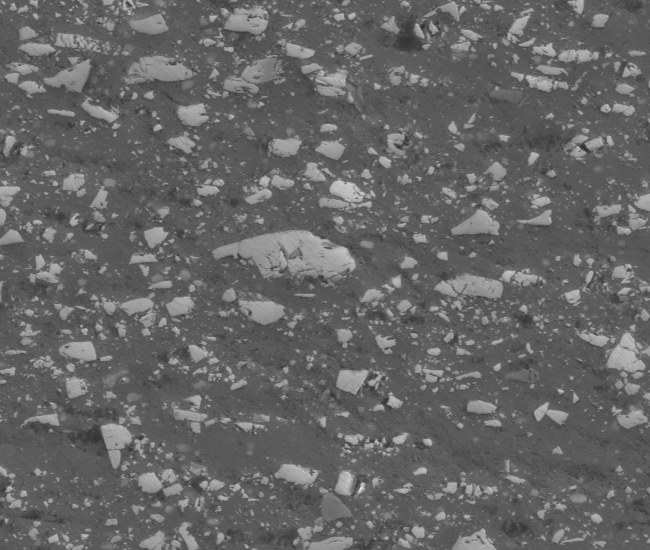};
			\addplot[point meta min=0.0163, point meta max=0.4, opacity=0.5] 
			graphics[xmin=0, xmax=650, ymin=0, ymax=550]{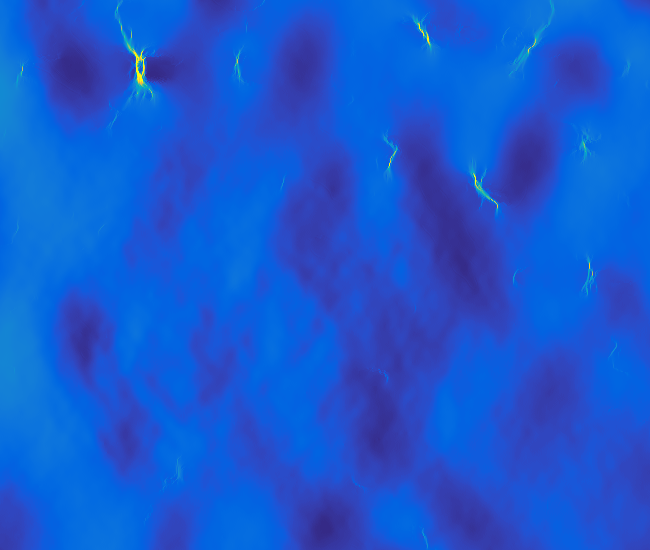};
			\end{axis}
			\end{tikzpicture} \hspace{0.5cm} 
			&
			\begin{tikzpicture}\tikzstyle{every node}=[font=\tiny]
			\begin{axis}[width=.284\textwidth,
			enlargelimits=false, 
			hide axis,
			axis equal image,
			colorbar,
			colormap name=parula,
			colorbar style={overlay, width=0.15cm, yticklabel style={
					scaled ticks=false,
					/pgf/number format/.cd,
					fixed,
					fixed zerofill,
					/tikz/.cd}}
			]
			\addplot 
			graphics[xmin=0, xmax=650, ymin=0, ymax=550]{./images/strain/results_Serie1/image_1.png};
			\addplot[point meta min=0.0167, point meta max=0.1225,opacity=0.5] 
			graphics[xmin=0, xmax=650, ymin=0, ymax=550]{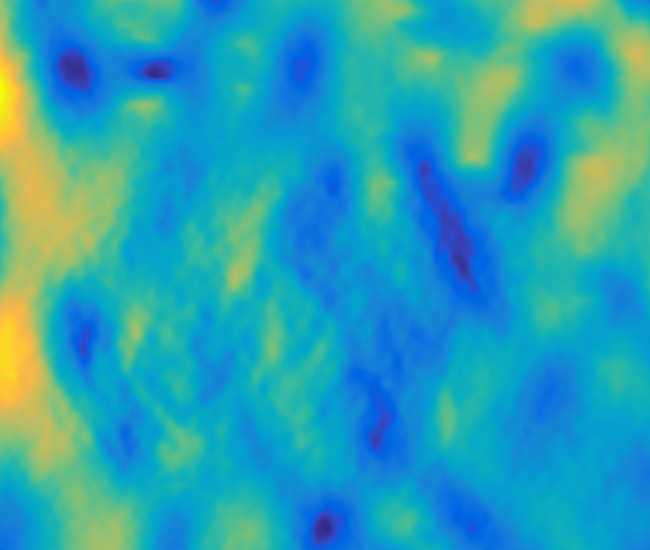};
			\end{axis}
			\end{tikzpicture} \hspace{0.5cm} 
			&
			\begin{tikzpicture}\tikzstyle{every node}=[font=\tiny]
			\begin{axis}[width=.284\textwidth,
			enlargelimits=false, 
			hide axis,
			axis equal image,
			colorbar,
			scaled ticks=false,
			colormap name=parula,
			colorbar style={overlay, width=0.15cm, yticklabel style={
					scaled ticks=false,
					/pgf/number format/.cd,
					fixed,
					fixed zerofill,
					/tikz/.cd}}
			]
			\addplot 
			graphics[xmin=0, xmax=650, ymin=0, ymax=550]{./images/strain/results_Serie1/image_1.png};
			\addplot[point meta min=0.0167, point meta max=0.1225,opacity=0.5] 
			graphics[xmin=0, xmax=650, ymin=0, ymax=550]{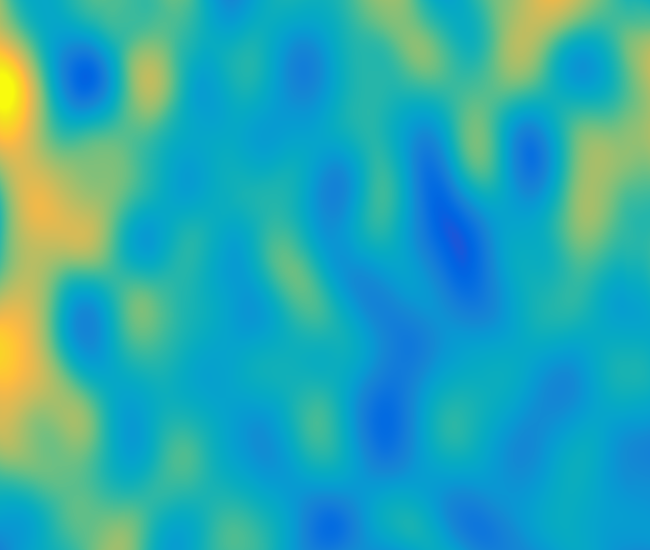};
			\end{axis}
			\end{tikzpicture} \hspace{0.5cm} 
			\\		
			\begin{tikzpicture}\tikzstyle{every node}=[font=\tiny]
			\begin{axis}[width=.31\textwidth,
			enlargelimits=false, 
			hide axis,
			axis equal image
			]
			\addplot
			graphics[xmin=0, xmax=80, ymin=0, ymax=80]{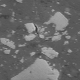};
			\draw[draw=none,fill=white] (50.5,1.5) rectangle (78.5,19);
			\draw[{|[width=4pt]}-{|[width=4pt]},thick] (56.5,6) -- (72.5,6) node[above=-2pt,pos=0.5] {2\,\textmu m};
			\end{axis}
			\end{tikzpicture} 			
			&
			\begin{tikzpicture}\tikzstyle{every node}=[font=\tiny]
			\begin{axis}[width=.31\textwidth,
			enlargelimits=false, 
			hide axis,
			axis equal image
			]
			\addplot
			graphics[xmin=0, xmax=16, ymin=0, ymax=16]{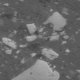};
			\addplot[opacity=0.5]
			graphics[xmin=0, xmax=16, ymin=0, ymax=16]{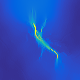};
			\end{axis}
			\end{tikzpicture} 
			&
			\begin{tikzpicture}\tikzstyle{every node}=[font=\tiny]
			\begin{axis}[width=.31\textwidth,
			enlargelimits=false, 
			hide axis,
			axis equal image
			]
			\addplot
			graphics[xmin=0, xmax=16, ymin=0, ymax=16]{./images/strain/results_Serie1/regionTGV_im0.png};
			\addplot[opacity=0.5]
			graphics[xmin=0, xmax=16, ymin=0, ymax=16]{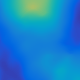};
			\end{axis}
			\end{tikzpicture} 
			&
			\begin{tikzpicture}\tikzstyle{every node}=[font=\tiny]
			\begin{axis}[width=.31\textwidth,
			enlargelimits=false, 
			hide axis,
			axis equal image
			]
			\addplot
			graphics[xmin=0, xmax=16, ymin=0, ymax=16]{./images/strain/results_Serie1/regionTGV_im0.png};
			\addplot[opacity=0.5]
			graphics[xmin=0, xmax=16, ymin=0, ymax=16]{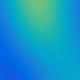};
			\end{axis}
			\end{tikzpicture} 		
		\end{tabular}
		\caption{ \label{fig:tgv_ncorr}
			Comparison of the proposed model to the correlation based method NCorr for three different image pairs and enlarged regions.} 
	\end{figure}	
	
	The second and third example show more complicated deformations.
	In addition to $\nabla_x u_1$, we depict the smooth strain part $a_{1}$, that is split off by the TGV regularizer.
	It is visible that the result of NCorr looks very similar to $a_{1}$, in particular the same peaks and structures in a degree of $\pm \frac{\pi}{4}$ can be observed.
	But in addition to this smooth part, which resembles the strain in the classical meaning in materials science, our method is able to resolve also the local damage.
	
	In summary, the proposed method is on the one hand able to extract the information that also commercial software extracts.
	This warrants that the model in general gives correct results from the viewpoint of materials science.
	On the other hand, it is possible to visualize the local behavior with a very high resolution where correlation-based methods fail.
	
	\subsection{Results for Other Materials and Experimental Settings} \label{sub:else}
	Finally, we leave the setting of tensile tests 
	and show some results for other experimental settings and materials.
	
	\textbf{Figure~\ref{fig:compress}} contains results for a compression test of a fiber reinforced aluminum matrix where load was applied in $x$-direction. 
	Since the material is compressed in $x$-direction, cracks are expected to open up in $y$-direction.
	Hence, we depict also $u_2$ and $\nabla_y u_2$ since these are more suitable for this setting.
	Also in this case, the proposed method provides reasonable results and resolves cracks.
	
	\begin{figure} \centering
		\begin{subfigure}[b]{0.99\textwidth}\centering
			\begin{tabular}{cc}
				\multicolumn{2}{c}{uncompressed state}
				\\
				\begin{tikzpicture}\tikzstyle{every node}=[font=\tiny] 
				\begin{axis}[ 
				height=.325\textwidth, 
				enlargelimits=false, 
				hide axis, 
				axis equal image, 
				axis on top
				]
				\addplot[point meta min=-4.88,
				point meta max=6.83]
				graphics[xmin=0, xmax=500, ymin=0, ymax=500] 
				{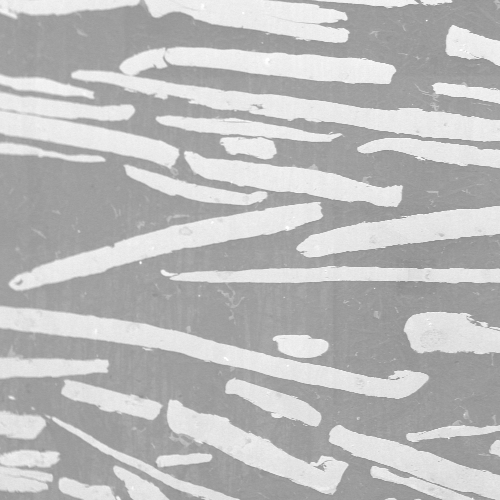}; 
				\draw[color=black,thick] (91,460) rectangle (390,290);
				\draw[draw=none,fill=white] (304,5) rectangle (494,85);
				\draw[{|[width=4pt]}-{|[width=4pt]},thick] (339,25) -- (459,25) node[above=-2pt,pos=0.5] {200\,\textmu m};
				\end{axis} 
				\end{tikzpicture}
				&				
				\begin{tikzpicture}\tikzstyle{every node}=[font=\tiny] 
				\begin{axis}[ 
				height=.3\textwidth, 
				enlargelimits=false, 
				xlabel={x}, 
				ylabel={y}, 
				scale only axis, 
				hide axis, 
				axis equal image, 
				axis on top
				] 
				\addplot[surf,point meta min=-5.1702, point meta max=10.5975]  
				graphics[xmin=0, xmax=300, ymin=0, ymax=170] 
				{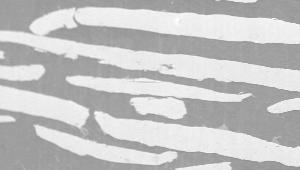}; 
				\draw[draw=none,fill=white] (228,2) rectangle (298,28);
				\draw[{|[width=4pt]}-{|[width=4pt]},thick] (233,8) -- (293,8) node[above=-2pt,pos=0.5] {100\,\textmu m};
				\end{axis} 
				\end{tikzpicture}
				\\[2ex]
				\multicolumn{2}{c}{compressed state}
				\\
				\begin{tikzpicture}\tikzstyle{every node}=[font=\tiny] 
				\begin{axis}[ 
				height=.415\textwidth, 
				enlargelimits=false, 
				xmin=-200,
				xmax=700,
				hide axis, 
				axis equal image, 
				axis on top
				]
				\addplot
				graphics[xmin=0, xmax=500, ymin=0, ymax=500] 
				{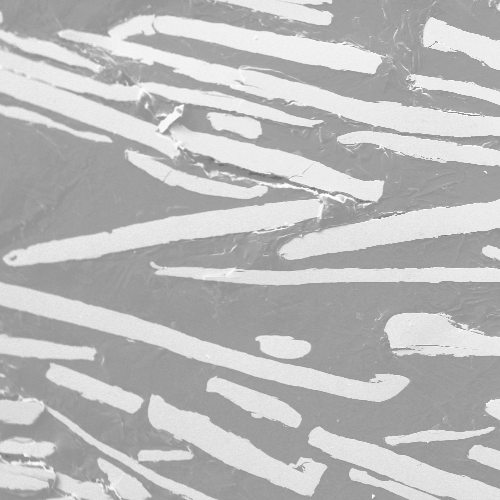}; 
				\draw[->,very thick] (-160,250) -- (-10,250) node[above=-2pt,pos=0.5] {\large $F$};
				\draw[->,very thick] (660,250) -- (510,250) node[above=-2pt,pos=0.5] {\large $F$};
				\draw[color=black,thick] (91,460) rectangle (390,290);
				\draw[draw=none,fill=white] (304,5) rectangle (494,85);
				\draw[{|[width=4pt]}-{|[width=4pt]},thick] (339,25) -- (459,25) node[above=-2pt,pos=0.5] {200\,\textmu m};
				\end{axis} 
				\end{tikzpicture}
				&				
				\begin{tikzpicture}\tikzstyle{every node}=[font=\tiny] 
				\begin{axis}[ 
				height=.42\textwidth, 
				enlargelimits=false,
				xmin=-60,
				xmax=360,
				xlabel={x}, 
				ylabel={y}, 
				scale only axis, 
				hide axis, 
				axis equal image, 
				axis on top
				] 
				\addplot[surf,point meta min=-5.1702, point meta max=10.5975]  
				graphics[xmin=0, xmax=300, ymin=0, ymax=170] 
				{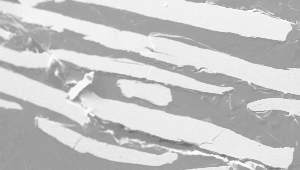};
				\draw[->,very thick] (-60,85) -- (-5,85) node[above=-2pt,pos=0.5] {\large $F$}; 
				\draw[->,very thick] (360,85) -- (305,85) node[above=-2pt,pos=0.5] {\large $F$};
				\draw[draw=none,fill=white] (228,2) rectangle (298,28);
				\draw[{|[width=4pt]}-{|[width=4pt]},thick] (233,8) -- (293,8) node[above=-2pt,pos=0.5] {100\,\textmu m};
				\end{axis} 
				\end{tikzpicture}
			\end{tabular}
			\caption{Microstructure images of the uncompressed and compressed state.}
		\end{subfigure}
		\\[2ex]
		\begin{subfigure}[b]{0.99\textwidth}\centering
			\begin{tabular}{ccc}
				$u_1$
				&		
				\raisebox{-.5\height}{
					\begin{tikzpicture}\tikzstyle{every node}=[font=\tiny] 
					\begin{axis}[ 
					height=.325\textwidth, 
					enlargelimits=false, 
					hide axis, 
					axis equal image, 
					axis on top
					]
					\addplot[point meta min=-4.88,
					point meta max=6.83]
					graphics[xmin=0, xmax=500, ymin=0, ymax=500] 
					{./images/strain/results_KIT3/image_1.png}; 
					\addplot[opacity=0.5]  
					graphics[xmin=0, xmax=500, ymin=0, ymax=500] 
					{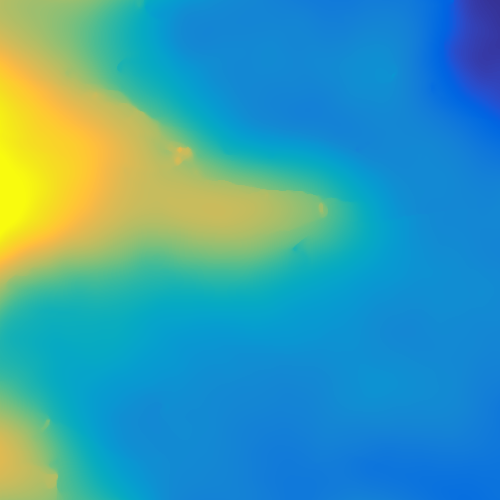}; 
					\end{axis} 
					\end{tikzpicture}}
				&				
				\raisebox{-.5\height}{
					\begin{tikzpicture}\tikzstyle{every node}=[font=\tiny] 
					\begin{axis}[ 
					height=.3\textwidth, 
					enlargelimits=false, 
					xlabel={x}, 
					ylabel={y}, 
					scale only axis, 
					hide axis, 
					axis equal image, 
					axis on top, 
					colorbar, 
					scaled ticks=false, 
					colormap name=parula, 
					colorbar right, 
					colorbar style={ 
						overlay,width=0.15cm
					}
					] 
					\addplot[surf,point meta min=-8.05, point meta max=11.27]  
					graphics[xmin=0, xmax=300, ymin=0, ymax=170] 
					{./images/strain/results_KIT3/regionu2_im0.png}; 
					\addplot[surf,point meta min=-5.1702, point meta max=10.5975,opacity=0.5]  
					graphics[xmin=0, xmax=300, ymin=0, ymax=170] 
					{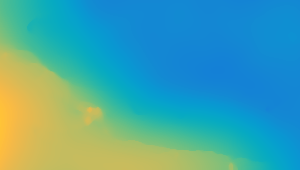}; 
					\end{axis} 
					\end{tikzpicture}}
				\\[9ex]
				$u_2$
				&		
				\raisebox{-.5\height}{
					\begin{tikzpicture}\tikzstyle{every node}=[font=\tiny] 
					\begin{axis}[ 
					height=.325\textwidth, 
					enlargelimits=false, 
					hide axis, 
					axis equal image, 
					axis on top
					]
					\addplot
					graphics[xmin=0, xmax=500, ymin=0, ymax=500] 
					{./images/strain/results_KIT3/image_1.png}; 
					\addplot[opacity=0.5]  
					graphics[xmin=0, xmax=500, ymin=0, ymax=500] 
					{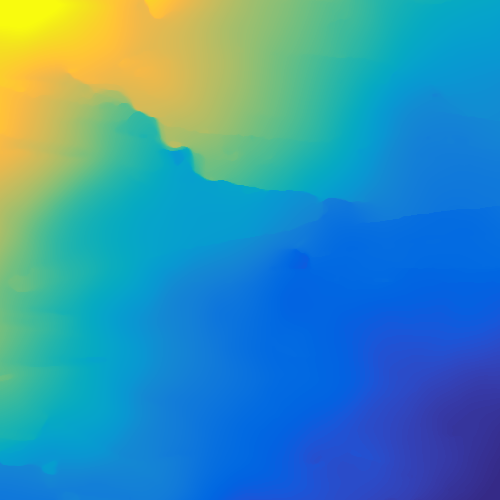}; 
					\end{axis} 
					\end{tikzpicture}}
				&				
				\raisebox{-.5\height}{
					\begin{tikzpicture}\tikzstyle{every node}=[font=\tiny] 
					\begin{axis}[ 
					height=.3\textwidth, 
					enlargelimits=false, 
					xlabel={x}, 
					ylabel={y}, 
					scale only axis, 
					hide axis, 
					axis equal image, 
					axis on top, 
					colorbar, 
					scaled ticks=false, 
					colormap name=parula, 
					colorbar right, 
					colorbar style={ 
						overlay,width=0.15cm
					}
					]  
					\addplot[surf,point meta min=-2.38, point meta max=23.64]  
					graphics[xmin=0, xmax=300, ymin=0, ymax=170] 
					{./images/strain/results_KIT3/regionu2_im0.png}; 
					\addplot[surf,point meta min=-5.1702, point meta max=10.5975,opacity=0.5]  
					graphics[xmin=0, xmax=300, ymin=0, ymax=170] 
					{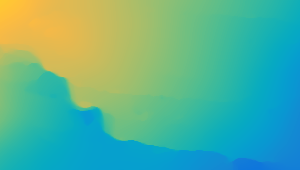}; 
					\end{axis} 
					\end{tikzpicture}}
				\\[9ex]
				$\nabla_x u_1$
				&		
				\raisebox{-.5\height}{
					\begin{tikzpicture}\tikzstyle{every node}=[font=\tiny] 
					\begin{axis}[ 
					height=.325\textwidth, 
					enlargelimits=false, 
					hide axis, 
					axis equal image, 
					axis on top
					]
					\addplot[point meta min=-4.88,
					point meta max=6.83]
					graphics[xmin=0, xmax=500, ymin=0, ymax=500] 
					{./images/strain/results_KIT3/image_1.png}; 
					\addplot[opacity=0.5]  
					graphics[xmin=0, xmax=500, ymin=0, ymax=500] 
					{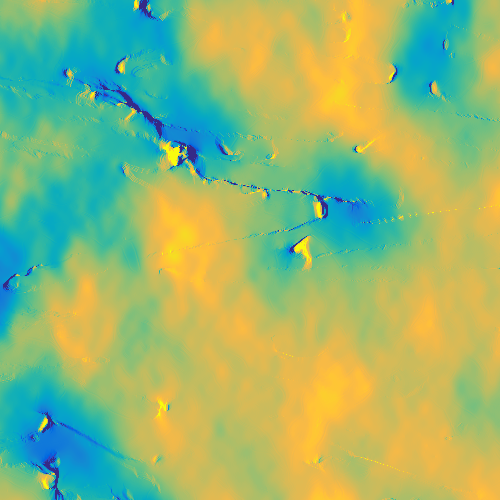}; 
					\end{axis} 
					\end{tikzpicture}} 
				&				
				\raisebox{-.5\height}{
					\begin{tikzpicture}\tikzstyle{every node}=[font=\tiny] 
					\begin{axis}[ 
					height=.3\textwidth, 
					enlargelimits=false, 
					xlabel={x}, 
					ylabel={y}, 
					scale only axis, 
					hide axis, 
					axis equal image, 
					axis on top, 
					colorbar, 
					scaled ticks=false, 
					colormap name=parula, 
					colorbar right, 
					colorbar style={overlay, width=0.15cm, yticklabel style={
							scaled ticks=false,
							/pgf/number format/.cd,
							fixed,
							fixed zerofill,
							/tikz/.cd}}
					] 
					\addplot[surf,point meta min=-0.0765, point meta max=0.019]  
					graphics[xmin=0, xmax=300, ymin=0, ymax=170] 
					{./images/strain/results_KIT3/regionu2_im0.png}; 
					\addplot[opacity=0.5]  
					graphics[xmin=0, xmax=300, ymin=0, ymax=170] 
					{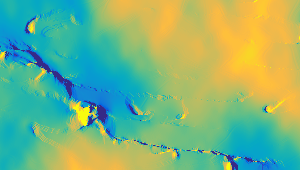}; 
					\end{axis} 
					\end{tikzpicture}}
				\\[9ex]
				$\nabla_y u_2$
				&		
				\raisebox{-.5\height}{
					\begin{tikzpicture}\tikzstyle{every node}=[font=\tiny] 
					\begin{axis}[ 
					height=.325\textwidth, 
					enlargelimits=false, 
					hide axis, 
					axis equal image, 
					axis on top
					]
					\addplot
					graphics[xmin=0, xmax=500, ymin=0, ymax=500] 
					{./images/strain/results_KIT3/image_1.png}; 
					\addplot[opacity=0.5]  
					graphics[xmin=0, xmax=500, ymin=0, ymax=500] 
					{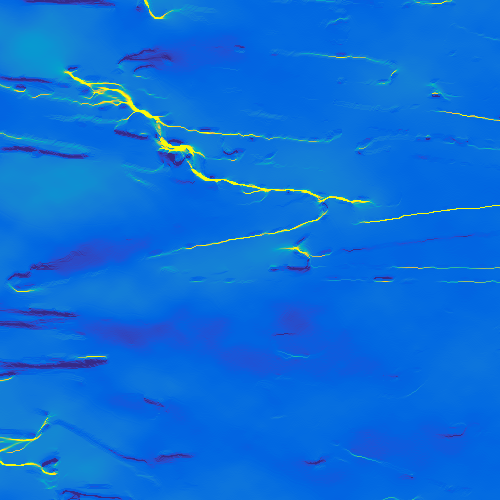}; 
					\end{axis} 
					\end{tikzpicture}} 
				&				
				\raisebox{-.5\height}{
					\begin{tikzpicture}\tikzstyle{every node}=[font=\tiny] 
					\begin{axis}[ 
					height=.3\textwidth, 
					enlargelimits=false, 
					xlabel={x}, 
					ylabel={y}, 
					scale only axis, 
					hide axis, 
					axis equal image, 
					axis on top, 
					colorbar, 
					scaled ticks=false, 
					colormap name=parula, 
					colorbar right, 
					colorbar style={ 
						overlay,width=0.15cm
					}
					] 
					\addplot[surf,point meta min=-0.037, point meta max=0.2592]  
					graphics[xmin=0, xmax=300, ymin=0, ymax=170] 
					{./images/strain/results_KIT3/regionu2_im0.png}; 
					\addplot[opacity=0.5]  
					graphics[xmin=0, xmax=300, ymin=0, ymax=170] 
					{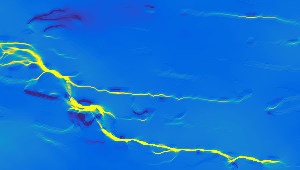}; 
					\end{axis} 
					\end{tikzpicture}}
			\end{tabular}
			\caption{Displacement $u_1$, $u_2$ in \textmu m ($10$\,\textmu m $= 6$\,px) and the strain $\nabla_x u_1$, $\nabla_y u_2$.}
		\end{subfigure}
		\caption{ \label{fig:compress}
			Results for a compression test of a fiber reinforced aluminum matrix.} 
	\end{figure}
	
	\textbf{Figure~\ref{fig:fatigue}} shows the results for a fatigue test.
	Here, we depict the images and results for one load cycle illustrated in Figure~\ref{fig:fatigue:plot}.
	The first image shows the specimen in state ``A'' without load $\sigma_A=0$\,MPa, then load is applied, i.e., state ``B'' with $\sigma_B=632$\,MPa.
	For the third image, i.e.~state ``C'', the load is released again, i.e., $\sigma_C=0$\,MPa and then the specimen is compressed to state ``D'' with $\sigma_D=-630$\,MPa. For the last image in state ``E'', the load is released again, i.e., $\sigma_E=0$\,MPa.
	We use the image of state ``A'' as the reference image.
	The computed results are shown in Figure~\ref{fig:fatigue:res}.
	Similar to the previous examples of tensile tests, we see that the strain 
	is positive almost everywhere for the image under load and we see exactly where cracks open up.
	For state ``C'', which corresponds to the state after releasing the tension, the strain is mainly positive as well since only the elastic deformation reverses.
	For the image under compression, the displacement $u_1$ slightly decreases from left to right and the strain is mainly negative.
	In the last state ``E'' after completing the load cycle,
	the specimen is slightly more strained than in the initial state ``A''.
	In particular, the crack in the particle opened up a bit.
	This is visible in the strain component $\nabla_x u_1$.
	Also the stress-strain curve in Figure~\ref{fig:fatigue:plot} shows a positive strain for state ``E''.		
	
	\begin{figure} \centering
		\begin{subfigure}[b]{0.99\textwidth}\centering	
			\raisebox{-.5\height}{	
				\includegraphics[width=0.55\textwidth]{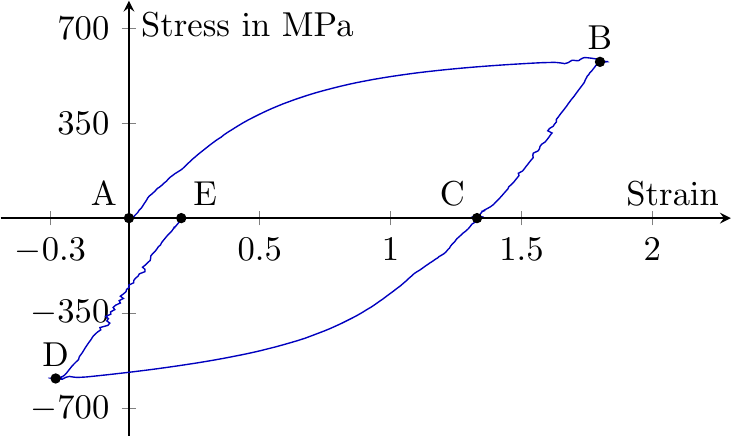}
			}
			\begin{tabular}{c}
				State ``A''
				\\
				\begin{tikzpicture}\tikzstyle{every node}=[font=\tiny] 
				\begin{axis}[ 
				height=.325\textwidth, 
				enlargelimits=false, 
				hide axis, 
				axis equal image, 
				axis on top
				]
				\addplot[point meta min=-4.88,
				point meta max=6.83]
				graphics[xmin=0, xmax=600, ymin=0, ymax=600] 
				{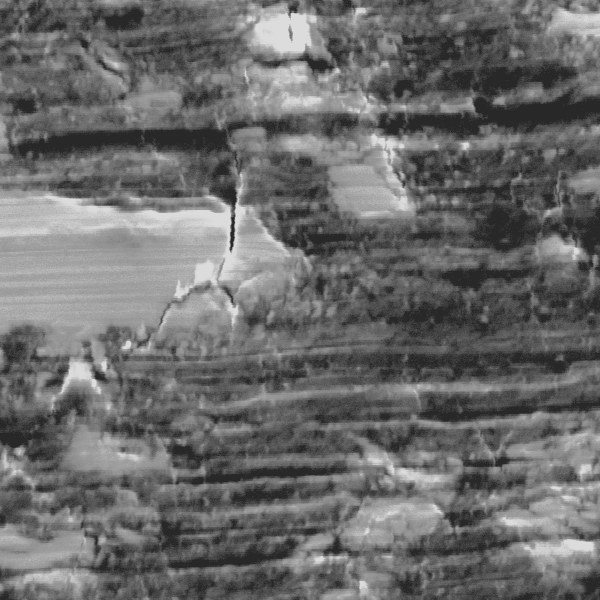};
				\draw[draw=none,fill=white] (404,5) rectangle (594,105);
				\draw[{|[width=4pt]}-{|[width=4pt]},thick] (419,25) -- (579,25) node[above=-2pt,pos=0.5] {5\,\textmu m};
				\end{axis} 
				\end{tikzpicture}
			\end{tabular}
			\caption[]{Hysteresis curve of one cycle of a fatigue experiment and microstructure image of the initial state ``A''.}\label{fig:fatigue:plot}
		\end{subfigure}\\[1ex]
		\begin{subfigure}[b]{0.99\textwidth}\centering
			\setlength{\tabcolsep}{3pt}
			\begin{tabular}{ccccc}
				& State ``B'' & State ``C'' & State ``D'' &	State ``E'' \qquad \qquad \\
				&	
				\raisebox{-.5\height}{
					\begin{tikzpicture}\tikzstyle{every node}=[font=\tiny] 
					\begin{axis}[ 
					width=.345\textwidth, 
					enlargelimits=false, 
					hide axis, 
					axis equal image, 
					axis on top, 
					scaled ticks=false
					] 
					\addplot
					graphics[xmin=0, xmax=600, ymin=0, ymax=600] 
					{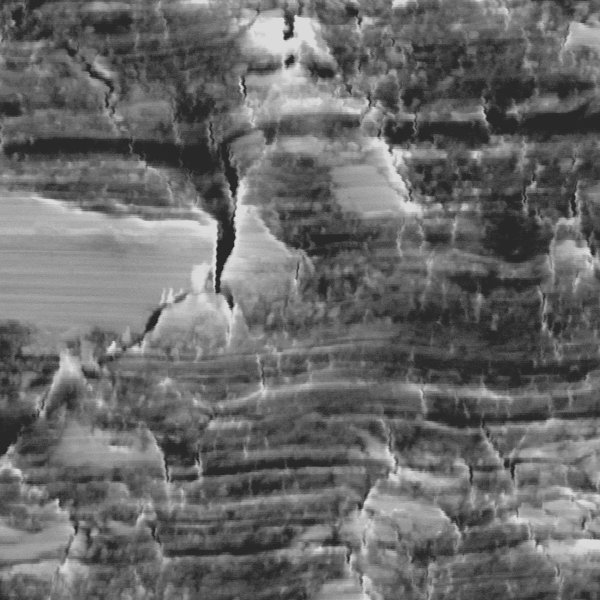}; 
					\draw[draw=none,fill=white] (404,5) rectangle (594,110);
					\draw[{|[width=4pt]}-{|[width=4pt]},thick] (419,30) -- (579,30) node[above=-2pt,pos=0.5] {5\,\textmu m};
					\end{axis} 
					\end{tikzpicture}} 
				&	
				\raisebox{-.5\height}{
					\begin{tikzpicture}\tikzstyle{every node}=[font=\tiny] 
					\begin{axis}[ 
					width=.345\textwidth, 
					enlargelimits=false, 
					hide axis, 
					axis equal image, 
					axis on top, 
					scaled ticks=false
					] 
					\addplot
					graphics[xmin=0, xmax=600, ymin=0, ymax=600] 
					{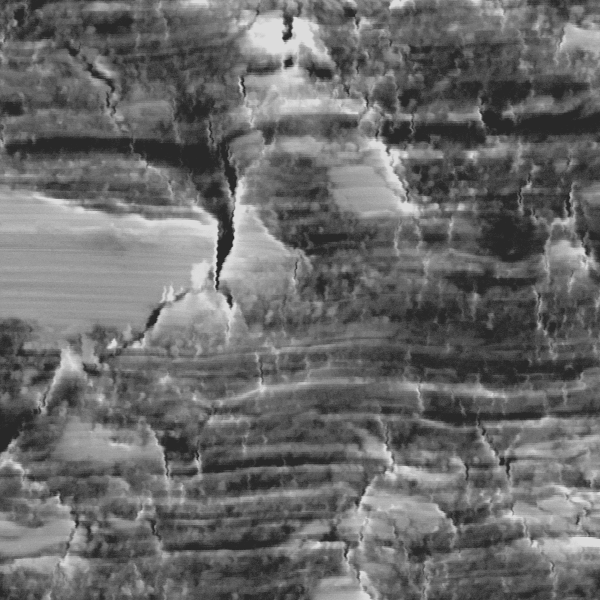};
					\draw[draw=none,fill=white] (404,5) rectangle (594,110);
					\draw[{|[width=4pt]}-{|[width=4pt]},thick] (419,30) -- (579,30) node[above=-2pt,pos=0.5] {5\,\textmu m};
					\end{axis} 
					\end{tikzpicture}}
				&	
				\raisebox{-.5\height}{
					\begin{tikzpicture}\tikzstyle{every node}=[font=\tiny] 
					\begin{axis}[ 
					width=.345\textwidth, 
					enlargelimits=false, 
					hide axis, 
					axis equal image, 
					axis on top, 
					scaled ticks=false
					] 
					\addplot
					graphics[xmin=0, xmax=600, ymin=0, ymax=600] 
					{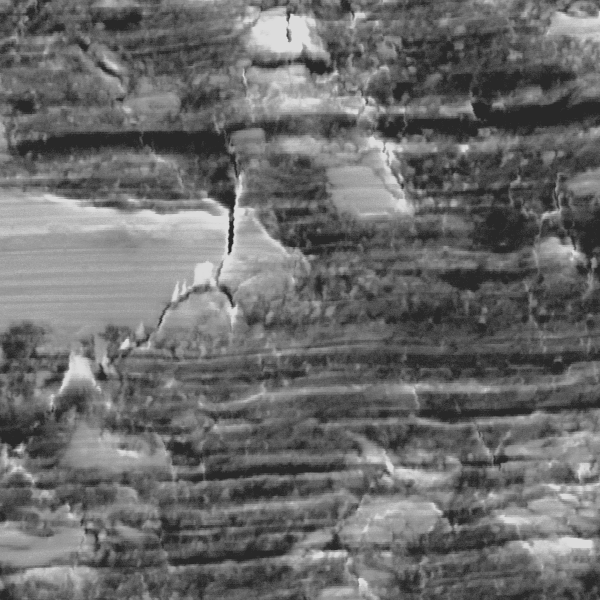}; 
					\draw[draw=none,fill=white] (404,5) rectangle (594,110);
					\draw[{|[width=4pt]}-{|[width=4pt]},thick] (419,30) -- (579,30) node[above=-2pt,pos=0.5] {5\,\textmu m};
					\end{axis} 
					\end{tikzpicture}}
				&
				\raisebox{-.5\height}{	
					\begin{tikzpicture}\tikzstyle{every node}=[font=\tiny] 
					\begin{axis}[ 
					width=.345\textwidth, 
					enlargelimits=false, 
					hide axis, 
					axis equal image, 
					axis on top, 
					scaled ticks=false
					] 
					\addplot
					graphics[xmin=0, xmax=600, ymin=0, ymax=600] 
					{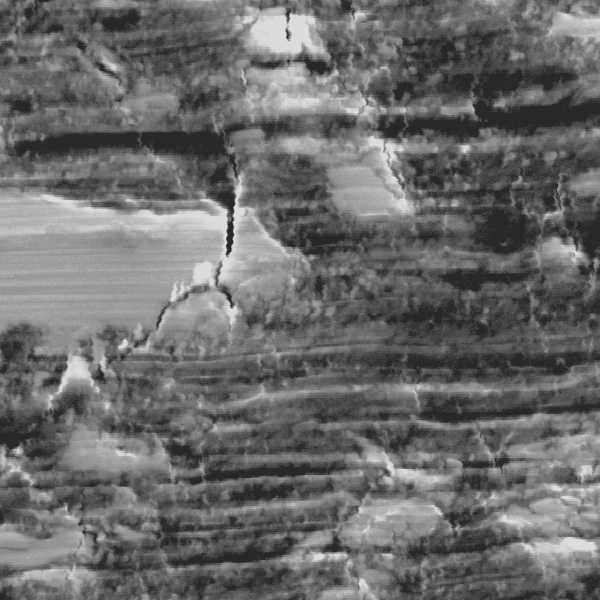}; 
					\draw[draw=none,fill=white] (404,5) rectangle (594,110);
					\draw[{|[width=4pt]}-{|[width=4pt]},thick] (419,30) -- (579,30) node[above=-2pt,pos=0.5] {5\,\textmu m};
					\end{axis} 
					\end{tikzpicture}}
				\hspace{0.8cm}
				\\[8ex]
				$u_1$
				&
				\raisebox{-.5\height}{
					\begin{tikzpicture}\tikzstyle{every node}=[font=\tiny] 
					\begin{axis}[ 
					width=.345\textwidth, 
					enlargelimits=false, 
					hide axis, 
					axis equal image, 
					axis on top
					]
					\addplot
					graphics[xmin=0, xmax=600, ymin=0, ymax=600] 
					{./images/strain/results_fatigue_new/image_1.png}; 
					\addplot[opacity=0.5]  
					graphics[xmin=0, xmax=600, ymin=0, ymax=600] 
					{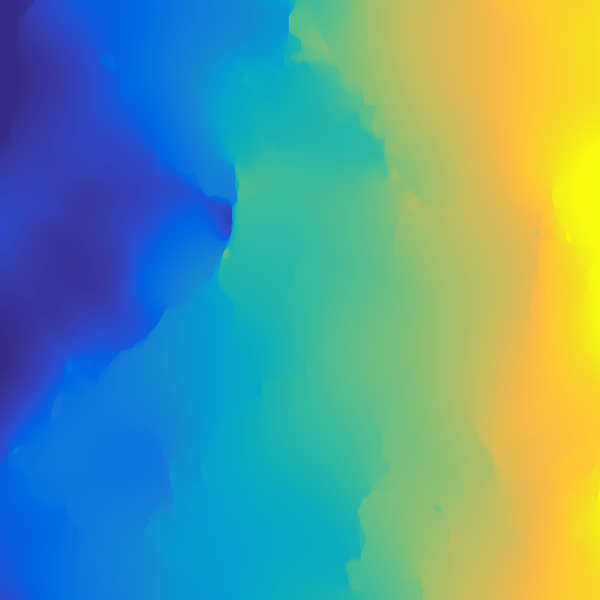}; 
					\end{axis} 
					\end{tikzpicture}}
				&
				\raisebox{-.5\height}{
					\begin{tikzpicture}\tikzstyle{every node}=[font=\tiny] 
					\begin{axis}[ 
					width=.345\textwidth, 
					enlargelimits=false, 
					hide axis, 
					axis equal image, 
					axis on top
					]
					\addplot
					graphics[xmin=0, xmax=600, ymin=0, ymax=600] 
					{./images/strain/results_fatigue_new/image_1.png}; 
					\addplot[opacity=0.5]  
					graphics[xmin=0, xmax=600, ymin=0, ymax=600] 
					{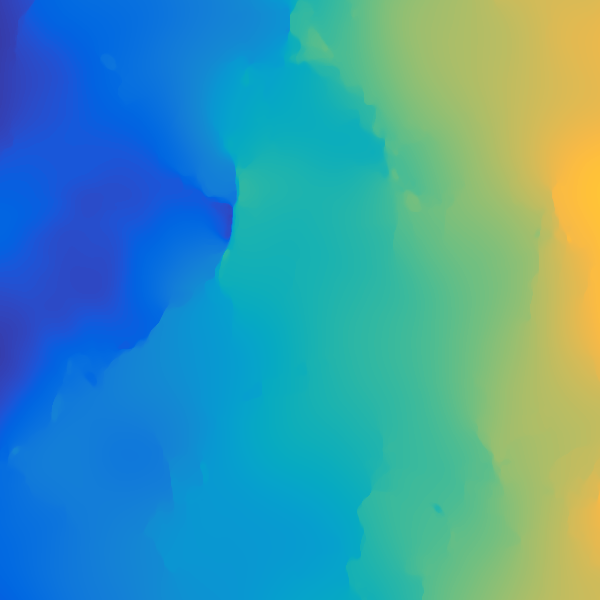}; 
					\end{axis} 
					\end{tikzpicture}} 
				&
				\raisebox{-.5\height}{
					\begin{tikzpicture}\tikzstyle{every node}=[font=\tiny] 
					\begin{axis}[ 
					width=.345\textwidth, 
					enlargelimits=false, 
					hide axis, 
					axis equal image, 
					axis on top
					]
					\addplot
					graphics[xmin=0, xmax=600, ymin=0, ymax=600] 
					{./images/strain/results_fatigue_new/image_1.png}; 
					\addplot[opacity=0.5]  
					graphics[xmin=0, xmax=600, ymin=0, ymax=600] 
					{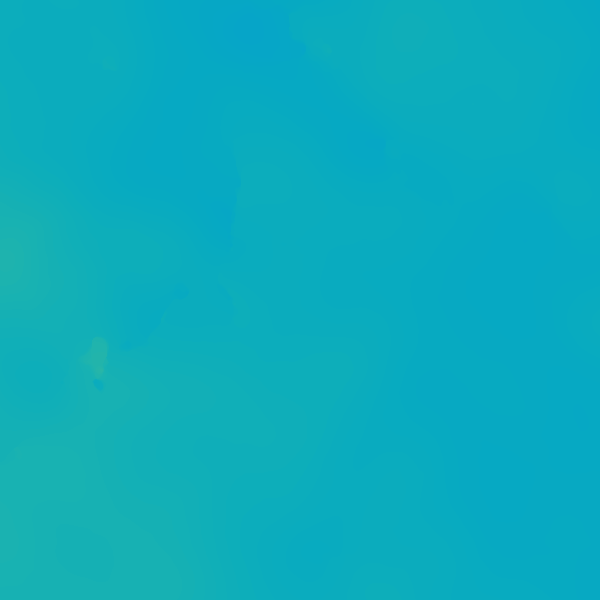}; 
					\end{axis} 
					\end{tikzpicture}}
				&
				\raisebox{-.5\height}{
					\begin{tikzpicture}\tikzstyle{every node}=[font=\tiny] 
					\begin{axis}[ 
					width=.345\textwidth, 
					enlargelimits=false, 
					hide axis, 
					axis equal image, 
					axis on top, 
					colorbar, 
					scaled ticks=false, 
					colormap name=parula, 
					colorbar right, 
					colorbar style={
						overlay, 
						width=0.15cm
					}
					]
					\addplot[point meta min=-0.5,
					point meta max=0.62]
					graphics[xmin=0, xmax=600, ymin=0, ymax=600] 
					{./images/strain/results_fatigue_new/image_1.png}; 
					\addplot[opacity=0.5]  
					graphics[xmin=0, xmax=600, ymin=0, ymax=600] 
					{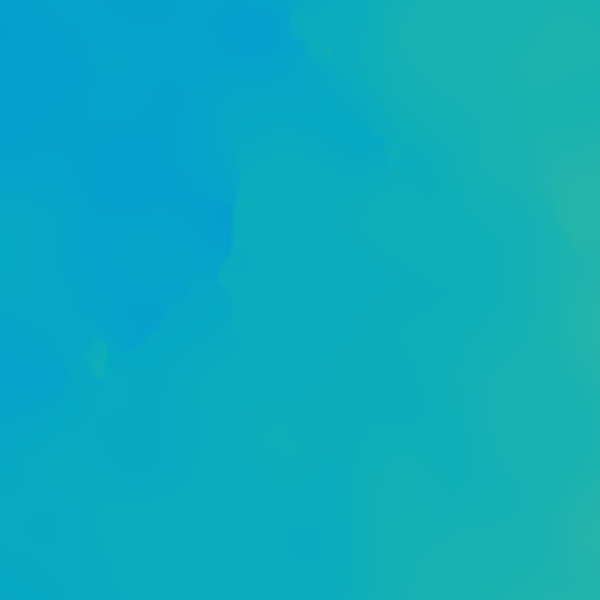}; 
					\end{axis} 
					\end{tikzpicture}} 
				\hspace{0.5cm}
				\\[8ex]
				$\nabla_x u_1$
				&
				\raisebox{-.5\height}{
					\begin{tikzpicture}\tikzstyle{every node}=[font=\tiny] 
					\begin{axis}[ 
					width=.345\textwidth, 
					enlargelimits=false, 
					hide axis, 
					axis equal image, 
					axis on top
					]
					\addplot
					graphics[xmin=0, xmax=600, ymin=0, ymax=600] 
					{./images/strain/results_fatigue_new/image_1.png}; 
					\addplot[opacity=0.5]  
					graphics[xmin=0, xmax=600, ymin=0, ymax=600] 
					{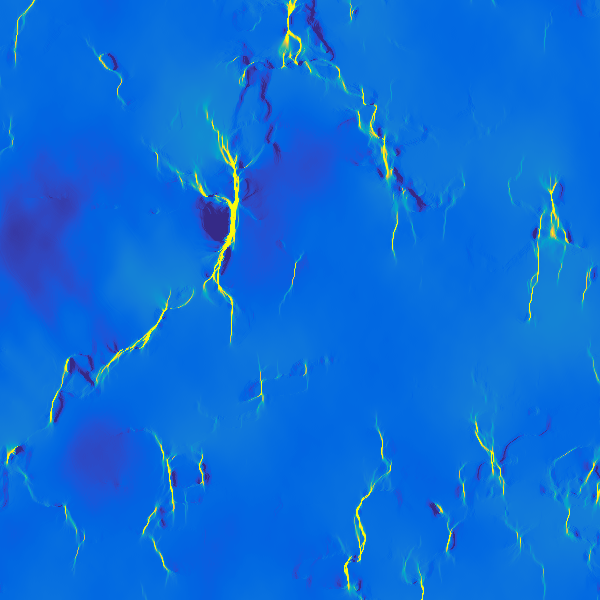}; 
					\end{axis} 
					\end{tikzpicture}} 
				&
				\raisebox{-.5\height}{
					\begin{tikzpicture}\tikzstyle{every node}=[font=\tiny] 
					\begin{axis}[ 
					width=.345\textwidth, 
					enlargelimits=false, 
					hide axis, 
					axis equal image, 
					axis on top
					]
					\addplot
					graphics[xmin=0, xmax=600, ymin=0, ymax=600] 
					{./images/strain/results_fatigue_new/image_1.png}; 
					\addplot[opacity=0.5]  
					graphics[xmin=0, xmax=600, ymin=0, ymax=600] 
					{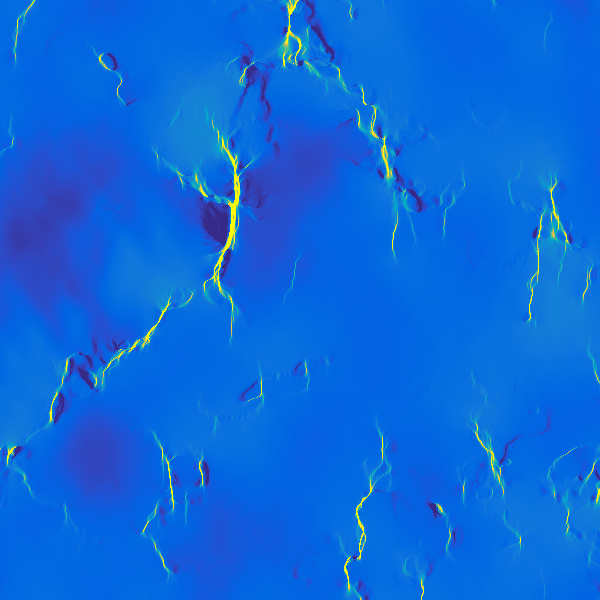}; 
					\end{axis} 
					\end{tikzpicture}} 
				&
				\raisebox{-.5\height}{
					\begin{tikzpicture}\tikzstyle{every node}=[font=\tiny] 
					\begin{axis}[ 
					width=.345\textwidth, 
					enlargelimits=false, 
					hide axis, 
					axis equal image, 
					axis on top
					]
					\addplot
					graphics[xmin=0, xmax=600, ymin=0, ymax=600] 
					{./images/strain/results_fatigue_new/image_1.png}; 
					\addplot[opacity=0.5]  
					graphics[xmin=0, xmax=600, ymin=0, ymax=600] 
					{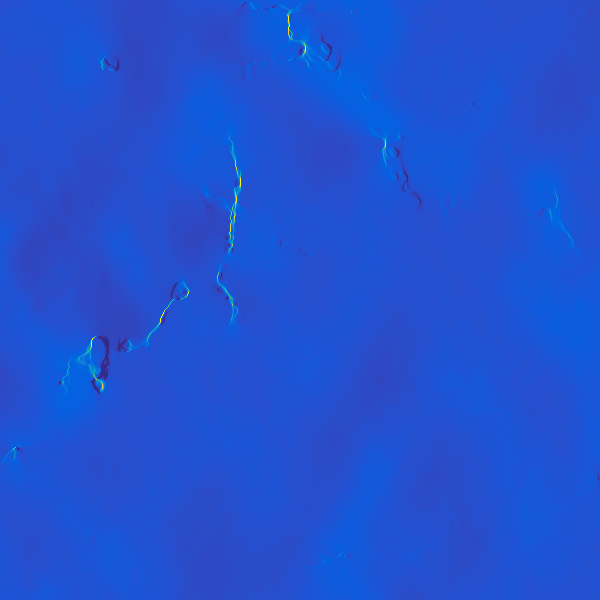}; 
					\end{axis} 
					\end{tikzpicture}} 
				&
				\raisebox{-.5\height}{
					\begin{tikzpicture}\tikzstyle{every node}=[font=\tiny] 
					\begin{axis}[ 
					width=.345\textwidth, 
					enlargelimits=false, 
					hide axis, 
					axis equal image, 
					axis on top, 
					colorbar, 
					scaled ticks=false, 
					colormap name=parula, 
					colorbar right, 
					colorbar style={overlay, width=0.15cm, yticklabel style={
							scaled ticks=false,
							/pgf/number format/.cd,
							fixed,
							fixed zerofill,
							/tikz/.cd}}
					]
					\addplot[point meta min=-0.0829,
					point meta max=0.613]
					graphics[xmin=0, xmax=600, ymin=0, ymax=600] 
					{./images/strain/results_fatigue_new/image_1.png}; 
					\addplot[opacity=0.5]  
					graphics[xmin=0, xmax=600, ymin=0, ymax=600] 
					{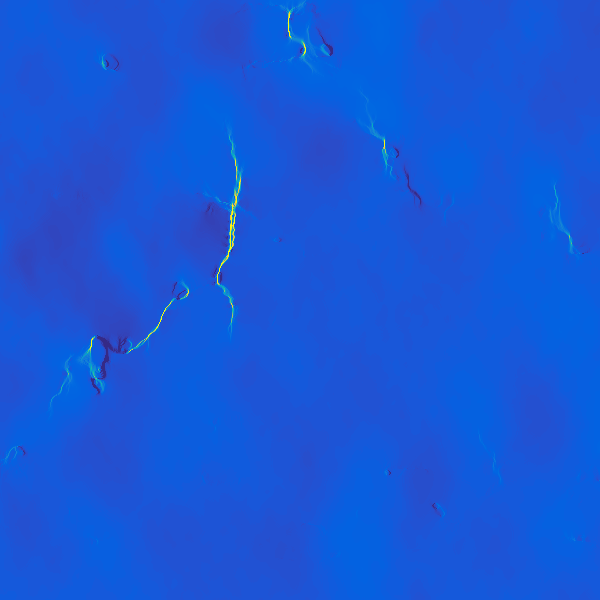}; 
					\end{axis} 
					\end{tikzpicture}} 
				\hspace{0.5cm}
			\end{tabular}
			\setlength{\tabcolsep}{6pt}
			\caption[]{Results for the states shown in (\subref{fig:fatigue:plot}).
				Displacement $u_1$ in \textmu m ($1$\,\textmu m $=32$\,px) and the strain $\nabla_x u_1$.\label{fig:fatigue:res}
			}
		\end{subfigure}
		\caption{ \label{fig:fatigue}
			Results for different states of a fatigue test.
			An exemplary hysteresis curve is shown in (\subref{fig:fatigue:plot}), the results for the states ``A''-``E'' are shown in (\subref{fig:fatigue:res}).		
		} 
	\end{figure}

	\section{Conclusions} \label{sec:strain:conc}
	In this paper, we proposed a variational optical flow model 
	for computing engineering strains on a microscale.
	Motivated by the setting of tensile tests, several first and a second order terms were evaluated for regularization.
	The method of our choice is a TGV regularized model with an optional constraint on the strain based on physical assumptions.
	The primal-dual method for finding a minimum of the corresponding energy function is especially suited since it directly computes the strain tensor within the iteration process.
	The results of the proposed model can be used for the detection of cracks and for an analysis of crack initiation and propagation.
	A comparison to state-of-the-art software used for strain analysis showed that the proposed method clearly outperforms this software, in particular when analyzing the local strain.
	Further, it highlights microstructural damage as an additional benefit.
	We have shown that, besides tensile tests with AMCs, the method is also applicable for different materials as well as in other experiments such as compression and fatigue tests.
	
	\section*{Acknowledgement}
	The authors thank K.~Lichtenberg from the ``Institute for Applied Materials (IAM)'' at the Karlsruhe Institute of Technology (KIT) for providing the data in Figure~\ref{fig:compress}.
	Funding by the German Research Foundation (DFG) 
	with\-in the Research Training Group 1932 ``Stochastic Models for Innovations in the Engineering Sciences'', project area P3, 
	is gratefully acknowledged.
	
	\appendix
	\section{Soft and Coupled Shrinkage}
	The following two propositions state well-known facts (see e.g.~\cite{BSS2016}), 
	which are required to compute the proximal maps in Algorithm~\ref{strain:alg1tgv}.
	
	\begin{proposition}
		Let $x \in \R$. 
		Then,
		\begin{align}
		\hat y = \argmin_{y\in\R} \left\{ \lambda |y| + \frac12 (y - x)^2\right\}
		\end{align}
		is given by the soft shrinkage
		\begin{align}\label{not:shr}
		\hat y = S_\lambda(x) \coloneqq & 
		\begin{cases} 
		0 & \mathrm{if} \ |x| \le \lambda,\\ 
		x(1- \frac{\lambda}{|x|}) & \mathrm{if} \ |x| > \lambda. 
		\end{cases} 
		\end{align} 
	\end{proposition}
	
	\begin{proposition}
		Let $x \in \R^d$. 
		Then,
		\begin{align}
		\hat y = \argmin_{y\in\R^d} \left\{ \lambda \|y\|_2 + \frac12 \|y - x\|_2^2\right\}
		\end{align}
		is given by the grouped or coupled shrinkage
		\begin{align}
		\hat y = \mathbf{S}_\lambda(x) \coloneqq 
		\begin{cases}
		0 & \mathrm{if} \ \| x\|_2 \le \lambda,\\ 
		x(1-\frac{\lambda}{\|x\|_2}) & \mathrm{if} \ \|x\|_2 > \lambda.
		\end{cases}
		\end{align}
	\end{proposition}
	
	\noindent
	\textbf{Computation of~\eqref{genshrink:eq}.}
	Finally, we compute for given $x \in \R^2$ the value
	\begin{align}\label{eq:genshrinkeq}
	\hat y = \argmin_{y\in\R^2} \left\{|\alpha y_1 + \beta y_2 + \gamma| + \frac12 \|x-y\|_2^2\right\},
	\end{align}
	where $\alpha, \beta, \gamma \in \R$.	
	We need to distinguish the following cases:
	\begin{enumerate}
		\item
		$\alpha = \beta = 0$: This gives obviously $\hat y=x$.	
		\item
		$\alpha = 0, \beta \ne 0 $: Here, we get
		$\hat y_1 = x_1$
		and
		\begin{align}
		\hat y_2 &= \argmin_{y_2 \in \R} \left\{| \beta y_2 + \gamma | + \frac{1}{2} (x_2-y_2)^2\right\} \\
		&= \argmin_{y_2 \in \R} \left\{| y_2 + \frac{\gamma}{\beta} | + \frac{1}{2|\beta|} (x_2-y_2)^2\right\} \\
		&= S_{|\beta|}\left(x_2+\frac{\gamma}{\beta}\right)-\frac{\gamma}{\beta}
		\end{align}
		with the soft shrinkage operator $S_{|\beta|}$
		
		For $\alpha \ne 0, \beta = 0$, we get analogously
		\begin{align}
		(\hat y_1, \hat y_2) = \left( S_{|\alpha|}\left(x_1+\frac{\gamma}{\alpha}\right) -\frac{\gamma}{\alpha}, x_2\right) .
		\end{align}				
		\item
		$\alpha, \beta \ne 0$:
		By substitution, we get the following equivalent minimization problem:
		\begin{align}\label{lemma2}
		(\hat z_1, \hat z_2) &= \argmin_{z \in \R^2} \left\{|z_1 + z_2| + \frac{1}{2\lambda_1} (z_1-\tilde x_1)^2 + \frac{1}{2\lambda_2} (z_2-\tilde x_2)^2 \right\}
		\end{align}
		with
		\begin{align}
		& z_1 \coloneqq \alpha y_1, 	
		\quad	
		& \tilde x_1 &\coloneqq \alpha x_1, 		
		\quad 	
		&\lambda_1 &\coloneqq \alpha^2 > 0, 
		\\   
		& z_2 \coloneqq \beta y_2 + \gamma,
		\quad	
		& \tilde x_2 &\coloneqq \beta x_2+\gamma, 
		\quad 	
		&\lambda_2 &\coloneqq \beta^2 > 0.
		\end{align}
	\end{enumerate}
	The following proposition provides a solution to the minimization problem.				
	\begin{proposition} \label{lemma:shrink}
		Let $(x_1,x_2)^\tT \in \R^2$. 
		Then, the unique minimizer of 
		\begin{align}
		F(z_1,z_2) \coloneqq |z_1 + z_2| + \frac{1}{2\lambda_1} (z_1-\tilde x_1)^2 + \frac{1}{2\lambda_2} (z_2-\tilde x_2)^2
		\end{align}
		is given by
		\begin{align}\label{not:genshr}
		(\hat z_1, \hat z_2) = \begin{cases} 
		\left( \tilde x_1-\lambda_1,  \tilde x_2-\lambda_2 \right) & \mathrm{if } \ \tilde x_1 + \tilde x_2 > \lambda_1 + \lambda_2, \\
		\left( \tilde x_1+\lambda_1,  \tilde x_2+\lambda_2 \right) & \mathrm{if } \ \tilde x_1 + \tilde x_2 < -(\lambda_1 + \lambda_2), \\
		\left( \frac{\tilde x_1 \lambda_2 - \lambda_1 \tilde x_2}{\lambda_1 + \lambda_2}, \frac{\tilde x_2 \lambda_1 - \lambda_2 \tilde x_1}{\lambda_1 + \lambda_2} \right) & \mathrm{otherwise.} \end{cases}
		\end{align}
	\end{proposition}
	\begin{proof}
		Since $F$ is not differentiable at $ z_1 = - z_2$, we distinguish whether $\hat z_1 + \hat z_2 = 0$ or not. 
		\begin{enumerate}
			\item
			$\hat z_1 + \hat z_2 \ne 0$:
			First, we assume that $\hat z_1 + \hat z_2 >0$.
			Since $F$ is differentiable around $\hat z$, we can compute the gradient and set it to zero, i.e.,
			\begin{align}
			0 &= \begin{pmatrix} 1 \\ 1 \end{pmatrix} + \begin{pmatrix} \frac{1}{\lambda_1}(\hat z_1 - \tilde x_1) \\ \frac{1}{\lambda_2}(\hat z_2 - \tilde x_2) \end{pmatrix}.
			\end{align}
			Hence, we obtain
			\begin{align}
			(\hat z_1, \hat z_2) = \left( \tilde x_1-\lambda_1,  \tilde x_2-\lambda_2 \right)
			\end{align}
			and $ \hat z_1 + \hat z_2 > 0$ if and only if $\tilde x_1 + \tilde x_2 > \lambda_1 + \lambda_2$.
			Analogously, one can show 
			$(\hat z_1, \hat z_2) = \left( \tilde x_1+\lambda_1,  \tilde x_2+\lambda_2 \right)$ for $\tilde x_1 + \tilde x_2 < -(\lambda_1 + \lambda_2)$.
			\item
			$ \hat z_1 + \hat z_2 = 0$:
			Then $\hat z_1 = -\hat z_2$ and we obtain 
			\begin{align}
			\hat z_1 &= \argmin_{z_1 \in \R} \left\{ \frac{1}{2\lambda_1} (z_1-\tilde x_1)^2 + \frac{1}{2\lambda_2}(z_1+\tilde x_2)^2 \right\}.
			\end{align}
			Setting the derivative to zero leads to
			\begin{align}
			0 = \frac{1}{\lambda_1} (\hat z_1-\tilde x_1) + \frac{1}{\lambda_2} (\hat z_1+\tilde x_2).
			\end{align}
			Thus, the solution is given by
			\begin{align}
			\hat z_1 = \frac{\tilde x_1 \lambda_2 - \lambda_1 \tilde x_2}{\lambda_1 + \lambda_2} \quad \mathrm{and} \quad
			\hat z_2 = \frac{\tilde x_2 \lambda_1 - \lambda_2 \tilde x_1}{\lambda_1 + \lambda_2}.				
			\end{align}
		\end{enumerate}
	\end{proof}
	\begin{remark}
		The soft shrinkage \eqref{not:shr} is obtained in the special case $\lambda_1 = \lambda_2$ and $\tilde x_1 = \tilde x_2$.
	\end{remark}	
	
	\bibliography{refs_phd_thesis_fitschen_clean_v2}
	\bibliographystyle{abbrv}
\end{document}